\documentclass[12pt,reqno] {article}

\usepackage{mathrsfs,amsfonts,amssymb,amsmath}
\usepackage{enumerate}
\usepackage{graphicx,cite}
\usepackage{romannum}
\usepackage{hyperref}
\usepackage{graphicx}
\newcommand\sbullet[1][.5]{\mathbin{\vcenter{\hbox{\scalebox{#1}{$\bullet$}}}}}
\hypersetup{colorlinks=true, linkcolor=blue, citecolor=red}

\DeclareMathOperator{\diam}{diam}
\DeclareMathOperator{\Leb}{Leb}
\DeclareMathOperator{\Cor}{Cor}
\DeclareMathOperator{\interior}{int}

\newcommand{\Lip}{\mbox{\rm Lip}}

\numberwithin{equation}{section}

\newtheorem{theorem}{Theorem}[section]
\newtheorem{corollary}[theorem]{Corollary}
\newtheorem{proposition}[theorem]{Proposition}
\newtheorem{lemma}[theorem]{Lemma}

\newtheorem{definition}[theorem]{Definition}
\newtheorem{remark}[theorem]{Remark}

\newtheorem{assumption}[theorem]{Assumption}
\newcommand{\qed}{\hfill \mbox{\raggedright \rule{.07in}{.1in}}}

\newenvironment{proof}{\vspace{1ex}\noindent{\bf
Proof}\hspace{0.5em}}{\hfill\qed\vspace{1ex}}

\begin{document}

\title
{Poisson Approximations and Convergence Rates for Hyperbolic Dynamical Systems}

%%%%%%%%%%%%%%%%%%%%%%%%%%%%%%%%

\author{Leonid Bunimovich\thanks{School of Mathematics, Georgia Institute of Technology, Atlanta, USA. \texttt{leonid.bunimovich@math.gatech.edu}}, Yaofeng Su\thanks{School of Mathematics, Georgia Institute of Technology, Atlanta, USA. \texttt{yaofeng.su@math.gatech.edu}}}

\date{\today}

\maketitle

\begin{abstract}
We prove the asymptotic functional Poisson laws in the total variation norm and obtain estimates of the corresponding convergence rates for a large class of hyperbolic dynamical systems. These results generalize the ones obtained before in this area. Applications to intermittent solenoids, Axiom A systems, H\'enon attractors and to billiards, are also considered. 
\end{abstract}

\tableofcontents

%
%%%%%%%%%%%%%%%%%%%%%%%%%%%%%%%%%%%%%%%%%%%%%%%%%%
\section{Introduction}\ \par

The studies of Poisson approximations of the process of recurrences to small subsets in the phase spaces of chaotic dynamical systems, started in \cite{Pitskel}, are developed now into a large active area of the  dynamical systems theory. Another view at this type of problems is a subject of the theory of open dynamical systems \cite{yorke}, where some positive measure subset $A$ of the phase space is named a hole, and hitting and escape the hole processes are studied. The third view at this type of problems concerns statistics of extreme events (``record values") in the theory of random processes \cite{Freitas}. In this paper we present new advances in this area.

In a general set up, one picks a small measure subset $A$ in the phase space $\mathcal{M}$ of hyperbolic (chaotic) ergodic dynamical system and attempts to prove that in the limit, when the measure of $A$ approaches  zero, the corresponding process of recurrences to $A$ converges to the Poisson process. 

This area received an essential boost in L-S.Young papers \cite{Y,Y1}, where a new general framework for analysis of statistical properties of hyperbolic dynamical systems was introduced. This approach employs representation of the phase space of a dynamical system as a tower (later called Young tower, Gibbs-Markov-Young tower, etc), which allow to study dynamics by analysing  recurrences to the base of this tower. Several developments of this method were proposed later, essentially all focused on the dynamical systems with weak hyperbolicity (slow decay of correlations). For such systems the method of inducing was employed, when the base of the tower is chosen as such subset of the phase space where the induced dynamics, generated by the recurrences to the base, is strongly hyperbolic \cite{chernovhighdim, markarian,chernov-dolgopyat}. 

Our approach to the Poisson approximations is slightly different. It employs pulling back a hole $A$ to a nice (strongly hyperbolic) reference set in the phase space, e.g., the base of the Young tower. This pull back method allows to improve various results previously obtained in this area.

The main results (Theorems \ref{thm} and \ref{thm2}) of the paper are dealing with convergence of a random process, generated by the measure preserving dynamics, to the functional Poisson law in the total variation (TV) norm. We also obtain estimates of the corresponding convergence rates in the form

\begin{equation}\label{equaintro}
   d_{TV}(N^{r,z,T}, P) \precsim_{T,z} r^{a} \text{ for almost every } z \in \mathcal{M} , 
\end{equation}
where $P$ is a Poisson point process and $N^{r,z,T}$ is a dynamical point process which counts a number of entrances by an orbit to a metric ball $B_{r}(z)$ with radius $r$ and the center $z$ in the phase space of a dynamical system during the time interval $[0,T]$.  The notation $\precsim_{T,z}$ means that a constant in (\ref{equaintro}) depends only on $z$ and $T$ (see the Definition \ref{dynamicptprocess} for more details).

    These results on convergence to the Poisson distribution are stronger than the ones obtained previously  \cite{ ptrf, nicol,pene,collet,haydndcds,haydncmp}. 
    
    Namely 
    \begin{enumerate}
        
        \item In \cite{ptrf,haydncmp,nicol} the following forms of convergence were obtained
        \[\lim_{r \to 0}\mathbb{P}\{N^{r,z,T}([0,T])=k\}=\mathbb{P}\{P([0,T]=k)\}\]
        
        and/or 
        
        \[(N^{r,z,T}(I_1), \cdots , N^{r,z,T}(I_m))  \xrightarrow[r \to 0]{distribution} (P(I_1), \cdots ,P(I_m)),\]
        where $m, k \in \mathbb{N}$ and intervals $I_1, \cdots I_m \subseteq [0,T]$. Clearly,  (\ref{equaintro}) implies these two forms. 
        \item  In \cite{ptrf,pene,nicol,haydncmp} only convergence to the Poisson law was considered, while the estimations of the convergence rates were not studied because the approaches used there did not allow for such estimates.
        \item In \cite{collet, haydndcds} the convergence rates were obtained in a weaker form. Namely,  $ \forall r \in (0,1)$ there exist positive constants $a$, $b$ and a set $M_r \subseteq \mathcal{M}$ with $\mathbb{P}(M_r)\le r^b$ such that $\forall z \notin M_r$
    \begin{equation}\label{weak}
        \sum_{k}|\mathbb{P}\{N^{r,z,T}([0,T])=k\}-\mathbb{P}\{P([0,T])=k\}|\precsim r^a.
    \end{equation}
     
    \end{enumerate}
    
    Besides, the following generalizations of the previous results are obtained
    \begin{enumerate}
        \item In \cite{collet, haydndcds} a relatively high regularity (at least bounded derivatives) of the dynamics was required, while we just need it to be a local $C^1$-diffeomorphism. Particularly, the derivatives can be unbounded.
        \item Unlike \cite{collet}, we do not assume that the unstable manifolds are one-dimensional.
        \item In \cite{ haydndcds, collet, pene,haydncmp} sufficiently fast decay rates of return times on hyperbolic towers were required. Our proofs of the existence of the Poisson limit laws use only contraction rates $\alpha$ of the (un)stable manifolds. Particularly, a simple, easy to verify, following criterion for the existence of the Poisson limit law is obtained
    \[\alpha>  \frac{2}{\dim \gamma^u}-\frac{1}{\dim_H \mu},\]
    where $\dim \gamma^u$ is the dimension of unstable manifolds and $\dim_H \mu$ is the Hausdorff dimension of the SRB measure on the Gibbs-Markov-Young tower (see the  details in the Theorems \ref{thm} and \ref{thm2}).
    
    This criterion allows to skip verification of the so called corona conditions (see the Definition \ref{defshortcoro} or \cite{pene}), which is usually rather cumbersome even for uniformly hyperbolic dynamical systems. Such verification becomes even more involved in case of non-uniformly hyperbolic systems.
    \end{enumerate}

    Also, some applications of our approach are considered to obtain the asymptotic functional Poisson laws and the corresponding convergence rates  for Axiom A attractors, H\'enon attractors, intermittent solenoids and billiards, which improve various previously known results for these systems.
    
   The structure of the paper is the following one. In section \ref{def} we introduce notations, which are used throughout the paper, give the necessary definitions and formulate the main results. The section \ref{fpl} presents a proof of the functional Poisson law (with the error term) for systems admitting Young towers of general type. The section \ref{provethm1} contains a proof of the Theorem \ref{thm}. A proof of the Theorem \ref{thm2} is in the section \ref{provethm2}.  In the section \ref{app}  applications to Axiom A attractors, intermittent solenoids, billiards and Henon attractors are considered.  

 %%%%%%%%%%%%%%%%%%%%%%%%%%%%%%%%%%%%%%%%%%%%%%%%%%%%%%%%%%%%%%%%%

\section{Definitions and Main Results }\label{def}

We start by introducing some notations
\begin{enumerate}

    \item $C_z$ denotes a constant depending on $z$.
    \item The notation $a\precsim_z b$ means that there is a constant  $C_z>0$, depending on $z$, such that (s.t.) $a \le C_z \cdot b$. 
    \item The relation $a=C_z^{\pm 1} \cdot b$ means that there is a constant $C_z\ge 1$ s.t. $C_z^{-1} \cdot b\le a \le C_z \cdot b$.
    \item The notation $\mathbb{P}$ refers to a probability distribution on the probability space, where a random variable lives, and $\mathbb{E}$ denotes the expectation of a random variable.
\end{enumerate}

\begin{definition}[Dynamical point processes]\label{dynamicptprocess}\ \par
Let $(\mathcal{M},d)$ be a Riemannian manifold (with or without boundaries, connected or non-connected, compact or non-compact), $d$ is the Riemannian metric on $\mathcal{M}$ and $B_{r}(z)$ is a geodesic ball in $\mathcal{M}$ with radius $r$ and center $z \in \mathcal{M}$. We assume that dynamics $f: (\mathcal{M},\mu) \to (\mathcal{M}, \mu)$ is ergodic with respect to (w.r.t.) some invariant probability measure $\mu$.

Let $T>0$. Consider a dynamical point process on $[0,T]$, so that for any $t \in [0,T]$ 

\[N^{r,T, z}_{t}:=\sum_{i=0}^{ \frac{t}{\mu(B_r(z))}} 1_{B_r(z)} \circ f^i.\]

Thus the dynamical point process $N^{r,T, z}$ is a random counting measure on $[0,T]$.

\end{definition}

\begin{definition}[Poisson point processes]\label{poisson}\ \par
For any $T>0$, we say that $P$ is a Poisson point process on $[0,T]$ if
\begin{enumerate}
    \item $P$ is a random counting measure on $[0,T]$.
    \item $P(A)$ is a Poisson-distributed random variable for any Borel set $A \subseteq [0,T]$.
    \item If $A_1, A_2, \cdots, A_n \subseteq [0,T]$ are pairwise disjoint, then  $P(A_1), \cdots, P(A_n)$ are independent.
    \item $\mathbb{E}P(A)=\Leb(A)$ for any Borel set $A\subseteq [0,T]$.
\end{enumerate}

\end{definition}

\begin{definition}[Total variation norms of point processes]\label{totalnorm}\ \par
For any $T>0$ consider the $\sigma$-algebra $\mathcal{C}$ on the space of point processes on $[0,T]$, defined as 

\begin{equation}\label{sigmaalg}
\sigma \{\pi^{-1}_AB: A \subseteq [0,T],
\end{equation}
where $B \subseteq \mathbb{N}$ is a Borel set and
$\pi_A$ is an evaluation map defined on the space of counting measures, so that for any counting measure $N$

\[\pi_AN:=N(A).\]

Now we can define the total variation norm for the Poisson approximation of a dynamical point process as

\[d_{TV}(N^{r,T,z},P):=\sup_{C\in \mathcal{C}}|\mu(N^{r,T,z} \in C)-\mathbb{P}(P \in C)|\]

\end{definition}

\begin{definition}[Convergence rates of the Poisson approximations]\label{convergencerate}\ \par
Suppose that for any $T>0$ there exists a constant $a>0$ s.t. for almost every  $z\in \mathcal{M}$
\[d_{TV}(N^{r,T,z},P) \precsim_{T,z} r^a \to 0.\]

Then $a$ is called a rate of a Poisson approximation.

\end{definition}

We now turn to the definition of the Gibbs-Markov-Young structures \cite{Alves, Y, Y1}:

\begin{definition}[Gibbs-Markov-Young structures]\label{gibbs}\ \par

Introduce at first several notions concerning hyperbolic dynamics $f$ on the Riemannian manifolds  $(\mathcal{M},d)$.

\begin{enumerate}
    \item An embedded disk $\gamma^u$ is called an unstable manifold if for every $x,y \in \gamma^u$ 
    \[\lim_{n \to \infty}d(f^{-n}(x), f^{-n}(y))=0\]
    \item An embedded disk $\gamma^s$ is called a stable manifold if for every $x,y \in \gamma^s$ 
    \[\lim_{n \to \infty}d(f^{n}(x), f^{n}(y))=0\]
    \item $\Gamma^u:=\{\gamma^u\}$ is called a continuous family of $C^1$-unstable manifolds if there is a compact
set $K^s$, a unit disk $D^u$ in some $\mathbb{R}^n$ and a map $\phi^u: K^s \times D^u \to \mathcal{M}$ such that 
\begin{enumerate}
    \item $\gamma^u=\phi^u(\{x\}\times D^u)$ is an unstable manifold,
    \item $\phi^u$ maps $K^s \times D^u$ homeomorphically onto its image,
    \item $x \to \phi^u|_{\{x\} \times D^u}$ defines a continuous map from $K^s$ to $Emb^1(D^u, \mathcal{M})$, where $Emb^1(D^u, \mathcal{M})$ is the space of $C^1$-embeddings of $D^u$ into $\mathcal{M}$.
\end{enumerate}
\end{enumerate}

A continuous family of $C^1$-stable manifolds $\Gamma^s:=\{\gamma^s\}$ is defined similarly.

\item We say that a compact set $ \Lambda \subseteq \mathcal{M}$ has a
hyperbolic product structure if there exist  continuous families of stable manifolds $\Gamma^s:=\{\gamma^s\}$ and of unstable manifolds $\Gamma^u:=\{\gamma^u\}$ such that

\begin{enumerate}

    \item $\Lambda=(\cup \gamma^s) \cap (\cup \gamma^u)$,
    
    \item $\dim \gamma^s+\dim \gamma^u=\dim \mathcal{M}$,
    
    \item each $\gamma^s$ intersects each $\gamma^u$ at
    exactly one point,
    
    \item stable and unstable manifolds are transversal, and the angles between them are uniformly bounded away from 0.
\end{enumerate}

A subset $\Lambda_1 \subseteq \Lambda$ is called a s-subset if $\Lambda_1$ has a hyperbolic product structure and, moreover, the corresponding families of stable and unstable manifolds  $\Gamma^s_1$ 
and $\Gamma^u_1$ can be chosen so that  $\Gamma^s_1 \subseteq \Gamma^s$ and $\Gamma^u_1 = \Gamma^u$. 

Analogously, a subset $\Lambda_2 \subseteq \Lambda$
is called an $u$-subset if $\Lambda_2$ has a hyperbolic product structure and the families $\Gamma^s_2$
and $\Gamma^u_2$ can be chosen so that  $\Gamma^u_2 \subseteq \Gamma^u$ and $\Gamma^s_2 = \Gamma^s$.

\item For $x \in \Lambda$, denote by $\gamma^u(x)$ (resp. $\gamma^s(x)$) the element of $\Gamma^u$ (resp. $\Gamma^s$)
which contains $x$. Also, for each $n \ge 1$, 
denote by $(f^n)^u$ the restriction of the map $f^n$ to $\gamma^u$-disks, and by $\det D(f^n)^u$ denote the Jacobian of $(f^n)^u$.

We say that the set $\Lambda$ with  hyperbolic product structure has also a \textbf{Gibbs-Markov-Young structure} if  the following  properties are satisfied 

\begin{enumerate}
    \item  Lebesgue detectability:  there exists $\gamma \in \Gamma^u$ such that $\Leb_{\gamma}(\Lambda \bigcap \gamma) > 0$.
    \item Markovian property: there exist pairwise disjoint $s$-subsets $\Lambda_1,\Lambda_2, \cdots \subseteq \Lambda$ such that
    \begin{enumerate}
        \item $\Leb_{\gamma}(\Lambda \setminus (\bigcup_{i \ge 1}\Lambda_i))=0$ on each $\gamma \in \Gamma^u$,
        \item for each $i \in \mathbb{N}$ there exists such $R_i \in \mathbb{N}$ that $f^{R_i} (\Lambda_i)$ is an $u$-subset, and for all $x \in \Lambda_i$
        % it should rather be intersection of LAMBDA with f^(R_i) with LAMBDA_i%
        \[f^{R_i}(\gamma^s(x)) \subseteq \gamma^s(f^{R_i}(x))\]
        and
    \[f^{R_i}(\gamma^u(x)) \supseteq \gamma^u(f^{R_i}(x)).\]
    \end{enumerate}
   
Define now a return time function $R : \Lambda \to \mathbb{N}$ and a return function $f^R: \Lambda \to \Lambda$, so that for each $i \in \mathbb{N}$ 
\[R|_{\Lambda_i}=R_i
\text{ and }
f^R|_{\Lambda_i}=f^{R_i}|_{\Lambda_i}\]

The separation time $s(x, y)$ for $x, y \in \Lambda$ is defined as
\[s(x,y):=\min \{n \ge 0: (f^R)^n(x) \text{ and } (f^R)^n(y) \text{belong to the different sets } \Lambda_i\}.\]

We also assume that there are constants $C > 1, \alpha > 0$ and $0 < \beta < 1$, which depend only on $f$ and $\Lambda$, such that the following conditions hold

\item Polynomial contraction on stable leaves: 
\[\forall \gamma^s \in \Gamma^s, \forall x,y \in \gamma^s, \forall n \in \mathbb{N}, d(f^n(x), f^n(y)) \le \frac{C}{n^{\alpha}}.\]

\item Backward polynomial contraction on unstable leaves: 
\[\forall \gamma^u \in \Gamma^u, \forall x,y \in \gamma^u, \forall n \in \mathbb{N}, d(f^{-n}(x), f^{-n}(y)) \le \frac{C}{n^{\alpha}}.\]
\item Bounded distortion: $\forall \gamma \in \Gamma^u \text{ and } x, y \in \gamma \bigcap \Lambda_i$ for some $\Lambda_i$, 
\[\log \frac{\det D(f^R)^u(x)}{\det D(f^R)^u(y)} \le C \cdot \beta^{s(f^R(x), f^R(y))}.\]
\item Regularity of the stable foliations: for each $\gamma, \gamma'\in \Gamma^u$ denote 
\[\Theta_{\gamma, \gamma'}: \gamma' \bigcap \Lambda \to \gamma \bigcap \Lambda: x \to \gamma^s(x)\bigcap \gamma.\]

Then the following properties hold
\begin{enumerate}
    \item $\Theta_{\gamma, \gamma'}$ is absolutely continuous and $\forall x \in \gamma \bigcap \Lambda$
    \[\frac{d(\Theta_{\gamma, \gamma'})_{*}\Leb_{\gamma'}}{d\Leb_{\gamma}}(x)=\prod_{n \ge 0} \frac{\det Df^u(f^n(x))}{\det Df^u(f^n(\Theta_{\gamma, \gamma'}^{-1}x))},\]
    \[ \frac{d(\Theta_{\gamma, \gamma'})_{*}\Leb_{\gamma'}}{d\Leb_{\gamma}}(x)= C^{\pm 1},\]
    \item for any $x,y \in \gamma \bigcap \Lambda$ 
    \[\log \frac{\frac{d(\Theta_{\gamma, \gamma'})_{*}\Leb_{\gamma'}}{d\Leb_{\gamma}}(x)}{\frac{d(\Theta_{\gamma, \gamma'})_{*}\Leb_{\gamma'}}{d\Leb_{\gamma}}(y)} \le C \cdot \beta^{s(x,y)}.\]
    
\end{enumerate}
\item Aperiodicity: $\gcd(R_i, i \ge 1)= 1$. 
   \item A decay rate of the return times $R$: there exist $\xi>1$ and $\gamma\in \Gamma^u$ such that 
   \[\Leb_{\gamma}(R>n) \le \frac{C}{n^{\xi}}.\]
\end{enumerate}

\textbf{SRB measures}: 
Let the dynamics $f: (\mathcal{M}, \mu) \to (\mathcal{M}, \mu)$ has a Gibbs-Markov-Young structure. It was proved in \cite{Alves, Y, Y1} that there exists  an ergodic probability measure $\mu$  such that for any unstable manifold $\gamma^u$ (including $\Gamma^u$) $\mu_{\gamma^u} \ll \Leb_{\gamma^u}$, where $\mu_{\gamma^u}$ is the conditional measure of $\mu$ on the unstable manifold $\gamma^u$. Such $\mu$ is called a Sinai-Ruelle-Bowen measure (SRB measure).

\end{definition}

\begin{assumption}[Geometric regularities]\label{geoassumption}\ \par
Assume that $f: \mathcal{M}\to \mathcal{M}$ has the Gibbs-Markov-Young structure, as described in the Definition \ref{gibbs}, and
\begin{enumerate}
    
   \item $f$ is bijective and a local $C^1$-diffeomorphism on $\bigcup_{i \ge 1} \bigcup_{0 \le j < R_i} f^j(\Lambda_i)$.
   \item the following limit exists

  \[ \dim_H\mu:=\lim_{r \to 0} \frac{\log \mu(B_r(z))}{\log r}\]
   
 for almost every  $z\in \mathcal{M}$. Then  $dim_H\mu$ is  called a Hausdorff dimension of the measure $\mu$.
 
 \item $\alpha \cdot \dim_H\mu>1$, where $\alpha$ is the contraction rate of the (un)stable manifolds in the Definition \ref{gibbs}.
\end{enumerate}
\end{assumption}

\begin{assumption}[The first returns \& interior assumptions on $\Lambda$]\label{assumption}\ \par

    Let $f: \mathcal{M}\to \mathcal{M}$ has the Gibbs-Markov-Young structure, and there are constants $C>1$ and $\beta\in (0,1)$ (the same as in the Definition \ref{gibbs}) such that
    
\begin{enumerate}   
    \item $R: \Lambda \to \mathbb{N}$ is the first return time and $f^R: \Lambda \to \Lambda$ is the first return map for $\Lambda$.
    
    This implies  that $f^R$ is actually bijective (see the Lemma \ref{inducemapbi} below).
    \item  $\forall \gamma \in \Gamma^s, \gamma_1 \in \Gamma^u, \forall x, y \in \gamma \bigcap \Lambda, \forall x_1, y_1 \in \gamma_1 \bigcap \Lambda $,
    \[d((f^R)^n(x), (f^R)^n(y)) \le C \cdot \beta^n,\]
    and
    \[d((f^R)^{-n}(x_1), (f^R)^{-n}(y_1)) \le C \cdot \beta^n \cdot d(x_1,y_1).\]

    \item $\mu\{\interior{(\Lambda)}\}>0$ and $\mu(\partial \Lambda)=0$, where 
    \[\interior{\Lambda}:=\{x \in \Lambda: \exists r_x>0 \text{ s.t. } \mu(B_{r_x}(x)\setminus \Lambda)=0 \},\]
    \[\partial \Lambda:=\Lambda \setminus \interior{\Lambda}.\]
    In other words, $x \in \interior{\Lambda}$ if and only if $x \in \Lambda$ and there is a small ball $B_{r_x}(x)$ s.t. $B_{r_x}(x) \subseteq \Lambda$ $\mu$-almost surely. 
\end{enumerate}

\end{assumption}

Now we can  formulate the first main result of the paper.

\begin{theorem}[Convergence rates for functional Poisson laws \Romannum{1}]\label{thm}\ \par
Assume that the dynamics $f: (\mathcal{M}, \mu) \to (\mathcal{M}, \mu)$ has a Gibbs-Markov-Young structure (see the Definition \ref{gibbs}) and satisfies the Assumptions \ref{geoassumption} and \ref{assumption}. Then for any $T>0$ the following results hold 

\begin{enumerate}
    \item $\dim_H\mu \ge \dim {\gamma}^u$.
    \item If $\alpha>  \frac{2}{\dim \gamma^u}-\frac{1}{\dim_H \mu}$, then for almost every $z \in \mathcal{M}$
    \[d_{TV}(N^{r,z,T}, P) \precsim_{T,\xi,z} r^{a},\]
    where the constant $a>0$ depends on $\xi>1$, $\dim_H \mu, \dim \gamma^u $ and $\alpha$, but it does not depend upon $z\in \mathcal{M}$. The expression for $a$ can be found in the Lemma \ref{lastproofs}. 
    \item If $\mu \ll \Leb_{\mathcal{M}}$ and $\frac{d\mu}{d\Leb_{\mathcal{M}}} \in L^{\infty}_{loc}(\mathcal{M})$, then \[d_{TV}(N^{r,z,T}, P) \precsim_{T,\xi,z} r^{a},\]
    where $a>0$ depends on $\xi>1$, and $\dim_H \mu, \dim \gamma^u, \alpha$, but does not depend upon $z\in \mathcal{M}$. The expression for $a$ can be found in the Lemma \ref{lastproofs}. 
    
\end{enumerate}

\end{theorem}

\begin{definition}[Induced measurable partitions]\label{mpartition}\ \par
We say a probability measure $\mu$ for the dynamics $f: \mathcal{M} \to \mathcal{M}$ has an induced measurable partition if there are constants $\beta \in (0,1), C>1$ (the same as in the Definition \ref{gibbs}) and $b>0$ such that
\begin{enumerate}
    \item There exists an area $U \subseteq \mathcal{M}$ with $\mu \{\interior{(U)}\}>0$, $\mu(\partial U)=0$.
    \item The area $U$ has a measurable  partition $\Theta:=\{\gamma^u(x)\}_{x \in U}$ (which could be different from $\Gamma^u$), such that the elements of $\Theta$ are  disjoint  connected unstable manifolds, so that $\mu$-almost surely $U=\bigsqcup_{x \in U} \gamma^u(x)$ and
    \[\mu_{U}(\cdot)=\int_U \mu_{\gamma^u(x)}(\cdot) d\mu_U(x),\]
    where $\mu_U:=\frac{\mu|_{U}}{\mu(U)}$ and $\mu_{\gamma^u(x)}$ is the conditional probability induced by $\mu$ on $\gamma^u(x)\in \Theta$.
    \item Each $\gamma^u\in \Theta$ is (at least $C^1$) smooth.
    \item All $\gamma^u\in \Theta$ have uniformly bounded sectional curvatures and the same dimensions.
    \item For any $\epsilon \in (0,1)$\[\mu_U \{x \in U: |\gamma^u(x)|< \epsilon\}\le C \cdot \epsilon^b,\] where $|\gamma^u(x)|$ is the radius of the largest inscribed geodesic ball  in $\gamma^u(x) \in \Theta$, and a 
    geodesic ball is defined with respect to the distance   $d_{\gamma^u(x)}$ on $\gamma^u(x)$, induced by the Riemannian metric. This property implies that almost every $\gamma^u(x)\in \Theta$ is non-degenerated, i.e. $|\gamma^u(x)|>0$ for almost every $x\in U$.
    \item For almost every point $x\in U$ we have $\mu_{\gamma^u(x)} \ll \Leb_{\gamma^u(x)}$, $\mu_{\gamma^u(x)}(\gamma^u(x))>0$, and  for any $y, z \in\gamma^u(x)$
    
    \[\frac{d\mu_{\gamma^u(x)}}{d\Leb_{\gamma^u(x)}}(y)=C^{\pm 1} \cdot \frac{d\mu_{\gamma^u(x)}}{d\Leb_{\gamma^u(x)}}(z).\]
    \item Denote by $\overline{R}$ the first return time to $U$ for $f$. Then the first return map $f^{\overline{R}}: U \to U$ has an exponential u-contraction, i.e.  $\forall \gamma^u \in \Theta, \forall x,y \in \gamma^u, n \ge 1$
    \[d((f^{\overline{R}})^{-n}(x),(f^{\overline{R}})^{-n}(y)) \le C \cdot \beta ^n \cdot d(x,y),\]
    
    and an exponential decay of correlation, i.e. for any $h \in \Lip(U)$
    \[\int h \circ (f^{\overline{R}})^n \cdot h d\mu_{U}-(\int h d\mu_U)^2= O(\beta^n) \cdot ||h||^2_{\Lip}. \]
    
\end{enumerate}
\end{definition}

Now we are able to formulate the second main result of the paper. 

\begin{theorem}[Convergence rates for the functional Poisson laws  \Romannum{2}]\label{thm2}\ \par
Assume that the dynamics $f: (\mathcal{M}, \mu) \to (\mathcal{M}, \mu)$ has the Gibbs-Markov-Young structure (see the Definition \ref{gibbs}), satisfies the Assumption \ref{geoassumption} and $\mu$ has an induced measurable partition (see the Definition \ref{mpartition}). Then for any $T>0$, the following results hold.

\begin{enumerate}
    \item $\dim_H \mu \ge \frac{ b}{b+\dim \gamma^u} \cdot \dim \gamma^u$.
    \item If $\alpha>  \frac{2}{\dim \gamma^u}\cdot \frac{b+\dim \gamma^u}{b}-\frac{1}{\dim_H \mu} $, then for almost every (a.e.) $z \in \mathcal{M}$,
    \[d_{TV}(N^{r,z,T}, P) \precsim_{T, \xi, z} r^{a},\]
    where a constant $a>0$ depends on $\xi>1$, $\dim_H \mu, \dim \gamma^u, b$ and $\alpha$, but it does not depend on $z\in \mathcal{M}$. The expression for $a$ can be found in the Lemma \ref{lastlastproof}. 
    \item If $\mu \ll \Leb_{\mathcal{M}}$ with $\frac{d\mu}{d\Leb_{\mathcal{M}}} \in L^{\infty}_{loc}(\mathcal{M})$, then \[d_{TV}(N^{r,z,T}, P) \precsim_{T, \xi,z} r^{a},\]
    where $a>0$ depends on $\xi>1, \dim_H \mu, \dim \gamma^u, \alpha$ and $b$, but it does not depend on $z\in \mathcal{M}$. The expression for $a$ can be found in the Lemma \ref{lastlastproof}. 
    
\end{enumerate}

\end{theorem}

\begin{remark}\label{remark}\ \par

\begin{enumerate}

\item The Assumption \ref{assumption} that $R$ is the first return time and $f^R$ is the first return map of $\Lambda$ is natural if the system has a Markov partition. Otherwise, we assume that the system has an area $U$ with an induced measurable partition (see the Definition \ref{mpartition}).

It will be shown in what follows that the Theorems \ref{thm} and \ref{thm2} work efficiently for various systems in applications (see the Section \ref{app}). Clearly, a key issue is a choice of the reference sets $\Lambda$ and $U$.

\item Under similar conditions to the Definition \ref{mpartition}, \cite{vaientizhang} proved that $\mu_{U}(\overline{R}>n)$ characterizes the optimal bound for the decay rates of correlations 
for sufficiently good observables supported on $U$ (see the Theorem 1.3 in \cite{vaientizhang}); \cite{melboune} uses operator renewal theory as a method to prove also sharp results on polynomial decay of correlations (see Theorem 3.1 in \cite{melboune}).  For many purposes the aperiodicity in the Definition \ref{gibbs} is irrelevant provided the dynamic $f: (\mathcal{M}, \mu) \to (\mathcal{M}, \mu) $ is mixing (see the Remark 2.2 in \cite{melboune}). Indeed all dynamical systems, which we consider in applications (section \ref{app}), do have a Markov partition. Also an ergodic completely hyperbolic (all Lyapunov exponents do not vanish) dynamical system is mixing. Therefore Young towers are mixing. So, to simplify the argument of our proof, we only assume the aperiodicity in the Gibbs-Markov-Young structures.

\item When dealing with applications, (see the section \ref{app}), it is always assumed that $\mu$ is a hyperbolic measure (i.e. the Lyapunov exponents do not vanish almost everywhere, see \cite{pesin}). Also, in applications most often there is an explicit natural invariant measure (sometimes called a physical measure). Therefore the Assumption \ref{geoassumption} that  $\dim_H\mu:=\lim_{r \to 0} \frac{\log \mu(B_r(z))}{\log r}$ is relevant to such approach. (However, another dimension conditions, like e.g., \cite{haydncmp}, could be used as well).

\item If an SRB measure $\mu$ is explicitly known, then the Poisson approximations are usually well understood  \cite{demers, haydncmp, nomixing, dolgopyat, nicol}. However, if it is not the case, then often essential difficulties arise, e.g. for  intermittent solenoid attractors, Axiom A attractors, etc (see \cite{pene}). The Theorem \ref{thm} provides an useful, easy to verify, criterion. Indeed, if $\alpha>\frac{2}{\dim \gamma^u}$, then there is no need to know $\dim_H \mu$. Moreover, estimations of the corresponding convergence rates can be obtained as well.

\item According to the Theorems \ref{thm} and \ref{thm2}, it is only required that $\xi>1$. In fact, it is a minimal requirement for the existence of the SRB measures (see \cite{Alves}).

Observe that for our approach only the contraction rate $O(\frac{1}{n^{\alpha}})$  along (un)stable manifolds matters, which is different from the ones employed in \cite{pene, haydncmp, collet}.
\item If $f$ has a sufficiently good regularity, then
$\dim_H \mu \ge \dim \gamma^u$ \cite{pesin,ledra1,ledra2}. Our only assumption is that $f$ is a local $C^1$-diffeomorphism. Observe that we do not even assume that $\mathcal{M}$  is a compact manifold (see the Definition \ref{gibbs} and the Assumption \ref{geoassumption}). Therefore the Theorem \ref{thm2} does not provide a good lower bound for $\dim_H \mu$. It is worthwhile to mention also that for all applications considered below  (see the section \ref{app}) the relation $\dim_H\mu \ge \dim \gamma^u$ always holds.
\end{enumerate}

\end{remark}

\begin{corollary}[The first hitting and survival probabilities]\label{cor}\ \par
Under the same conditions as in the Theorem \ref{thm} or \ref{thm2} consider the first hitting moment of time $\tau_{B_r(z)}(x):=\inf\{n \ge 0: f^n(x) \in B_r(z)\}$. Then for almost every $z \in \mathcal{M}$, any $T>0$ and any $t\le T$ the following relation holds for the first hitting probability 
\begin{equation}\label{escape1}
    \mu(\tau_{B_r(z)}>\frac{t}{\mu(B_r(z))})-e^{-t}=O_{T,\xi, z}(r^a).
\end{equation}
Particularly, the survival probability at time $T$ can be approximated as
\[\mu(\tau_{B_r(z)}>T)=e^{-T\cdot \mu(B_r(z))}+\min\{O_{T,\xi,z}(r^a),2\}.\]
Moreover, the following limiting relations hold 

\begin{equation}\label{escape}
    \lim_{T \to \infty}\lim_{r \to 0} \frac{\log \mu(\tau_{B_r(z)}>T )}{-T\cdot \mu(B_r(z))}=1, 
\end{equation}

and for any $ T>0$

\begin{equation}\label{escape2}
    \lim_{r\to 0} \frac{\log \mu(\tau_{B_r(z)}>\frac{T}{\mu(B_r(z))})}{-T}=1.
\end{equation}

\end{corollary}
\begin{proof}
Clearly $\mu(\tau_{B_r(z)}>\frac{t}{\mu(B_r(z))})=\mu(N^{r,T,z}[0,\frac{t}{\mu(B_r(z))}]=0)$. Apply now a relevant one of the Theorems \ref{thm} and \ref{thm2}. Then $O_{T,\xi,z}(r^a)$ is the error term with the convergence rate $a$. For the survival probability at time $T$ take $t=T \cdot \mu(B_r(z))$. The relation  (\ref{escape1}) implies (\ref{escape2}). According to the Assumption \ref{geoassumption}, $f$ is a local diffeomorphism almost everywhere. Besides, the set of all periodic points has measure zero. Hence $\mu(\tau_{B_r(z)}>T)=1-\mu(\bigcup_{i\le T}f^{-i}B_r(z) )=1-(T+1) \cdot \mu(B_r(z))$, if $r$ is small enough. Therefore (\ref{escape}) holds.
\end{proof}

\section{Functional Poisson Limit Laws}\label{fpl}

This section deals with the functional Poisson limit laws and with convergence rates of $d_{TV}(N^{r,z,T}, P)$ for the dynamics $f$ described in the Definition \ref{gibbs} and satisfying the  Assumption \ref{geoassumption} only. For any $n \ge 0, I \subseteq[0,n]$, let

\[X_i:=1_{B_r(z)} \circ f^i, \text{ } X_I:=\sum_{i \in I} 1_{B_r(z)} \circ f^i.\]

Denote by $\{\hat{X}_i\}_{i\ge 0}$ the i.i.d. random variables, such that for each $i \ge 0$
\[X_i\overset{d}{=}\hat{X}_i.\]

Let  
\[\hat{X_I}:=\sum_{i \in I}\hat{X}_i.\]

Observe that generally $\hat{X_I}$ and ${X_I}$ are not identically distributed. 

\begin{lemma}\label{poissonappro}
For any disjoint sets $I_1, I_2, \cdots, I_m \subseteq [0,n]$ and any integer $p\in (0,n)$,

\[d_{TV}((X_{I_1}, \cdots, X_{I_m}),(\hat{X}_{I_1}, \cdots, \hat{X}_{I_m})) \]
  \[\le 2\cdot  \sum_{0 \le l \le n-p}\sup_{h\in[0,1]} | \mathbb{E}[1_{X_0=1} \cdot  h(X_p,\cdots, X_{n-l})]-\mathbb{E}1_{X_0=1} \cdot \mathbb{E}[h(X_p,\cdots, X_{n-l})]|\]
\[+4(n-p)\cdot \mathbb{E} 1_{X_0=1} \cdot 1_{\sum_{1\le j \le p-1}X_j\ge 1}+4p \cdot (n-p) \cdot \mu(B_{r}(z))^2+4 p \cdot \mu(B_r(z)),\]
where $h$ is a measurable function with values in $[0,1]$ . 

Observe that we obtain a slightly better error bound here, compared to the Theorem 2.1 in \cite{collet}.
\end{lemma}

\begin{proof}
 By the definition of the total variation norm
\[d_{TV}((X_{I_1}, \cdots, X_{I_m}),(\hat{X}_{I_1}, \cdots, \hat{X}_{I_m})) =\sup_{h\in[0,1]} |\mathbb{E}h(X_{I_1}, \cdots, X_{I_m})-h(\hat{X}_{I_1}, \cdots, \hat{X}_{I_m})|\]
\[\le \sup_{h\in[0,1]} |\mathbb{E}h(X_{1},\cdots, X_n)- h(\hat{X}_{1}, \cdots, \hat{X}_{n})|= d_{TV}((X_{1}, \cdots, X_{n}),(\hat{X}_{1}, \cdots, \hat{X}_{n})).\]

Hence, it suffices to estimate
\[d_{TV}((X_{1}, \cdots, X_{n}),(\hat{X}_{1}, \cdots, \hat{X}_{n}))=\sup_{h\in[0,1]} |\mathbb{E}h(X_{1},\cdots, X_{n})- h(\hat{X}_{1}, \cdots, \hat{X}_{n})| \]
\[= \sup_{h\in[0,1]}|\sum_{0 \le l \le n} \mathbb{E} h(\hat{X}_1, \cdots, \hat{X}_{l-1},X_l, \cdots, X_n)- \mathbb{E} h(\hat{X}_1, \cdots, \hat{X}_{l-1},\hat{X}_l, \cdots, X_n)|\]
\[\le \sup_{h\in[0,1]}|\sum_{0 \le l \le n}  \mathbb{E} h_l(X_l,,X_{l+1}, \cdots,X_n)-  \mathbb{E}h_l(\hat{X}_l, X_{l+1}, \cdots, X_n)|,\]
here $h_l(\sbullet[.75]):=h(\hat{X}_1, \cdots,\hat{X}_{l-1}, \sbullet[.75])$. Since $\hat{X}_1, \cdots, \hat{X}_{l-1}$ are independent of other random variables, without loss of generality, $h_l$ can be regarded as a function which does not depend on $\hat{X}_1, \cdots, \hat{X}_{l-1}$. Note that $X_l \overset{d}{=} \hat{X}_l$ are $\{0,1\}$-valued random variables. Thus
\[\mathbb{E} h_l(X_l,,X_{l+1}, \cdots,X_n)-  \mathbb{E}h_l(\hat{X}_l, X_{l+1}, \cdots, X_n)|\]
\[=|\mathbb{E}[ 1_{X_l=0} \cdot h_l(0,X_{l+1}, \cdots,X_n)+\mathbb{E}[ 1_{X_l=1} \cdot h_l(1, X_{l+1}, \cdots,X_n)]\]
\[-\mathbb{E}1_{\hat{X}_l=0} \cdot \mathbb{E}h_l(0,X_{l+1},\cdots, X_n)- \mathbb{E}1_{\hat{X}_l=1} \cdot \mathbb{E}h_l(1,X_{l+1},\cdots, X_n)|\]
\[=|\mathbb{E}[ 1_{X_l=1} \cdot h_l(0,X_{l+1}, \cdots,X_n)+\mathbb{E}[ 1_{X_l=1} \cdot h_l(1, X_{l+1}, \cdots,X_n)]\]
\[-\mathbb{E}1_{\hat{X}_l=1} \cdot \mathbb{E}h_l(0,X_{l+1},\cdots, X_n)- \mathbb{E}1_{\hat{X}_l=1} \cdot \mathbb{E}h_l(1,X_{l+1},\cdots, X_n)|\]
\[\le 2 \cdot \sup_{h\in[0,1]} |\mathbb{E}[ 1_{X_l=1} \cdot h(X_{l+1}, \cdots, X_n)]-\mathbb{E}1_{X_l=1} \cdot \mathbb{E}h(X_{l+1}, \cdots, X_n)|.\]

Therefore,
\[d_{TV}((X_{1}, \cdots, X_{n}),(\hat{X}_{1}, \cdots, \hat{X}_{n}))\]
\begin{equation}\label{sum}
 \le 2 \sum_{0 \le l \le n} \sup_{h\in[0,1]} |\mathbb{E}[ 1_{X_l=1} \cdot h(X_{l+1}, \cdots, X_n)]-\mathbb{E}1_{X_l=1} \cdot \mathbb{E}h(X_{l+1}, \cdots, X_n)|.   
\end{equation}

We will first estimate the terms with $l\le n-p$ in the sum (\ref{sum}).
\[\mathbb{E}[ 1_{X_l=1} \cdot h(X_{l+1},\cdots, X_n)]-\mathbb{E}1_{X_l=1} \cdot \mathbb{E}h(X_{l+1}, \cdots, X_n)\]
\[\le \mathbb{E}[ 1_{X_l=1} \cdot h(X_{l+1},\cdots, X_n)]-\mathbb{E}[ 1_{X_l=1} \cdot h(0,\cdots, 0,X_{l+p},\cdots,X_n)]\]
\[+\mathbb{E}[ 1_{X_l=1} \cdot h(0,\cdots, 0,X_{l+p},\cdots,X_n)]-\mathbb{E}1_{X_l=1} \cdot \mathbb{E}h(X_{l+1}, \cdots, X_n)\]
\[+\mathbb{E}1_{X_l=1} \cdot \mathbb{E}h(0,\cdots, 0,X_{l+p},\cdots,X_n)-\mathbb{E}1_{X_l=1} \cdot \mathbb{E}h(0,\cdots, 0,X_{l+p},\cdots,X_n)\]
\[\le \mathbb{E}\big\{ 1_{X_l=1} \cdot [h(X_{l+1},\cdots, X_n)-h(0,\cdots, 0,X_{l+p},\cdots,X_n)]\big\}\]
\[+\mathbb{E}1_{X_l=1} \cdot \mathbb{E}[h(0,\cdots, 0,X_{l+p},\cdots,X_n)-h(X_{l+1}, \cdots, X_n)]\]
\[+\mathbb{E}[ 1_{X_l=1} \cdot h(0,\cdots, 0,X_{l+p},\cdots,X_n)]-\mathbb{E}1_{X_l=1} \cdot \mathbb{E}h(0,\cdots, 0,X_{l+p},\cdots,X_n)\]

Observe that 
\[ |h(X_{l+1}, \cdots, X_n)- h(0, \cdots, 0,X_{l+p}, \cdots, X_n)| \le 2 \cdot 1_{\sum_{l+1\le j \le l+p-1}X_j\ge 1}.\]
Now, because of the stationarity of $(X_i)_{i\ge 0}$, we can continue the estimates as 
\[\le | \mathbb{E}[1_{X_l=1} \cdot  h(0, \cdots, 0,X_{l+p}, \cdots, X_n)]-\mathbb{E}1_{X_l=1} \cdot \mathbb{E}h(0, \cdots, 0,X_{l+p}, \cdots, X_n)|\]
\[+2|\mathbb{E} 1_{X_l=1} \cdot 1_{\sum_{l+1\le j \le l+p-1}X_j\ge 1}|+2\mathbb{E}1_{X_l=1} \cdot \mathbb{E}1_{\sum_{l+1\le j \le l+p-1}X_j\ge 1}\]
\[\le | \mathbb{E}[1_{X_l=1} \cdot  h(0, \cdots, 0,X_{l+p}, \cdots, X_n)]-\mathbb{E}1_{X_l=1} \cdot \mathbb{E}h(0, \cdots, 0,X_{l+p}, \cdots, X_n)|\]
\[+2|\mathbb{E} 1_{X_0=1} \cdot 1_{\sum_{1\le j \le p-1}X_j\ge 1}|+2\mathbb{E}1_{X_0=1} \cdot \mathbb{E}1_{\sum_{1\le j \le p-1}X_j\ge 1}.\]

Note that $1_{\sum_{1\le j \le p-1}X_j\ge 1}=1_{\cup_{1 \le j \le p-1} f^{-j} B_{r}(z)}$.  Hence, we can continue the sequence of the inequalities above as 
\[\le | \mathbb{E}[1_{X_l=1} \cdot  h(0, \cdots, 0,X_{l+p}, \cdots, X_n)]-\mathbb{E}1_{X_l=1} \cdot \mathbb{E}h(0, \cdots, 0,X_{l+p}, \cdots, X_n)|\]
\[+2|\mathbb{E} 1_{X_0=1} \cdot 1_{\sum_{1\le j \le p-1}X_j\ge 1}|+2 (p-1)\mu(B_{r}(z))^2.\]

Therefore for the terms with $l\le n-p$ in the sum (\ref{sum}) we have 

\[\mathbb{E}[ 1_{X_l=1} \cdot h(X_{l+1}, \cdots, X_n)]-\mathbb{E}1_{X_l=1} \cdot \mathbb{E}h(X_{l+1}, \cdots, X_n)\]
\[\le | \mathbb{E}[1_{X_l=1} \cdot  h(0, \cdots, 0,X_{l+p}, \cdots, X_n)]-\mathbb{E}1_{X_l=1} \cdot \mathbb{E}h(0, \cdots, 0,X_{l+p}, \cdots, X_n)|\]
\[+2\mathbb{E} 1_{X_0=1} \cdot 1_{\sum_{1\le j \le p-1}X_j\ge 1}+2p \cdot \mu(B_{r}(z))^2.\]

Consider now the terms with $l> n-p$ in the sum (\ref{sum}). Since $||h||_{\infty} \le 1$, then
 \[\mathbb{E}[ 1_{X_l=1} \cdot h(X_{l+1},\cdots, X_n)]-\mathbb{E}1_{X_l=1} \cdot \mathbb{E}h(X_{l+1},\cdots, X_n) \le 2 \mu(B_r(z)).\]
 
 Therefore 
 \[(\ref{sum})=2 \cdot \sum_{0 \le l \le n} \sup_{h\in[0,1]}|\mathbb{E}[ 1_{X_l=1} \cdot h(X_{l+1},\cdots, X_n)]-\mathbb{E}1_{X_l=1} \cdot \mathbb{E}h(X_{l+1},\cdots,X_n)|\]
 \[\le 2\cdot \sum_{0 \le l \le n-p} \sup_{h\in[0,1]} |\mathbb{E}[1_{X_l=1} \cdot  h(X_{l+p},\cdots, X_n)]-\mathbb{E}1_{X_l=1} \cdot \mathbb{E}[h(X_{l+p},\cdots, X_n)]|\]
\[+4(n-p)\cdot \mathbb{E} 1_{X_0=1} \cdot 1_{\sum_{1\le j \le p-1}X_j\ge 1}+4p \cdot (n-p) \cdot \mu(B_{r}(z))^2+4 p \cdot \mu(B_r(z)).\]

Again, by making use of the stationarity of $(X_i)_{i \ge 0}$, the last expression above can be estimated as
\[\le 2\cdot  \sum_{0 \le l \le n-p}\sup_{h\in[0,1]} | \mathbb{E}[1_{X_0=1} \cdot  h(X_p,\cdots, X_{n-l})]-\mathbb{E}1_{X_0=1} \cdot \mathbb{E}[h(X_p,\cdots, X_{n-l})]|\]
\[+4(n-p)\cdot \mathbb{E} 1_{X_0=1} \cdot 1_{\sum_{1\le j \le p-1}X_j\ge 1}+4p \cdot (n-p) \cdot \mu(B_{r}(z))^2+4 p \cdot \mu(B_r(z)).\]
\end{proof}

For further estimates we will need the following lemma.

\begin{lemma}[Hyperbolic towers, see \cite{pene}, page 2609]\label{towerresult}\ \par

Define a tower $\Delta$ and a map $F : \Delta \to \Delta$ as 

\[\Delta:=\{(x,l) \in \Lambda \times \mathbb{N}: 0 \le l < R(x) \},\]

\[F(x,l):=\begin{cases}
 (x,l+1),      &l < R(x)-1\\
(f^R(x),0),  & l=R(x)-1\\
\end{cases}.\]

The equivalence relation $\sim$ on $\Lambda$ is then
\[x \sim y \text{ if and only if } x,y \in \gamma^s \text{ for some } \gamma^s \in \Gamma^s.\]

Now we can define a quotient tower  $\widetilde{\Delta}:=\Delta/\sim$, a quotient Gibbs-Markov-Young product structure $\widetilde{\Lambda}:=\Lambda/\sim$, quotient maps $\widetilde{F} : \widetilde{\Delta} \to \widetilde{\Delta}, \widetilde{f^R}: \widetilde{\Lambda} \to \widetilde{\Lambda}$, and canonical projections $\widetilde{\pi}_{\Delta}:\Delta \to \widetilde{\Delta}$ and $\widetilde{\pi}_{\Lambda}:\Lambda \to \widetilde{\Lambda}$.

At first, we introduce a family of partitions $(\mathcal{Q}_k)_{k \ge 0}$ of $\Delta$ as 

\[\mathcal{Q}_0:=\{\Lambda_i \times \{l\}, i \ge 1, l < R_i\}, \mathcal{Q}_k:= \bigvee_{0 \le i \le k} F^{-i} \mathcal{Q}_0.\]

Next, a projection $\pi: \Delta \to M $ is defined as 
\[\pi(x,l):=f^l(x).\]

Then there exists a constant $C>1$ (the same as in the Definition \ref{gibbs}) such that for any $Q \in \mathcal{Q}_{2k}$ 
\begin{equation}\label{coronadiam}
   \diam (\pi \circ F^k(Q)) \le \frac{C}{k^{\alpha}}. 
\end{equation}

There exist also probability measures $\mu_{\Delta}, \mu_{\Lambda}$ on $\Delta$ and $\Lambda$, respectively, such that  
\begin{equation}\label{allmeasure}
    \pi_{*}\mu_{\Delta}=\mu, \text{ } F_{*}\mu_{\Delta}=\mu_{\Delta}, \text{ } f_{*}\mu=\mu, \text{ } (f^R)_{*}\mu_{\Lambda}=\mu_{\Lambda}.
\end{equation}

Further, there exist probability measures $\mu_{\widetilde{\Delta}}, \mu_{\widetilde{\Lambda}}$ on $\widetilde{\Delta}$ and $\widetilde{\Lambda}$ respectively, such that  
\begin{equation}\label{allquotientmeasure}
    (\widetilde{\pi}_{\Delta})_{*}\mu_{\Delta}=\mu_{\widetilde{\Delta}}, \text{ } (\widetilde{\pi}_{\Lambda})_{*}\mu_{\Lambda}=\mu_{\widetilde{\Lambda}},  \text{ } \widetilde{F}_{*}\mu_{\widetilde{\Delta}}=\mu_{\widetilde{\Delta}}, \text{ }  (\widetilde{f^R})_{*}\mu_{\widetilde{\Lambda}}=\mu_{\widetilde{\Lambda}}.
\end{equation}

Thus $\mu$ is supported on $\bigcup_{i \ge 1} \bigcup_{j < R_i} f^j(\Lambda_i)$, i.e. \[\mu(\bigcup_{i \ge 1} \bigcup_{j < R_i} f^j(\Lambda_i))=1.\]

Moreover 
\begin{equation}\label{horsesrbleb}
    (\mu_{\Lambda})_{\gamma^u} \ll \Leb_{\gamma^u}, \text{ }  \frac{d(\mu_{\Lambda})_{\gamma^u}}{d\Leb_{\gamma^u}}= C^{\pm 1},
\end{equation}
where $(\mu_{\Lambda})_{\gamma^u}$ is the conditional measure of $\mu_{\Lambda}$ on $\gamma^u \in \Gamma^u$. Since $R$ is the first return time, (see the Assumption \ref{assumption}), then

\[\mu_{\Lambda}=\frac{\mu|_{\Lambda}}{\mu (\Lambda)}.\]

 Finally, for any $k\ge 1$ and any $(Q_i)_{i \ge 1} \subseteq \mathcal{Q}_k$, any $h: \Delta \to \mathbb{R}$ satisfying $||h||_{\infty}\le 1$ and $h(x,l)=h(y,l)$ for any $x,y \in \gamma^s \in \Gamma^s$ , and any allowable $l \in \mathbb{N}$, (i.e. $h$ is $\sigma(\bigcup_{k\ge 0}\mathcal{Q}_k)$-measurable), we have the following estimate for decay of correlations
\begin{equation}\label{decorrelation}
    |\int 1_{\bigcup_{i \ge 1} Q_i} \cdot h \circ F^{2k} d\mu_{\Delta}-\mu_{\Delta}(\bigcup_{i \ge 1} Q_i) \cdot \int h d\mu_{\Delta}| \le \frac{C}{k^{\xi-1}} \cdot \mu_{\Delta}(\bigcup_{i \ge 1} Q_i).
\end{equation}
\end{lemma}

\begin{lemma}\label{decaylemma1}
For any $l\ge 0, \frac{1}{p^{\alpha}}\ll r$ and any measurable function $h$ with values in $[0, 1]$
\[| \mathbb{E}[1_{X_0=1} \cdot  h(X_p,\cdots, X_{p+l})]-\mathbb{E}1_{X_0=1} \cdot \mathbb{E}h(X_p,\cdots, X_{p+l})|\]
\[\le 4\cdot \frac{C}{p^{\xi-1}} \cdot \mu(B_{r+\frac{4^{\alpha}C}{p^{\alpha}}}(z))+[2 +4l \cdot \mu(B_{r+\frac{4^{\alpha}C}{p^{\alpha}}}(z))] \cdot \mu(B_{r+\frac{4^{\alpha}C}{p^{\alpha}}}(z)\setminus B_{r-\frac{4^{\alpha}C}{p^{\alpha}}}(z)), \]
where a constant $C$ is the same as in the Definition \ref{gibbs}.
\end{lemma}
\begin{proof}
Let $m:=\lfloor \frac{p}{4} \rfloor$. By (\ref{allmeasure}) and invariance of $F$ (i.e. $F_{*}\mu_{\Delta}=\mu_{\Delta}$) we have
\[ \mathbb{E}[1_{X_0=1} \cdot  h(X_p,\cdots, X_{p+l})]=\int 1_{B_r(z)} \cdot h(1_{B_r(z)} \circ f^p,\cdots, 1_{B_r(z)} \circ f^{p+l} ) d\mu\]
\[=\int 1_{B_r(z)} \circ \pi \circ F^{m}\cdot h(1_{B_r(z)} \circ \pi \circ  F^{p+m-p},\cdots, 1_{B_r(z)} \circ \pi \circ  F^{p+l+m-p} )\circ F^p d\mu_{\Delta}.\]
Denote
 $A_1:=F^{-m} \pi^{-1} B_r(z), A_0:=\bigsqcup_{Q \in \mathcal{Q}_{2m}:Q \bigcap A_1\neq \emptyset}Q$ and $A_2:=\bigsqcup_{Q \in \mathcal{Q}_{2m}:Q \bigcap (A_0\setminus A_1)\neq \emptyset}Q$. Then $A_1 \bigcup A_2=A_0$. The sets $A_0,A_2$ are  $\sigma(\bigcup_{k\ge 0}\mathcal{Q}_k)$-measurable. Therefore we can continue the equality above as 
\[=\int 1_{A_1} \cdot h(1_{A_1},\cdots, 1_{A_1}\circ F^{l}) \circ F^p d\mu_{\Delta}=\int 1_{A_1} \cdot h(1_{A_1\bigcup A_2},\cdots,1_{A_1\bigcup A_2}\circ F^{l}) \circ F^p d\mu_{\Delta}\]
\[+\int 1_{A_1} \cdot h(1_{A_1},\cdots, 1_{A_1}\circ F^{l}) \circ F^p d\mu_{\Delta}-\int 1_{A_1} \cdot h(1_{A_1\bigcup A_2},\cdots,1_{A_1\bigcup A_2}\circ F^{l}) \circ F^p d\mu_{\Delta}.\]

{\bf Claim}: $| h(1_{A_1},\cdots, 1_{A_1}\circ F^{l})- h(1_{A_1\bigcup A_2},\cdots,1_{A_1\bigcup A_2}\circ F^{l})|\le 2\cdot 1_{\cup_{j \le l}F^{-j}A_2}$. 

Indeed, if $F^j(x,l) \notin A_2 $ for all $j\le l$, then $1_{A_1}\circ F^{j-p}(x,l)=1_{A_1\cup A_2}\circ F^{j-p}(x,l)$. On the other hand, $||h_j||_{\infty} \le 1$. Hence, the claim holds. 

Therefore,
\[| \mathbb{E}[1_{X_0=1} \cdot  h(X_p,\cdots, X_{p+l})]-\mathbb{E}1_{X_0=1} \cdot \mathbb{E}h(X_p,\cdots, X_{p+l})|\]
\[=|\int 1_{A_1} \cdot h(1_{A_1\bigcup A_2},\cdots,1_{A_1\bigcup A_2}\circ F^{l}) \circ F^p d\mu_{\Delta}- \int 1_{A_1} d\mu_{\Delta} \cdot \int h(1_{A_1\bigcup A_2},\cdots,1_{A_1\bigcup A_2}\circ F^{l}) \circ F^p d\mu_{\Delta}\]
\[+\int 1_{A_1} \cdot h(1_{A_1},\cdots, 1_{A_1}\circ F^{l}) \circ F^p d\mu_{\Delta}-\int 1_{A_1} \cdot h(1_{A_1\bigcup A_2},\cdots,1_{A_1\bigcup A_2}\circ F^{l}) \circ F^p d\mu_{\Delta}\]
\[-\int 1_{A_1} d\mu_{\Delta} \cdot \int h(1_{A_1},\cdots, 1_{A_1}\circ F^{l}) \circ F^p d\mu_{\Delta}+\int 1_{A_1} d\mu_{\Delta} \cdot \int h(1_{A_1\bigcup A_2},\cdots,1_{A_1\bigcup A_2}\circ F^{l}) \circ F^p d\mu_{\Delta}|\]
\[\le |\int 1_{A_1} \cdot h(1_{A_1\bigcup A_2},\cdots,1_{A_1\bigcup A_2}\circ F^{l}) \circ F^p d\mu_{\Delta}- \int 1_{A_1} d\mu_{\Delta} \cdot \int h(1_{A_1\bigcup A_2},\cdots,1_{A_1\bigcup A_2}\circ F^{l}) \circ F^p d\mu_{\Delta}|\]
\[+2\int 1_{A_1} \cdot 1_{\bigcup_{j \le l}F^{-j}A_2} \circ F^p d\mu_{\Delta}+ 2\int 1_{A_1} d\mu_{\Delta} \cdot \int 1_{\bigcup_{j \le l}F^{-j}A_2}\circ F^p d\mu_{\Delta}\]
\[\le |\int 1_{A_0} \cdot h(1_{A_1\bigcup A_2},\cdots,1_{A_1\bigcup A_2}\circ F^{l}) \circ F^p d\mu_{\Delta}- \int 1_{A_0} d\mu_{\Delta} \cdot \int h(1_{A_1\bigcup A_2},\cdots,1_{A_1\bigcup A_2}\circ F^{l}) \circ F^p d\mu_{\Delta}|\]
\[+ |\int 1_{A_0\setminus A_1} \cdot h(1_{A_1\bigcup A_2},\cdots,1_{A_1\bigcup A_2}\circ F^{l})\circ F^p d\mu_{\Delta}- \int 1_{A_0\setminus A_1} d\mu_{\Delta} \cdot \int h(1_{A_1\bigcup A_2},\cdots,1_{A_1\bigcup A_2}\circ F^{l})\circ F^p d\mu_{\Delta}|\]
\[+2\int 1_{A_1} \cdot 1_{\bigcup_{j \le l}F^{-j}A_2}\circ F^p d\mu_{\Delta}+ 2\int 1_{A_1} d\mu_{\Delta} \cdot \int 1_{\bigcup_{j \le l}F^{-j}A_2} \circ F^p d\mu_{\Delta}.\]

Note that $A_0\setminus A_1 \subseteq A_2, A_1 \subseteq A_0$, which means that we can continue the estimate above as 
\[\le |\int 1_{A_0} \cdot h(1_{A_1\bigcup A_2},\cdots,1_{A_1\bigcup A_2}\circ F^{l}) \circ F^p d\mu_{\Delta}- \int 1_{A_0} d\mu_{\Delta} \cdot \int h(1_{A_1\bigcup A_2},\cdots,1_{A_1\bigcup A_2}\circ F^{l})\circ F^p d\mu_{\Delta}|\]
\[+ \int 1_{A_2} \cdot h(1_{A_1\bigcup A_2},\cdots,1_{A_1\bigcup A_2}\circ F^{l}) \circ F^p d\mu_{\Delta}+ \int 1_{A_2} d\mu_{\Delta} \cdot \int h(1_{A_1\bigcup A_2},\cdots,1_{A_1\bigcup A_2}\circ F^{l})\circ F^p d\mu_{\Delta}\]
\[+2\int 1_{A_0} \cdot 1_{\bigcup_{j \le l}F^{-j}A_2} \circ F^p d\mu_{\Delta}+ 2\int 1_{A_1} d\mu_{\Delta} \cdot \int 1_{\bigcup_{j \le l}F^{-j}A_2} d\mu_{\Delta}\]
\[\le |\int 1_{A_0} \cdot h(1_{A_1\bigcup A_2},\cdots,1_{A_1\bigcup A_2}\circ F^{l}) \circ F^p d\mu_{\Delta}- \int 1_{A_0} d\mu_{\Delta} \cdot \int h(1_{A_1\bigcup A_2},\cdots,1_{A_1\bigcup A_2}\circ F^{l})\circ F^p d\mu_{\Delta}|\]
\[+ |\int 1_{A_2} \cdot h(1_{A_1\bigcup A_2},\cdots,1_{A_1\bigcup A_2}\circ F^{l}) \circ F^p d\mu_{\Delta}- \int 1_{A_2} d\mu_{\Delta} \cdot \int h(1_{A_1\bigcup A_2},\cdots,1_{A_1\bigcup A_2}\circ F^{l}) \circ F^p d\mu_{\Delta}|\]
\[+2\int 1_{A_2} d\mu_{\Delta} \cdot \int h(1_{A_1\bigcup A_2},\cdots,1_{A_1\bigcup A_2}\circ F^{l})\circ F^p d\mu_{\Delta}\]
\[+2|\int 1_{A_0} \cdot 1_{\bigcup_{j \le l}F^{-j}A_2} \circ F^p d\mu_{\Delta}-\int 1_{A_0}d\mu_{\Delta} \cdot \int 1_{\bigcup_{j \le l}F^{-j}A_2} \circ F^p d\mu_{\Delta}|\]
\begin{equation}\label{1}
    +2\int 1_{A_0}d\mu_{\Delta} \cdot \int 1_{\bigcup_{j \le l}F^{-j}A_2} \circ F^p d\mu_{\Delta}+ 2\int 1_{A_1} d\mu_{\Delta} \cdot \int 1_{\cup_{j \le l}F^{-j}A_2} d\mu_{\Delta}.
\end{equation}

{\bf Claim}: $h(1_{A_1\bigcup A_2},\cdots,1_{A_1\bigcup A_2}\circ F^{l})$ is $\sigma(\bigcup_{k\ge 0}\mathcal{Q}_k)$-measurable. 
Observe, that for any $(x,l), (y,l) \in \Delta$, $x,y\in \gamma^s \in \Gamma^s$ we have $F^{j-p}(x,l)=(x',l'), F^{j-p}(y,l)=(y',l')$ for some $l'\in \mathbb{N}$ and some $x',y'\in (\gamma^s)'\in \Gamma^s$. Since $1_{A_1 \bigcup A_2}$ is $\sigma(\bigcup_{k\ge 0}\mathcal{Q}_k)$-measurable,
$1_{A_1\cup A_2}\circ F^{j-p}(x,l)=1_{A_1\cup A_2}\circ F^{j-p}(y,l)$. Therefore  $h(1_{A_1\bigcup A_2},\cdots,1_{A_1\bigcup A_2}\circ F^{l})$ is  $\sigma(\cup_{k\ge 0}\mathcal{Q}_k)$-measurable. 

{\bf Claim:} $1_{\bigcup_{j \le l}F^{-j}A_2} $ is also $\sigma(\bigcup_{k\ge 0}\mathcal{Q}_k)$-measurable. 

Indeed, each set $F^{-j}A_2$ is $\sigma(\bigcup_{k\ge 0}\mathcal{Q}_k)$-measurable. So their union is also $\sigma(\bigcup_{k\ge 0}\mathcal{Q}_k)$-measurable.

{\bf Claim:} $\mu_{\Delta}(A_2)\le \mu(B_{r+\frac{4^{\alpha}C}{p^{\alpha}}}(z)\setminus B_{r-\frac{4^{\alpha}C}{p^{\alpha}}}(z))$. 
 
Observe that \[\mu_{\Delta}(A_2)\le \mu_{\Delta}(F^{-m}\pi^{-1}\pi( F^m A_2))=\mu(\pi( F^m A_2)).\]

By the definition of $A_2:=\bigsqcup_{Q \in \mathcal{Q}_{2m}:Q \cap (A_0\setminus A_1)\neq \emptyset}Q$, for each $Q$, contained in $A_2$, there exist $x_1, x_2 \in Q $, such that  $\pi(F^mx_1)\in B_r(z), \pi(F^mx_2)\notin B_r(z)$. Now, by making use of (\ref{coronadiam}) and taking $m=\lfloor \frac{p}{4} \rfloor$, we obtain $\pi(F^m A_2) \subseteq B_{r+\frac{4^{\alpha}C}{p^{\alpha}}}(z)\setminus B_{r-\frac{4^{\alpha}C}{p^{\alpha}}}(z)$. Hence the claim holds. 

Having these claims and (\ref{decorrelation}), we can continue the estimate of (\ref{1}) as \[\le \frac{C}{p^{\xi-1}} \cdot \mu_{\Delta}(A_0) +\frac{C}{p^{\xi-1}} \cdot \mu_{\Delta}(A_2)+2\mu_{\Delta}(A_2) +2 \cdot \frac{C}{p^{\xi-1}} \cdot \mu_{\Delta}(A_0)\]
\[+2\mu_{\Delta}(A_0) \cdot  \mu_{\Delta}(\bigcup_{j \le l}F^{-j}A_2)+ 2 \mu_{\Delta}(A_1) \cdot \mu_{\Delta}(\bigcup_{j \le l}F^{-j}A_2)\]
\[\le \frac{C}{p^{\xi-1}} \cdot \mu_{\Delta}(A_1)+\frac{C}{p^{\xi-1}} \cdot \mu_{\Delta}(A_2) +\frac{C}{p^{\xi-1}} \cdot \mu_{\Delta}(A_2)+2\mu_{\Delta}(A_2) +2 \cdot \frac{C}{p^{\xi-1}} \cdot \mu_{\Delta}(A_1)\]
\begin{equation}\label{10}
    +2 \cdot \frac{C}{p^{\xi-1}} \cdot \mu_{\Delta}(A_2)+2[\mu_{\Delta}(A_1)+\mu_{\Delta}(A_2)] \cdot  \mu_{\Delta}(\bigcup_{j \le l}F^{-j}A_2)+ 2 \mu_{\Delta}(A_1) \cdot \mu_{\Delta}(\bigcup_{j \le l}F^{-j}A_2).
\end{equation}

Recall that $\mu(B_r(z))=\mu_{\Delta}(A_1)$ and $\mu_{\Delta}(A_2)\le \mu(B_{r+\frac{C}{p^{\alpha}}}(z)\setminus B_{r-\frac{C}{p^{\alpha}}}(z))$. Therefore the estimate of (\ref{10}) can be continued as
\[\le \frac{C}{p^{\xi-1}} \cdot \mu(B_r(z))+2\cdot \frac{C}{p^{\xi-1}} \cdot \mu(B_{r+\frac{4^{\alpha}C}{p^{\alpha}}}(z)\setminus B_{r-\frac{4^{\alpha}C}{p^{\alpha}}}(z)) +2\mu(B_{r+\frac{4^{\alpha}C}{p^{\alpha}}}(z)\setminus B_{r-\frac{4^{\alpha}C}{p^{\alpha}}}(z)) \]
\[+2 \cdot \frac{C}{p^{\xi-1}} \cdot \mu(B_r(z))+2 \cdot \frac{C}{p^{\xi-1}} \cdot \mu(B_{r+\frac{4^{\alpha}C}{p^{\alpha}}}(z)\setminus B_{r-\frac{4^{\alpha}C}{p^{\alpha}}}(z))\]
\[+2[\mu(B_r(z))+\mu(B_{r+\frac{4^{\alpha}C}{p^{\alpha}}}(z)\setminus B_{r-\frac{4^{\alpha}C}{p^{\alpha}}}(z))] \cdot l \cdot \mu(B_{r+\frac{4^{\alpha}C}{p^{\alpha}}}(z)\setminus B_{r-\frac{4^{\alpha}C}{p^{\alpha}}}(z))\]
\[+ 2 \mu(B_r(z)) \cdot l \cdot \mu(B_{r+\frac{4^{\alpha}C}{p^{\alpha}}}(z)\setminus B_{r-\frac{4^{\alpha}C}{p^{\alpha}}}(z))\]
\[\le 4\cdot \frac{C}{p^{\xi-1}} \cdot \mu(B_{r+\frac{4^{\alpha}C}{p^{\alpha}}}(z))+2\mu(B_{r+\frac{4^{\alpha}C}{p^{\alpha}}}(z)\setminus B_{r-\frac{4^{\alpha}C}{p^{\alpha}}}(z)) +4l \cdot \mu(B_{r+\frac{4^{\alpha}C}{p^{\alpha}}}(z))  \cdot \mu(B_{r+\frac{4^{\alpha}C}{p^{\alpha}}}(z)\setminus B_{r-\frac{4^{\alpha}C}{p^{\alpha}}}(z)).\]
\end{proof}
\begin{proposition}[Functional Poisson laws]\ \label{rate}

Let $p:=\lfloor {\lfloor \frac{T}{\mu(B_r(z))}\rfloor}^{\frac{\dim_H\mu-\epsilon}{\dim_H\mu}}\rfloor$, where $\epsilon$ is so small that

\[\alpha \cdot (\dim_H\mu-\epsilon)>\frac{\dim_H\mu}{\dim_H\mu-\epsilon}>1 \text{ (see the Assumption } \ref{geoassumption}).\] Then for almost any $z \in \mathcal{M}$ there is $r_z>0$, such that $\forall r< r_z$

\[d_{TV}(N^{r,T,z},P)
\precsim_{T, \xi, \epsilon}\]
\[r^{\dim_H\mu-\epsilon}+r^{\frac{(\dim_H\mu-\epsilon)^2}{\dim_H\mu} \cdot (\xi-1)}+ r^{ \frac{\epsilon \cdot(\dim_H\mu-\epsilon)}{\dim_H\mu}}+\frac{\mu(B_{r+C' \cdot r^{\frac{(\dim_H\mu-\epsilon)^2\cdot \alpha }{\dim_H\mu} } }(z)\setminus B_{r-C' \cdot r^{\frac{(\dim_H\mu-\epsilon)^2\cdot \alpha }{\dim_H\mu} }}(z))}{\mu(B_r(z))}\]
\[+[\frac{\mu(B_{r+C' \cdot r^{\frac{(\dim_H\mu-\epsilon)^2\cdot \alpha }{\dim_H\mu}  } }(z)\setminus B_{r-C' \cdot r^{\frac{(\dim_H\mu-\epsilon)^2\cdot \alpha }{\dim_H\mu}}}(z))}{\mu(B_r(z))}]^2+\frac{1}{\mu(B_r(z))}\cdot \int_{B_r(z)} 1_{\bigcup_{1\le j \le p}f^{-j}B_r(z)}d\mu,\]
where $C'$ depends on $\alpha, \epsilon $ and on the constant $C$ in the Definition \ref{gibbs}.
\end{proposition}

\begin{proof}
Let $n:=\lfloor \frac{T}{\mu(B_r(z))}\rfloor, p:=\lfloor n^{\frac{\dim_H\mu-\epsilon}{\dim_H\mu}}\rfloor$, where $\epsilon$ is so small that

\[\alpha \cdot (\dim_H\mu-\epsilon)>\frac{\dim_H\mu}{\dim_H\mu-\epsilon}>1 \text{ (in view of the Assumption } \ref{geoassumption}).\] 

Therefore for a.e. $z \in \mathcal{M}$ there exists $r_z>0$ such that $\forall r < r_z$, in view of the Assumption \ref{geoassumption}, 
\[\frac{T}{r^{\dim_H\mu-\epsilon}} \precsim n \precsim \frac{T}{r^{\dim_H\mu+\epsilon}}\]
\[\frac{1}{p^{\alpha}} \precsim_{T, \alpha} \frac{1}{n^{\frac{\alpha \cdot (\dim_H\mu-\epsilon)}{\dim_H\mu}}}\precsim_{T, \alpha} r^{\frac{\dim_H\mu-\epsilon}{\dim_H\mu} \cdot \alpha \cdot (\dim_H\mu-\epsilon)} \ll r. \]

Hence by the Lemmas \ref{poissonappro} and \ref{decaylemma1}, for any disjoint sets $I_1, \cdots, I_m \subseteq [0,n]$ we have
\[d_{TV}((X_{I_1}, \cdots, X_{I_m}),(\hat{X}_{I_1}, \cdots, \hat{X}_{I_m}))\]
 \[\le 2\cdot  \sum_{0 \le l \le n-p}\sup_{h\in[0,1]} | \mathbb{E}[1_{X_0=1} \cdot  h(X_p,\cdots, X_{n-l})]-\mathbb{E}1_{X_0=1} \cdot \mathbb{E}[h(X_p,\cdots, X_{n-l})]|\]
\[+4(n-p)\cdot \mathbb{E} 1_{X_0=0} \cdot 1_{\sum_{1\le j \le p-1}X_j\ge 1}+4p \cdot (n-p) \cdot \mu(B_{r}(z))^2+4 p \cdot \mu(B_r(z))\]
\[\le \sum_{0 \le l \le n-p} 8\cdot \frac{C}{p^{\xi-1}} \cdot \mu(B_{r+\frac{4^{\alpha}C}{p^{\alpha}}}(z))+[4 +8(n-l-p) \cdot \mu(B_{r+\frac{4^{\alpha}C}{p^{\alpha}}}(z))] \cdot \mu(B_{r+\frac{4^{\alpha}C}{p^{\alpha}}}(z)\setminus B_{r-\frac{4^{\alpha}C}{p^{\alpha}}}(z)) \]
\[+4(n-p)\cdot \mathbb{E} 1_{X_0=0} \cdot 1_{\sum_{1\le j \le p-1}X_j\ge 1}+4p \cdot (n-p) \cdot \mu(B_{r}(z))^2+4 p \cdot \mu(B_r(z))\]
\[\le  8\cdot n\cdot \frac{C}{p^{\xi-1}} \cdot \mu(B_{r+\frac{4^{\alpha}C}{p^{\alpha}}}(z))+[4 +8n \cdot \mu(B_{r+\frac{4^{\alpha}C}{p^{\alpha}}}(z))] \cdot n \cdot \mu(B_{r+\frac{4^{\alpha}C}{p^{\alpha}}}(z)\setminus B_{r-\frac{4^{\alpha}C}{p^{\alpha}}}(z)) \]
\[+4n\cdot \mathbb{E} 1_{X_0=0} \cdot 1_{\sum_{1\le j \le p-1}X_j\ge 1}+4p \cdot n \cdot \mu(B_{r}(z))^2+4 p \cdot \mu(B_r(z))\]
\[\le  8\cdot n\cdot \frac{C}{p^{\xi-1}} \cdot \mu(B_{r}(z))+ 8\cdot n\cdot \frac{C}{p^{\xi-1}} \cdot \mu(B_{r+\frac{4^{\alpha}C}{p^{\alpha}}}(z)\setminus B_{r-\frac{4^{\alpha}C}{p^{\alpha}}}(z))\]
\[+[4 +8n \cdot \mu(B_{r}(z))+8n \cdot \mu(B_{r+\frac{4^{\alpha}C}{p^{\alpha}}}(z)\setminus B_{r-\frac{4^{\alpha}C}{p^{\alpha}}}(z))] \cdot n \cdot \mu(B_{r+\frac{4^{\alpha}C}{p^{\alpha}}}(z)\setminus B_{r-\frac{4^{\alpha}C}{p^{\alpha}}}(z)) \]
\[+4n\cdot \mathbb{E} 1_{X_0=0} \cdot 1_{\sum_{1\le j \le p}X_j\ge 1}+4p \cdot n \cdot \mu(B_{r}(z))^2+4 p \cdot \mu(B_r(z)).\]

Note that
\[p \approx_{T,\epsilon} [\frac{1}{\mu(B_r(z))}]^{\frac{\dim_H\mu-\epsilon}{\dim_H\mu}}, \text{ }  n \cdot \mu(B_r(z))\le T.\]

Thus we can continue the inequality above as
\[\precsim_{T, \xi, \epsilon} \mu(B_r(z))^{\frac{\dim_H\mu-\epsilon}{\dim_H\mu} \cdot (\xi-1)}+ \frac{\mu(B_{r+\frac{4^{\alpha}C}{p^{\alpha}}}(z)\setminus B_{r-\frac{4^{\alpha}C}{p^{\alpha}}}(z))}{\mu(B_r(z))}+[\frac{\mu(B_{r+\frac{4^{\alpha}C}{p^{\alpha}}}(z)\setminus B_{r-\frac{4^{\alpha}C}{p^{\alpha}}}(z))}{\mu(B_r(z))}]^2\]
\[+\frac{1}{\mu(B_r(z))}\cdot \int_{B_r(z)} 1_{\bigcup_{1\le j \le p}f^{-j}B_r(z)}d\mu+\mu(B_r(z))^{\frac{\epsilon}{\dim_H\mu}}.\]

By applying the Theorem 2 of \cite{chenmethod} to $(\hat{X}_{I_1}, \cdots, \hat{X}_{I_m})$ one gets
\[d_{TV}((\hat{X}_{I_1}, \cdots, \hat{X}_{I_m}), (P(I_1), \cdots, P(I_m)))\le 4 \cdot n \cdot \mu(B_r(z))^2 \precsim_T \mu(B_r(z)). \]

Approximate now $(X_{I_1}, \cdots, X_{I_m})$ by the Poisson point process $P$. Then
\[d_{TV}((X_{I_1}, \cdots, X_{I_m}),(P(I_1), \cdots, P(I_m)))\]
\[\le d_{TV}((\hat{X}_{I_1}, \cdots, \hat{X}_{I_m}), (P(I_1), \cdots, P(I_m)))+d_{TV}((X_{I_1}, \cdots, X_{I_m}),(\hat{X}_{I_1}, \cdots, \hat{X}_{I_m}))\]
\[\precsim_{T, \xi, \epsilon}\mu(B_r(z))+ \mu(B_r(z))^{\frac{\dim_H\mu-\epsilon}{\dim_H\mu} \cdot (\xi-1)}+ \frac{\mu(B_{r+\frac{4^{\alpha}C}{p^{\alpha}}}(z)\setminus B_{r-\frac{4^{\alpha}C}{p^{\alpha}}}(z))}{\mu(B_r(z))}\]
\[+[\frac{\mu(B_{r+\frac{4^{\alpha}C}{p^{\alpha}}}(z)\setminus B_{r-\frac{4^{\alpha}C}{p^{\alpha}}}(z))}{\mu(B_r(z))}]^2+\frac{1}{\mu(B_r(z))}\cdot \int_{B_r(z)} 1_{\bigcup_{1\le j \le p}f^{-j}B_r(z)}d\mu+\mu(B_r(z))^{\frac{\epsilon}{\dim_H\mu}}.\]

Since $\mathcal{C}':=\{\pi_{I'_1}^{-1}A_1 \cap \cdots \cap \pi_{I'_m}^{-1}A_m: \forall m \ge 1,  A_i \subseteq \mathbb{N}, \forall \text{ disjoint sets } I'_1, \cdots, I'_m \subseteq [0,T]\}$ generates $\sigma$-algebra
$\mathcal{C}=\sigma \{\pi^{-1}_AB: \text{ any Borel sets } A \subseteq [0,T], B \subseteq \mathbb{N}\}$, 
we obtain $\forall r< r_z$ the following  functional Poisson approximation
\[d_{TV}(N^{r,T,z},P)=\sup_{\text{disjoint }I_i \subseteq [0,n]}d_{TV}((X_{I_1}, \cdots, X_{I_m}), (P(I_1), \cdots, P(I_m)))\]
\[\precsim_{T, \xi, \epsilon}\mu(B_r(z))+ \mu(B_r(z))^{\frac{\dim_H\mu-\epsilon}{\dim_H\mu} \cdot (\xi-1)}+\mu(B_r(z))^{\frac{\epsilon}{\dim_H\mu}} +\frac{\mu(B_{r+\frac{4^{\alpha}C}{p^{\alpha}}}(z)\setminus B_{r-\frac{4^{\alpha}C}{p^{\alpha}}}(z))}{\mu(B_r(z))}\]
\[+[\frac{\mu(B_{r+\frac{4^{\alpha}C}{p^{\alpha}}}(z)\setminus B_{r-\frac{4^{\alpha}C}{p^{\alpha}}}(z))}{\mu(B_r(z))}]^2+\frac{1}{\mu(B_r(z))}\cdot \int_{B_r(z)} 1_{\bigcup_{1\le j \le p}f^{-j}B_r(z)}d\mu\]
\[\precsim_{T, \xi, \epsilon}r^{\dim_H\mu-\epsilon}+r^{\frac{(\dim_H\mu-\epsilon)^2}{\dim_H\mu} \cdot (\xi-1)}+ +r^{ \frac{\epsilon \cdot(\dim_H\mu-\epsilon)}{\dim_H\mu}}+\frac{\mu(B_{r+C' \cdot r^{\frac{(\dim_H\mu-\epsilon)^2\cdot \alpha }{\dim_H\mu} } }(z)\setminus B_{r-C' \cdot r^{\frac{(\dim_H\mu-\epsilon)^2\cdot \alpha }{\dim_H\mu} }}(z))}{\mu(B_r(z))}\]
\[+[\frac{\mu(B_{r+C' \cdot r^{\frac{(\dim_H\mu-\epsilon)^2\cdot \alpha }{\dim_H\mu}  } }(z)\setminus B_{r-C' \cdot r^{\frac{(\dim_H\mu-\epsilon)^2\cdot \alpha }{\dim_H\mu}}}(z))}{\mu(B_r(z))}]^2+\frac{1}{\mu(B_r(z))}\cdot \int_{B_r(z)} 1_{\bigcup_{1\le j \le p}f^{-j}B_r(z)}d\mu,\]
where $C'$ depends upon $\alpha, \epsilon $ and the constant $C$ in the Definition \ref{gibbs}. 
\end{proof}

\begin{definition}[Short returns and Coronas]\label{defshortcoro}\ \par
Let $p$ be the one in the Proposition \ref{rate}. Define
\begin{enumerate}
    \item {\bf Short returns}:
    \[\int_{B_r(z)} 1_{\bigcup_{1\le j \le p}f^{-j}B_r(z)}d\mu.\] 

    \item {\bf Coronas}:
    \[\mu(B_{r+C' \cdot r^{\frac{(\dim_H\mu-\epsilon)^2\cdot \alpha }{\dim_H\mu} } }(z)\setminus B_{r-C' \cdot r^{\frac{(\dim_H\mu-\epsilon)^2\cdot \alpha }{\dim_H\mu} }}(z)).\]
\end{enumerate}

It will be shown below that these quantities tend to $0$ for almost all  $z \in \mathcal{M}$ with certain convergence rates.
\end{definition}
\section{Proof of the Theorem \ref{thm}}\label{provethm1}
\subsection{Properties of the First Return to  $\Lambda$}
Before studying convergence rates for short returns and coronas we will prove several lemmas for the first return time $R$ and for the first return map $f^R: \Lambda \to \Lambda$ under the Assumptions \ref{geoassumption} and \ref{assumption}.

\begin{lemma}\label{inducemapbi}
The map $f^R: \Lambda \to \Lambda$ is bijective.
\end{lemma}
\begin{proof}
We show first that $f^R$ is {\bf one-to-one} for any $x,y \in \Lambda$, and $f^R(x)=f^R(y)$. If $x, y \in \Lambda_i $ for some $i$,  then $f^{R_i}(x)=f^{R_i}(y)$. On the other hand, it follows from the Assumption \ref{geoassumption} that $f$ is bijective on $\cup_{i \ge 1} \cup_{j < R_i} f^j(\Lambda_i)$. Thus we have inductively the following reduction
\[f(f^{R_i-1}x)=f(f^{R_i-1}y) \Rightarrow f(f^{R_i-2}x)=f(f^{R_i-2}y)\Rightarrow \cdots \Rightarrow f(x)=f(y) \Rightarrow x=y. \]
Let $x \in \Lambda_i, y \in \Lambda_j$ for some $i \neq j$ and $f^R(x)=f^R(y)$. Without any loss of generality, we may assume that $R_i<R_j$. Then   $f^{R_i}(x)=f^{R_j}(y)$. Again, by the Assumptions \ref{geoassumption}  
\[f(f^{R_i-1}x)=f(f^{R_j-1}y) \Rightarrow f(f^{R_i-2}x)=f(f^{R_j-2}y)\Rightarrow \cdots  \Rightarrow x=f^{R_j-R_i}y \in \Lambda. \]
But the first return time of $y$ to $\Lambda$  is $R_j$, i.e., $f^{R_j-R_i}y \notin \Lambda$. So we came to a contradiction, and therefore this case can not occur.  

Show now that $f^R$ is {\bf onto} map. Let $y \in \Lambda$ and $y\in \Lambda_i$ for some $i$. Then $y \in \cup_{i \ge 1} \cup_{j < R_i} f^j(\Lambda_i)$. By the Assumption \ref{geoassumption} $f$ is bijective on $\cup_{i \ge 1} \cup_{j < R_i} f^j(\Lambda_i)$. Therefore there exists $x' \in \cup_{i \ge 1} \cup_{j < R_i} f^j(\Lambda_i)$, i.e. there is $x\in \Lambda_k$ such that $f^j(x)=x'$, where $j < R_k$ and $f(x')=f^{j+1}(x)=y$. Since $R_k$ is the first return time for $x$, then $j+1=R_k$ and $f^R(x)=y$.
\end{proof}

\begin{lemma}\label{pibi}
The following maps satisfy the properties that 
\[\pi: \Delta \to \cup_{i \ge 1} \cup_{ j < R_i} f^j(\Lambda_i) \text{ is bijective,}\]
\[\pi: \Delta_0  \to \Lambda \text{ is identity},\] 
\[\pi:\Delta_{\ge 1} \to \bigcup_{i \ge 1}\bigcup_{ 1 \le j < R_i} f^j(\Lambda_i) \text{ is bijective},\] 
\[\text{where } \Delta_{\ge 1}:=\{(x,l) \in \Lambda \times \mathbb{N}: 1 \le l < R(x)\},\]
\[\Delta_0:=\{(x,0): x \in \Lambda \}\subseteq \Delta.\]
\end{lemma}
\begin{proof}
Clearly, it is enough to prove just the first statement. By the definition of $\Delta$ the first map $\pi$ is onto. Let us show now that it is actually one-to-one map. For all $ (x,l), (x',l')\in \Delta $ with $\pi(x,l)=\pi(x',l')$ it holds that $f^l(x)=f^{l'}(x')$. Without loss of generality, let $l \le l'$. By the Assumption \ref{geoassumption} $f$ is bijective on $\cup_{i \ge 1} \cup_{j < R_i} f^j(\Lambda_i)$. Then
\[f(f^{l-1}x)=f(f^{l'-1}y) \Rightarrow f(f^{l-2}x)=f(f^{l'-2}y)\Rightarrow \cdots  \Rightarrow x=f^{l'-l}y. \]
Since $x,y \in \Lambda$ and $l'-l$ is less than the first return time of $y$, one gets that $l'=l$ and $x=y$.
\end{proof}

\par
By the Birkhoff's Ergodic Theorem, for almost every $z \in  \bigcup_{i \ge 1} \bigcup_{0 \le j < R_i} f^j(\Lambda_i)$ we have $z=f^{j_z}(z')$ for some $z' \in \interior{(\Lambda)}$ and $j_z \in \mathbb{N}$. Recall that $\mu\{\interior{(\Lambda)}\}>0$.  

\begin{lemma}[Pulling metric balls back to $\Lambda$]\ \label{ballinhorseshoe}\ \par

There exists a small enough $r$ such that 

\[\mu(f^{-j_z}B_{r}(z) \bigcap \Lambda^c)=0,\]

\[\mu_{\Delta}(\{\pi^{-1}f^{-j_z}B_r(z)\} \bigcap \Delta_{\ge 1})=0,\]
and 

\[\mu_{\Delta}(\{\pi^{-1}f^{-j_z}B_r(z)\} \bigcap \Delta_{0})=1.\]
\end{lemma}
\begin{proof}
By the Assumption \ref{assumption} there is a small neighborhood $U_{z'}\subseteq \mathcal{M}$ of $z'\in \interior{(\Lambda)}$ such that
    \[\mu(U_{z'}\bigcap \Lambda^c)=\mu(U_{z'} \bigcap \{\bigcup_{i \ge 1} \bigcup_{1\le j < R_i} f^j(\Lambda_i)\} )=0.\] 
Because $f^{j_z}$ is a local $C^1$-diffeomorphism, there exists a small ball $B_r(z)$ such that $f^{-j_z}  B_r(z) \subseteq U_{z'}$. So $\mu(f^{-j_z}B_{r}(z) \bigcap \Lambda^c)=\mu(f^{-j_z}B_{r}(z) \bigcap \{\bigcup_{i \ge 1} \cup_{1 \le j < R_i} f^j(\Lambda_i)\})=0$.
Hence by Lemma \ref{pibi}, 
 \[\mu_{\Delta}(\{\pi^{-1}f^{-j_z}B_r(z)\}\bigcap \Delta_{\ge 1})=0 \text{ and } \mu_{\Delta}(\{\pi^{-1}f^{-j_z}B_r(z)\} \bigcap \Delta_{0})=1.\]
\end{proof}

\begin{definition}[Topological balls]\ \par
We say that a set $TB_r(z') \subseteq \mathcal{M}$ is a topological ball if there is a ball $B_r(z) \subseteq \mathcal{M}$ and a map $T$ of $B_r(z)$, such that 

\[T:B_r(z) \to TB_r(z') \text{ is a } C^1 \text{-diffeomorphism and } T(z)=z'.\]

Denote by $r, z'$ the radius and the center of $TB_r(z')$.
\end{definition}
For almost every $z \in  \bigcup_{i \ge 1} \bigcup_{0 \le j < R_i} f^{j_z}(\Lambda_i)$ we have $z=f^j(z')$, where $z' \in \interior{(\Lambda)}$ and $j_z \in \mathbb{N}$. Since $f^{j_z}$ is a local diffeomorphism (by the Assumption \ref{geoassumption}), the set $TB_r(f^{-j_z}z):=f^{-j_z}B_r(z)$ is a topological ball for sufficiently small $r>0$.

\begin{lemma}[Comparisons of topological and metric balls]\ \label{comparetopballmetricball}\ \par 

There exist constants $C_z\ge 1$ and $ r_z>0$, such that $\forall r < r_z$
\[B_{C^{-1}_z\cdot r} (f^{-j_z}z) \subseteq TB_r(f^{-j_z}z) \subseteq B_{C_z\cdot r} (f^{-j_z}z).\]
\end{lemma}

\begin{proof}
Since $f^{-j_z}$ is a local diffeomorphism near $z$, then $f^{-j_z}(\partial B_r(z))=\partial TB_r(f^{-j_z}z)$. We will estimate $\sup_{x \in \partial TB_r(f^{-j_z}z)}d(x, f^{-j_z}z)$ and  $\inf_{x \in \partial TB_r(f^{-j_z}z)}d(x, f^{-j_z}z)$. For any $x\in \partial TB_r(f^{-j_z}z)$ one has $f^{j_z}(x) \in \partial B_r(z)$. Let $(\gamma_t)_{0 \le t \le 1}$ be the geodesic connecting $x$ and $f^{-j_z}z$, and a curve $\hat{\gamma}:=f^{j_z}\gamma$ is connecting $f^{j_z}x$ and $z$. Then
\begin{equation}\label{6}
    d(x, f^{-j_z}z)=\int_0^1 \sqrt{\langle \gamma_t',\gamma_t'\rangle_{\gamma_t}}dt=\int_0^1 \sqrt{\langle Df^{-j_z}\hat{\gamma}_t',Df^{-j_z}\hat{\gamma}_t'\rangle_{\gamma_t}}dt.
\end{equation}

If $r$ is sufficiently small (i.e. $r< r_z$ for some $r_z>0$), then $Df^{-j_z}$ and the Riemannian metric $\langle \cdot, \cdot \rangle_{\gamma_t}$ are close to $Df^{-j_z}(z)$ and $\langle \cdot, \cdot \rangle_{z}$, respectively. Then there exists $C_z\ge 1$ such that  

\[d(x, f^{-j_z}z)
\ge C^{-1}_z \cdot \int_0^1 \sqrt{\langle\hat{\gamma}_t',\hat{\gamma}_t'\rangle_{\hat{\gamma}_t}}dt \ge C^{-1}_z \cdot d(f^{j_z}x,z)=C^{-1}_z \cdot r.\]

Similarly, let $(\gamma_t)_{0 \le t \le 1}$ be a curve connecting $x$ and $f^{-j_z}z$, such that $\hat{\gamma}:=f^{j_z}\gamma$ is a geodesic connecting $f^{j_z}x$ and $z$. Then
\[d(x, f^{-j_z}z)\le \int_0^1 \sqrt{\langle\gamma_t',\gamma_t'\rangle_{\gamma_t}}dt=\int_0^1 \sqrt{\langle Df^{-j_z}\hat{\gamma}_t',Df^{-j_z}\hat{\gamma}_t'\rangle_{\gamma_t}}dt \]
\[\le C_z \cdot \int_0^1 \sqrt{\langle\hat{\gamma}_t',\hat{\gamma}_t'\rangle_{\hat{\gamma}_t}}dt = C_z \cdot d(f^{j_z}x,z)=C_z \cdot r,\]
which proves the lemma.
\end{proof}

\begin{definition}[Two-sided cylinders in $\Lambda$]\ \par
Since $f^R: \Lambda \to \Lambda $ is bijective, we can define two-sided cylinders as
\[\xi_{i_{-n}\cdots i_{0} \cdots i_{n}}:=(f^R)^n\Lambda_{i_{-n}} \bigcap (f^R)^{n-1}\Lambda_{i_{-(n-1)}} \bigcap \cdots \bigcap \Lambda_{i_0} \bigcap \cdots \bigcap (f^R)^{-(n-1)}\Lambda_{i_{n-1}} \bigcap (f^R)^{-n}\Lambda_{i_{n}}. \]

Let us introduce now a new partition of $\Lambda$ as
\[\mathcal{M}_0:=\{\Lambda_i, i \ge 1\}, \mathcal{M}_k:= \bigvee_{0 \le i \le k} (f^R)^{-i} \mathcal{M}_0,\]

and a quotient partition of $\widetilde{\Lambda}$
\[\widetilde{\mathcal{M}}_0:=\{\widetilde{\Lambda}_i, i \ge 1\}, \widetilde{\mathcal{M}}_k:= \bigvee_{0 \le i \le k} (\widetilde{f^R})^{-i} \widetilde{\mathcal{M}}_0,\]
\[\widetilde{\mathcal{M}}_{\infty}:= \bigvee^{\infty}_{i=0} (\widetilde{f^R})^{-i} \widetilde{\mathcal{M}}_0 \text{ is the } \sigma \text{-algebra of }\widetilde{\Lambda}.\]

\end{definition}

\begin{lemma}[Diameters of two-sided cylinders]\ \label{diamcylinder}\ \par
 \[\diam \xi_{i_{-n}\cdots i_{0} \cdots i_{n}} \le 2C \cdot \beta^n,\]
 where  $C\ge 1$ and $\beta\in (0,1)$ are the same as in the Assumption \ref{assumption}.
\end{lemma}
\begin{proof}
Observe that $\xi_{i_{-n}\cdots i_{0} \cdots i_{n}}=\xi_{i_{-n}\cdots i_{0}} \bigcap \xi_{ i_{0} \cdots i_{n}}$ is an intersection of two one-sided cylinders $\xi_{i_{-n}\cdots i_{0}}$ and $\xi_{ i_{0} \cdots i_{n}}$, where $\xi_{i_{-n}\cdots i_{0}}$ is long and $\xi_{ i_{0} \cdots i_{n}}$ is slim. We will now estimate their widths. By the Assumption \ref{assumption} $(f^R)^{-1}$ contracts exponentially along any $\gamma^u \in \Gamma^u$. Hence $\diam \gamma^u \bigcap \xi_{ i_{0} \cdots i_{n}} \le C \cdot \beta^n$. Similarly, since
\[\xi_{ i_{-n} \cdots i_{0}}=(f^R)^n\Lambda_{i_{-n}} \bigcap (f^R)^{n-1}\Lambda_{i_{-(n-1)}} \bigcap \cdots \bigcap \Lambda_{i_0}\]
\[=(f^R)^n[\Lambda_{i_{-n}} \bigcap (f^R)^{-1}\Lambda_{i_{-(n-1)}} \bigcap \cdots \bigcap (f^R)^{-n}\Lambda_{i_0}]\]
and $f^R$ contracts exponentially along any $\gamma^s \in \Gamma^s$, we have $\diam \gamma^s \bigcap \xi_{ i_{-n} \cdots i_{0}} \le C \cdot \beta^n$.

The properties of the Gibbs-Markov-Young structures ensure that  $\forall x,y \in \xi_{i_{-n}\cdots i_{0} \cdots i_{n}}$ there is $z \in \gamma^u(x)$, $z\in \gamma^s(y)$, such that
\[d(x,z) \le C \cdot \beta^n, \text{ } d(y,z) \le C \cdot \beta^n.\]
Hence $d(x,y) \le 2C\cdot \beta^n$, i.e. $\diam \xi_{i_{-n}\cdots i_{0} \cdots i_{n}} \le 2C \cdot \beta^n$.
\end{proof}

\begin{lemma}[Decay of correlations for $f^R: \Lambda \to \Lambda$]\label{decorrelationcylinder}\ \par
There exist constants $C''>1, \beta_1 \in (0,1)$, such that for any one-sided cylinder $\xi_{ i_{0} \cdots i_{n}} \in \mathcal{M}_n $ and any $A \in \sigma(\cup_{k\ge 0}\mathcal{M}_k)$
\begin{equation}
 |\int 1_{\xi_{ i_{0} \cdots i_{n}}} \cdot 1_A \circ (f^R)^{2n}d\mu_{\Lambda}-\int 1_{\xi_{ i_{0} \cdots i_{n}}} d\mu_{\Lambda} \cdot \int 1_A d\mu_{\Lambda}| \le C'' \cdot \beta_1^n \cdot \mu_{\Lambda}(\xi_{ i_{0} \cdots i_{n}}).
\end{equation}

\end{lemma}

\begin{proof}
At first, we prove that
\[|\int 1_{\widetilde{\xi_{ i_{0} \cdots i_{n}}}} \cdot 1_{\widetilde{A}} \circ (\widetilde{f^R})^{2n}d\mu_{\widetilde{\Lambda}}-\int 1_{\widetilde{\xi_{ i_{0} \cdots i_{n}}}} d\mu_{\widetilde{\Lambda}} \cdot \int 1_{\widetilde{A}} d\mu_{\widetilde{\Lambda}}| \le C'' \cdot \beta_1^n \cdot \mu_{\widetilde{\Lambda}}(\widetilde{\xi_{ i_{0} \cdots i_{n}}}),\]

where $\widetilde{\xi_{ i_{0} \cdots i_{n}}} \in \widetilde{\mathcal{M}_n},  \widetilde{A} \in \sigma(\cup_{k\ge 0}\widetilde{\mathcal{M}}_k)$.

Denote the transfer operator of $\widetilde{f^R}$ by $\widetilde{P}$.

{\bf Claim:}  
\[||\widetilde{P}^{n+1} 1_{\widetilde{\xi_{ i_{0} \cdots i_{n}}}}||_{\infty} \precsim \mu_{\widetilde{\Lambda}}(\widetilde{\xi_{ i_{0} \cdots i_{n}}}),\]
where a constant in $\precsim$ depends on the constant $C$ in the Definition \ref{gibbs}.

Indeed, $\forall \widetilde{x} \in \widetilde{\Lambda}$ we have $(\widetilde{f^R})^{n+1}(\xi_{ i_{0} \cdots i_{n}})=\widetilde{\Lambda}$, and 

\[\widetilde{P}^{n+1} 1_{\widetilde{\xi_{ i_{0} \cdots i_{n}}}}(\widetilde{x})=\sum_{(\widetilde{f^R})^{n+1}(\widetilde{y})=\widetilde{x}} \frac{1_{\widetilde{\xi_{ i_{0} \cdots i_{n}}}}(\widetilde{y})}{|\det D(\widetilde{f^R})^{n+1}(\widetilde{y})|}=\frac{1}{|\det D(\widetilde{f^R})^{n+1}[(\widetilde{f^R})^{n+1}|_{\widetilde{\xi_{ i_{0} \cdots i_{n}}}}^{-1}\widetilde{x}]|}.\]

By making use of the distortion estimate of $\widetilde{f^R}$ (see the Lemma 3.2 in \cite{Alves} or the Lemma 1 in \cite{Y}) and (\ref{horsesrbleb}) we get 
\[\widetilde{P}^{n+1} 1_{\widetilde{\xi_{ i_{0} \cdots i_{n}}}}(\widetilde{x})=\frac{1}{|\det D(\widetilde{f^R})^{n+1}[(\widetilde{f^R})^{n+1}|_{\widetilde{\xi_{ i_{0} \cdots i_{n}}}}^{-1}\widetilde{x}]|} \precsim \mu_{\widetilde{\Lambda}}(\widetilde{\xi_{ i_{0} \cdots i_{n}}}).\]

Therefore $||\widetilde{P}^{n+1} 1_{\widetilde{\xi_{ i_{0} \cdots i_{n}}}}||_{\infty} \precsim \mu_{\widetilde{\Lambda}}(\widetilde{\xi_{ i_{0} \cdots i_{n}}})$.

{\bf Claim:}
\[|\widetilde{P}^{n+1} 1_{\widetilde{\xi_{ i_{0} \cdots i_{n}}}}(\widetilde{x}_1)-\widetilde{P}^{n+1} 1_{\widetilde{\xi_{ i_{0} \cdots i_{n}}}}(\widetilde{x}_2)| \precsim \cdot \beta^{\widetilde{s}(\widetilde{x}_1,\widetilde{x}_2)} \cdot \mu_{\widetilde{\Lambda}}(\widetilde{\xi_{ i_{0} \cdots i_{n}}}),\]
where the constant in $\precsim$ depends on the constant $C$ in the Definition \ref{gibbs}, $\beta \in (0,1)$ is the same as in the Definition \ref{gibbs}, $\widetilde{x}_1, \widetilde{x_2} \in \widetilde{\Lambda}_i$ for some $\widetilde{\Lambda}_i \in \widetilde{\mathcal{M}}_0$ and

\[\widetilde{s}(\widetilde{x}_1,\widetilde{x}_2):=\inf \{n: (\widetilde{f^R})^n(\widetilde{x}_1), (\widetilde{f^R})^n(\widetilde{x}_2) \text{ lie in different elements of } \widetilde{\mathcal{M}}_0\}.\]

Again, by using the distortion of $\widetilde{f^R}$, we get 
\[|\widetilde{P}^{n+1} 1_{\widetilde{\xi_{ i_{0} \cdots i_{n}}}}(\widetilde{x}_1)-\widetilde{P}^{n+1} 1_{\widetilde{\xi_{ i_{0} \cdots i_{n}}}}(\widetilde{x}_2)|\]
\[=|\frac{1}{|\det D(\widetilde{f^R})^{n+1}[(\widetilde{f^R})^{n+1}|_{\widetilde{\xi_{ i_{0} \cdots i_{n}}}}^{-1}\widetilde{x}_1]|}-\frac{1}{|\det D(\widetilde{f^R})^{n+1}[(\widetilde{f^R})^{n+1}|_{\widetilde{\xi_{ i_{0} \cdots i_{n}}}}^{-1}\widetilde{x}_2]|}|\]
\[=\frac{1}{|\det D(\widetilde{f^R})^{n+1}[(\widetilde{f^R})^{n+1}|_{\widetilde{\xi_{ i_{0} \cdots i_{n}}}}^{-1}\widetilde{x}_1]|} \cdot |1-\frac{|\det D(\widetilde{f^R})^{n+1}[(\widetilde{f^R})^{n+1})^{n+1}|_{\widetilde{\xi_{ i_{0} \cdots i_{n}}}}^{-1}\widetilde{x}_1]|}{|\det D(\widetilde{f^R})^{n+1}[(\widetilde{f^R})^{n+1}|_{\widetilde{\xi_{ i_{0} \cdots i_{n}}}}^{-1}\widetilde{x}_2]|}|\]
\[\precsim  \beta^{\widetilde{s}(\widetilde{x}_1,\widetilde{x}_2)} \cdot \mu_{\widetilde{\Lambda}}(\widetilde{\xi_{ i_{0} \cdots i_{n}}}).\]

So the claim holds. 

\textbf{Claim}: $\widetilde{f^R}: (\widetilde{\Lambda}, \mu_{\widetilde{\Lambda}}) \to (\widetilde{\Lambda}, \mu_{\widetilde{\Lambda}})$ is exact, so mixing.

Take a set $A\in \bigcap_{i \ge 0}(\widetilde{f^R})^{-i}\widetilde{\mathcal{M}}_{\infty}$. So $\mu_{\widetilde{\Lambda}}((\widetilde{f^R})^{i}A)=\mu_{\widetilde{\Lambda}}(A), \forall i \ge 0$. If $\mu_{\widetilde{\Lambda}}(A)>0$, then there is $B_k \in \widetilde{\mathcal{M}}_k$ for a large $k \gg 1$ s.t. $(\widetilde{f^R})^k: B_k \to \widetilde{\Lambda}$ is bijective and

\[\frac{\mu_{\widetilde{\Lambda}}(B_k \bigcap A) }{\mu_{\widetilde{\Lambda}}(B_k)}\approx 1.\]

By the distortion of $\widetilde{f^R}$, we have
\[\mu_{\widetilde{\Lambda}}(A)=\frac{\mu_{\widetilde{\Lambda}}\{(\widetilde{f^R})^k(A)\}}{\mu_{\widetilde{\Lambda}}(\widetilde{\Lambda})}\ge \frac{\mu_{\widetilde{\Lambda}}\{(\widetilde{f^R})^k(B_k \bigcap A)\} }{\mu_{\widetilde{\Lambda}}\{\widetilde{f^R}(B_k)\}}\approx\frac{\mu_{\widetilde{\Lambda}}(B_k \bigcap A) }{\mu_{\widetilde{\Lambda}}(B_k)}\approx 1.\]

So this claim holds.

From the these three claims it follows that $\widetilde{P}^{n+1} 1_{\widetilde{\xi_{ i_{0} \cdots i_{n}}}}$ is a locally Lipschitz function on $\widetilde{\Lambda}$ with a Lipschitz constant $\mu_{\widetilde{\Lambda}}(\widetilde{\xi_{ i_{0} \cdots i_{n}}})$ and $\widetilde{f^R}: (\widetilde{\Lambda}, \mu_{\widetilde{\Lambda}}) \to (\widetilde{\Lambda}, \mu_{\widetilde{\Lambda}})$ is mixing. By applying the Corollary 2.3 (b) of \cite{asipmelbourne} to $\widetilde{P}^{n+1} 1_{\widetilde{\xi_{ i_{0} \cdots i_{n}}}}$ we obtain that there exist constants $C''>1, \beta_1 \in (0,1)$, such that  
\[|\int (\widetilde{P}^{n+1} 1_{\widetilde{\xi_{ i_{0} \cdots i_{n}}}} )\cdot 1_{\widetilde{A}} \circ (\widetilde{f^R})^{n}d\mu_{\widetilde{\Lambda}}-\int \widetilde{P}^{n+1} 1_{\widetilde{\xi_{ i_{0} \cdots i_{n}}}}  d\mu_{\widetilde{\Lambda}} \cdot \int 1_{\widetilde{A}} d\mu_{\widetilde{\Lambda}}| \le C'' \cdot \beta_1^n \cdot \mu_{\widetilde{\Lambda}}(\widetilde{\xi_{ i_{0} \cdots i_{n}}}).\]

Therefore
\[|\int  1_{\widetilde{\xi_{ i_{0} \cdots i_{n}}}} \cdot 1_{\widetilde{A}} \circ (\widetilde{f^R})^{2n}d\mu_{\widetilde{\Lambda}}-\int  1_{\widetilde{\xi_{ i_{0} \cdots i_{n}}}}  d\mu_{\widetilde{\Lambda}} \cdot \int 1_{\widetilde{A}} d\mu_{\widetilde{\Lambda}}| \le C'' \cdot \beta_1^n \cdot \mu_{\widetilde{\Lambda}}(\widetilde{\xi_{ i_{0} \cdots i_{n}}}).\]

Now, by (\ref{allquotientmeasure}) $(\widetilde{\pi}_{\Lambda})_{*}\mu_{\Lambda}=\mu_{\widetilde{\Lambda}}$, we have 
\[|\int 1_{\xi_{ i_{0} \cdots i_{n}}} \cdot 1_A \circ (f^R)^{2n}d\mu_{\Lambda}-\int 1_{\xi_{ i_{0} \cdots i_{n}}} d\mu_{\Lambda} \cdot \int 1_A d\mu_{\Lambda}| \le C'' \cdot \beta_1^n \cdot \mu_{\Lambda}(\xi_{ i_{0} \cdots i_{n}}).\]

\end{proof}

\subsection{Short Returns}

Recall that  $n:=\lfloor \frac{T}{\mu(B_r(z))}\rfloor, p:=\lfloor n^{\frac{\text{dim}_H\mu-\epsilon}{\text{dim}_H\mu}}\rfloor$ (see the Proposition \ref{rate}). We now consider short returns to $\Lambda$. A reason to do this is that the short returns problem on $\mathcal{M}$ can be turned into the short returns problem on  $\Lambda$ (see the Lemma \ref{pullbackhorseshoe}).
For any $z' \in \interior{(\Lambda)}$, a fixed positive integer $M>0$, sufficiently small constants $\epsilon'>0, r>0$, such that $ n^{\epsilon'}\ll p$, $B_{M\cdot r} (z') \subseteq \interior{(\Lambda)}$ almost surely and 
\begin{equation}\label{sumhorshose}
    \int_{B_{M\cdot r} (z')} 1_{\bigcup_{1\le k \le p}(f^R)^{-k}B_{M\cdot r} (z')}d\mu_{\Lambda}=\int_{B_{M\cdot r} (z')} 1_{\bigcup_{1\le k \le N}(f^R)^{-k}B_{ M\cdot r} (z')}d\mu_{\Lambda}
\end{equation}

\[+\int_{B_{M\cdot r} (z')} 1_{\bigcup_{N\le k \le n^{\epsilon'}}(f^R)^{-k}B_{ M\cdot r} (z')}d\mu_{\Lambda}+\int_{B_{ M\cdot r} (z')} 1_{\bigcup_{n^{\epsilon'}\le k \le p}(f^R)^{-k}B_{ M\cdot r} (z')}d\mu_{\Lambda},\]
where $N= \lfloor\frac{-\log 2C }{\log \beta}\rfloor+1$, and the constants $C$ and $\beta$ are defined in the Assumption \ref{assumption}. 

We begin with the short fixed-length returns described by the integral

\[\int_{B_{M\cdot r} (z')} 1_{\bigcup_{1\le k \le N}(f^R)^{-k}B_{ M\cdot r} (z')}d\mu_{\Lambda}.\]

\begin{lemma}[Short fixed-length returns]\label{fixreturn}\ \par
For almost every $z' \in \interior{(\Lambda)}$ and sufficiently small $r_{N,M, z'}>0$ we have for any $ r < r_{N,M,z'}$ \[\int_{B_{M\cdot r} (z')} 1_{\bigcup_{1\le k \le N}(f^R)^{-k}B_{ M\cdot r} (z')}d\mu_{\Lambda}=0.\]

(Actually, a stronger result will be proved, i.e.  $\forall k\in \mathbb{N}$ a map $(f^{{R}})^k$ is a local diffeomorphism at almost every point $z' \in \interior{(\Lambda)}$).
\end{lemma}
\begin{proof}
From the Assumption \ref{assumption} $\mu(\partial \Lambda)=0$, and then $\mu(\bigcup_{i \in \mathbb{Z}} f^{-i}\partial \Lambda)=0$. By making use of the decay of correlations for the quotient map $\widetilde{f^R}: \widetilde{\Lambda} \to \widetilde{\Lambda}$, we have that $(\widetilde{f^R})^i$ are ergodic for all $i \ge 1$. Therefore the set of  periodic points $A_{per}$ of $\widetilde{f^R}$ has $\mu_{\widetilde{\Lambda}}(A_{per})=0$, and $\mu_{\Lambda}(\widetilde{\pi}_{\Lambda}^{-1}(A_{per}))=0$ due to (\ref{allquotientmeasure}).

Choose now $z' \in \widetilde{\pi}_{\Lambda}^{-1}(A_{per}^c)\bigcap \interior{(\Lambda)} \bigcap [\bigcup_{i \in \mathbb{Z}} f^{-i}\partial \Lambda]^c$. Then there is $r_M>0$ s.t. for any $r<r_M$ almost surely $B_{M \cdot r} \subseteq \interior{(\Lambda)}$. We will make now several claims.

\textbf{Claim:} For any $k\in \mathbb{N}$ let $R^k$ be the k-th return time. Then ${R}^k|_{B_{M \cdot r}(z')}={R}^k(z'), \forall k\in[1, N]$ if $r< r'_{N,M, z'}$ for a small enough $ r'_{N,M, z'}>0$.

From the choice of $z'$ we have 
\[f^{m}(z') \notin \Lambda, \forall m \in [{R}^k(z')+1, {R}^{k+1}(z')-1], \forall k\in [0,N-1]\]
and 
\[f^{{R}^k(z')}(z')\in \interior{(\Lambda)}, \forall k\in [0,N].\]

Due to the Assumption \ref{geoassumption}, there is $r_{N,M,z'}>0$ such that, if $r<r'_{N,M,z'}$, $B_{M \cdot r}(z') \subseteq \interior{(\Lambda)}$ almost surely, then 
\[f^{m}(B_{M \cdot r}(z')) \subseteq \Lambda^c, \forall m \in [{R}^k(z')+1, {R}^{k+1}(z')-1], \forall k\in [0,N-1]\]
and 
\[f^{{R}^k(z')}(B_{M \cdot r}(z'))\subseteq \Lambda \text{ almost surely}, \forall k\in [0,N].\]

Since ${R}, {R}^2, \cdots, {R}^N$ are consecutive return times to $\Lambda$, we have ${R}^k|_{B_{M \cdot r}(z')}={R}^k(z'), \forall k\in[1, N]$. Thus this claim holds.

\textbf{Claim:} $f^{{R}^k}$ for all $k\in[1, N]$ is a local diffeomorphism at $z'$.

This claim holds because $f$ is a local diffeomorphism on $\bigcup_{i \ge 1} \bigcup_{0 \le j < R_i} f^j(\Lambda_i)$  and ${R}^k|_{B_{M \cdot r}(z')}={R}^k(z')$ for sufficiently small $r>0$. 

These two claims, together with the fact that  $f^{{R}^k}(z')\in \interior{(\Lambda)}$ are distinct  $\forall k\in[0, N],$  imply that there exists a small enough $r_{z',M,N}>0$, such that $\forall r<r_{z',M,N}$ the sets $f^{{R}^k}(B_{M \cdot r}(z'))$ are disjoint for all $k\in [0,N]$. So the lemma holds. 
\end{proof}

Before estimating the super-short returns $\int_{B_{M\cdot r} (z')} 1_{\bigcup_{N\le k \le n^{\epsilon'}}(f^R)^{-k}B_{ M\cdot r} (z')}d\mu_{\Lambda}$, we need one more lemma.

\begin{lemma}[Recurrences]\label{recurrence}
There exists $r_M>0$, such that $ \forall r < r_M$ and $\forall \gamma^u \in \Gamma^u, i\ge N$ the following inequality holds
\[\Leb_{\gamma^u}\{z'\in \Lambda \bigcap \gamma^u: d((f^{{R}})^{-i}z',z') \le M \cdot r\}\precsim_{\dim \gamma^u} (M \cdot r)^{\dim \gamma^u},\]
where a constant in $\precsim_{\dim \gamma^u}$ depends upon $\dim \gamma^u$, but it does not depend upon $i\ge {N}$ and $\gamma^u \in \Gamma^u$.

\end{lemma}
\begin{proof}
We start with making

\textbf{Claim}: there are finitely many balls $\{B_{r'_i}(z'_i): 1 \le i \le N'\}$, where $N'\in \mathbb{N}$ depends only on $\Lambda$, such that all unstable fibers $\gamma^u \in\Gamma^u$ are almost flat in each $B_{r_i'}(z'_i)$.

In view of the Definition \ref{gibbs} of $\Gamma^s$ (respectively, of $\Gamma^u$), for any $z'' \in \Lambda$ there exists a small open ball $B_{r''}(z'')$, such that all $\gamma^s \in \Gamma^s$ (respectively, $\gamma^u \in \Gamma^u)$ intersecting $B_{r''}(z'')$ are almost flat and parallel. Since $\Lambda$ is compact, one can find finitely many open balls $\{B_{{{r_1''}}}(z''_1), \cdots, B_{{r''_{N'}}}(z''_{N'})\}$, which cover $\Lambda$. Hence, the claim holds. 

Take now any of these balls, say $B_{r''_1}(z''_1)$, and any $\gamma^u \in \Gamma^u, z_1',z_2' \in \{z'\in \Lambda \bigcap \gamma^u\bigcap B_{r_1''}(z_1''): d((f^{{R}})^{-i}z',z')\le M \cdot r\} $. Then for any $r<\tau_M:=\min\{\frac{r_1''}{8M}, \cdots, \frac{r''_{N'}}{8M}\}$,
\[d((f^{{R}})^{-i}z'_1,z_1') \le M \cdot r,\text{ } d((f^{{R}})^{-i}z'_2,z_2') \le M \cdot r.\]

By making use of the Assumption \ref{assumption} together with $i \ge N$, we get that  $C \cdot \beta^i \le C \cdot \beta^{N}< \frac{1}{2}$ and 
\[d(z'_1,z'_2)\le d(z'_1, (f^{{R}})^{-i}z'_1)+d((f^{{R}})^{-i}z'_1, (f^{{R}})^{-i}z'_2)+d((f^{{R}})^{-i}z'_2, z'_2)\]
\[\le 2M \cdot r+ C \cdot \beta^i \cdot d(z'_1,z'_2) \le 2M \cdot r +\frac{1}{2}  d(z'_1,z'_2).\]

Thus $d(z'_1,z'_2) \le 4M \cdot r < \frac{r_1''}{2} \implies \diam \{z'\in \Lambda \bigcap B_{r_1''}(z_1'') \bigcap \gamma^u: d((f^{{R}})^{-i}z',z') \le M \cdot r\} \le 4M \cdot r$.

Since $\gamma^u$ in the ball $B_{r_1''}(z_1'')$ is almost flat, then its Lebesgue measure can be estimated by diameters, i.e. 
\[\Leb_{\gamma^u}\{z'\in \Lambda \bigcap B_{r''_1}(z''_1): d((f^{{R}})^{-i}z',z') \le M \cdot r\} \precsim_{\dim \gamma^u} (M \cdot r)^{\dim \gamma^u}.\]

This estimate also holds for the balls $B_{r_2''}(z_2''), \cdots, B_{r_{N'}''}(z_{N'}'')$. By summing over all the balls and noting that $\Lambda \subseteq \bigcup_{1\le i\le {N'}} B_{r_i''}(z_i'')$, we get 
\[\Leb_{\gamma^u}\{z'\in \Lambda: d((f^{{R}})^{-i}z',z') \le M \cdot r\} \precsim_{\dim \gamma^u} (M \cdot r)^{\dim \gamma^u},\]
where a constant in $\precsim_{\dim \gamma^u}$ depends upon $\dim \gamma^u$, but it does not depend upon $i\ge {N}$ and $\gamma^u \in \Gamma^u$.

\end{proof}

\begin{lemma}[Super-short returns]\label{veryshortreturn}\ \par
Choose $n= \lfloor \frac{T}{\mu(B_r(z))}\rfloor$.  Then for almost every $z\in \mathcal{M}, z'\in \Lambda$ there exists $r_{z,z',M}>0$, such that $\forall r< r_{z,z',M}$
\[\int_{B_{M\cdot r} (z')} 1_{\bigcup_{N\le k \le n^{\epsilon'}}(f^{{R}})^{-k}B_{ M\cdot r} (z')}d\mu_{\Lambda}\precsim_{T,\epsilon} (M \cdot r)^{\dim_H\mu -\epsilon} \cdot M^{\frac{\min\{{\dim \gamma^u}, \dim_H \mu\}}{6}}\cdot r^{\frac{\min\{{\dim \gamma^u}, \dim_H \mu\}}{12}},\]
where  $\epsilon>0$ is the same as in the Proposition \ref{rate} and $\epsilon'<\frac{\min\{\dim_H \mu, { \dim \gamma^u}\}}{12 \dim_H \mu+12 \epsilon}$.

\end{lemma}
\begin{proof}
It follows from (\ref{horsesrbleb}), $(f^R)_{*}\mu_{\Lambda}=\mu_{\Lambda}$ and Lemma \ref{recurrence} that

\[\mu_{\Lambda}\{z'\in \Lambda: d((f^{{R}})^{k}z',z')\le  M\cdot r\}=\mu_{\Lambda}\{z'\in \Lambda: d((f^{{R}})^{-k}z',z')\le  M\cdot r\}\]
\[= \int \mu_{\gamma^u}\{z'\in \Lambda: d((f^{{R}})^{-k}z',z')\le  M\cdot r\}d\mu_{\Lambda}\]
\[\precsim \int \Leb_{\gamma^u}\{z'\in \Lambda: d((f^{{R}})^{-k}z',z')\le  M\cdot r\}d\mu_{\Lambda} \precsim_{\dim \gamma^u} (M \cdot r)^{\dim \gamma^u}.\]

By the Assumption \ref{geoassumption}, for  $\delta=\frac{\min\{{\dim \gamma^u}, \dim_H \mu\}}{6}$ and almost every $z'\in \Lambda$ there exists $r_{z',\delta}>0$, such that $r^{\dim_H\mu +\delta} \le \mu (B_r(z'))\le r^{\dim_H\mu-\delta}, \forall r< r_{z',\delta}$. Let $A_m:=\{z' \in \Lambda: r_{z',\delta}>\frac{1}{m}\}$. Then $\bigcup_{m} A_m=\Lambda$ and $\forall z'\in A_m$, $\forall r< \frac{1}{m}$
\begin{equation}\label{11'}
    r^{\dim_H\mu +\delta} \le \mu (B_r(z'))\le r^{\dim_H\mu-\delta} 
  \end{equation}  
and 
\[\int_{B_{M\cdot r} (z')} 1_{\bigcup_{N\le k \le n^{\epsilon'}}(f^{{R}})^{-k}B_{ M\cdot r} (z')}d\mu_{\Lambda} \le \int_{B_{M\cdot r} (z')} 1_{\bigcup_{N\le k\le n^{\epsilon'}}d((f^{{R}})^{k}y,y)\le  2M\cdot r}d\mu_{\Lambda}(y).\]

 Let a kernel on $\Lambda\times A_m$ be $K(y,z'):=1_{B_{M \cdot r}(z')}(y)$. If $r<\frac{1}{3m \cdot M}$ then by (\ref{11'})
 
 \[\int K(y,z') d\mu_{\Lambda}(y)=\mu_{\Lambda}(B_{M \cdot r}(z'))\precsim (M \cdot r)^{\dim_H \mu-\delta}.\]

In order to estimate $\int K(y,z')\cdot 1_{A_m}(z')d\mu_{\Lambda}(z')=\mu_{\Lambda}(A_m \bigcap B_{M \cdot r}(y))$, observe that if $z'' \in A_m \bigcap B_{M \cdot r}(y)\neq \emptyset$, then $B_{M \cdot r}(y)\subseteq B_{3M \cdot r}(z'')$. By (\ref{11'}) again, $\int K(y,z')\cdot 1_{A_m}(z')d\mu_{\Lambda}(z')=\mu_{\Lambda}(A_m \bigcap B_{M \cdot r}(y)) \le \mu_{\Lambda}(B_{3M \cdot r}(z'')) \precsim (3M \cdot r)^{\dim_H \mu-\delta}$.

Having the estimates of $K(y,z')$ above and Lemma \ref{recurrence}, we can use the Schur's test (see the Theorem 5.6 in \cite{tao}), i.e. for all $r<\min\{\frac{1}{3m \cdot M}, r_M\}$
\[\int 1_{A_m}(z')d\mu_{\Lambda}(z')\int_{B_{M\cdot r} (z')} 1_{d((f^{{R}})^{k}y,y)\le  2M\cdot r}d\mu_{\Lambda}(y) \precsim_{\dim \gamma^u} (3M\cdot r)^{\dim_H \mu-\delta} \cdot (M \cdot r)^{{ \dim \gamma^u}},\]
where a constant in $\precsim_{\dim \gamma^u}$ does not depend upon $M,r,m,k$.

Since $r^{\dim_H\mu +\delta} \le \mu (B_r(z'))$, one gets 
\[\int_{A_m}\frac{\int_{B_{M\cdot r} (z')} 1_{\bigcup_{N \le k \le [2^{\frac{4(\dim_H\mu+\epsilon)}{\min\{\dim_H\mu, \dim \gamma^u\}}}\cdot \frac{T}{r^{\dim_H\mu +\epsilon}}]^{\epsilon'}}d((f^{{R}})^{k}y,y)\le  2M\cdot r}d\mu_{\Lambda}(y) }{(\frac{T}{r^{\dim_H \mu+\epsilon}})^{\epsilon'}\cdot \mu (B_{M\cdot r}(z')) \cdot (M \cdot r)^{\delta}}d\mu_{\Lambda}(z')\]
\[\precsim_{\dim_H \mu, \dim \gamma^u, \epsilon, \epsilon'} \frac{(M\cdot r)^{\dim_H \mu-\delta+\dim \gamma^u}}{(M\cdot r)^{\dim_H \mu+2\delta}} \le(M \cdot r)^{\frac{\min\{\dim_H\mu,{\dim \gamma^u}\}}{2}}.\] 

Choose now $r_i=i^{-\frac{4}{\min\{\dim_H\mu,\dim \gamma^u\}}}<\min\{\frac{1}{3m\cdot M},r_M\}$. Then, by the Borel-Cantelli Lemma, for almost every $z'\in A_m$ there exists $N_{M, m, z'}>\min\{\frac{1}{3m\cdot M}, r_M\}^{-\frac{\min\{\dim_H\mu,\dim \gamma^u\}}{4}}$, such that $\forall i> N_{M,m,z'}$ 
\begin{equation}\label{21}
    \int_{B_{M\cdot r_i} (z')} 1_{\bigcup_{N \le k \le [2^{\frac{4(\dim_H\mu+\epsilon)}{\min\{\dim_H\mu, \dim \gamma^u\}}}\cdot \frac{T}{r_i^{\dim_H\mu +\epsilon}}]^{\epsilon'}}d((f^{{R}})^{k}y,y)\le  2M\cdot r_i}d\mu_{\Lambda}(y)\le (\frac{T}{r_i^{\dim_H \mu+\epsilon}})^{\epsilon'}\cdot  \mu (B_{M\cdot r_i}(z')) \cdot (M \cdot r_i)^{\delta}.
\end{equation}

Hence for almost every $z'\in \Lambda$ (in particular, for $z'\in \interior{(\Lambda)}$) there is $m_{z'}>0$, such that $z' \in A_{m_{z'}}$. Let $r_{z',M}:=\min\{r_M, N_{M,m_{z'},z'}^{-\frac{4}{\min\{\dim_H\mu, \dim \gamma^u\}}}\}$. By the Assumption \ref{geoassumption}, for $\epsilon$ from the Proposition \ref{rate} there exists such $r_{z,z',M}\in (0,r_{z',M})$ that $\forall r<r_{z,z',M}$ 
\[n\le \frac{T}{\mu(B_r(z))} \le \frac{T}{r^{\dim_H \mu +\epsilon}} \text{ and } \mu(B_{M \cdot r}(z'))\le (M \cdot r)^{\dim_H\mu -\epsilon}.\] 

Then $\forall r< r_{z,z',M}$, $\exists i >0$, such that $r_{i+1}\le r \le r_i$, and the following estimates hold 
\[\int_{B_{M\cdot r} (z')} 1_{\bigcup_{N\le k \le n^{\epsilon'}}(f^{{R}})^{-k}B_{ M\cdot r} (z')}d\mu_{\Lambda} \le \int_{B_{M\cdot r_i} (z')} 1_{\bigcup_{N\le k\le (\frac{T}{r^{\dim_H \mu +\epsilon}})^{\epsilon'}}d((f^{{R}})^{k}y,y)\le  2M\cdot r_i}d\mu_{\Lambda}(y)\]
\[\le \int_{B_{M\cdot r_i} (z')} 1_{\bigcup_{N\le k\le (\frac{T}{r_{i+1}^{\dim_H \mu +\epsilon}})^{\epsilon'}}d((f^{{R}})^{k}y,y)\le 2M \cdot r_i}d\mu_{\Lambda}(y)\]
\[\le \int_{B_{M\cdot r_i} (z')} 1_{\bigcup_{N\le k\le [(\frac{r_i}{r_{i+1}})^{\dim_H \mu +\epsilon}\cdot  (\frac{T}{r_{i}^{\dim_H \mu +\epsilon}})]^{\epsilon'}}d((f^{{R}})^{k}y,y)\le 2M \cdot r_i}d\mu_{\Lambda}(y)\]
\begin{equation}\label{12'}
    \le \int_{B_{M\cdot r_i} (z')} 1_{\bigcup_{N \le k \le (2^{\frac{4(\dim_H\mu+\epsilon)}{\min\{\dim_H\mu, \dim \gamma^u\}}}\cdot \frac{T}{r_i^{\dim_H\mu +\epsilon}})^{\epsilon'}}d((f^{{R}})^{k}y,y)\le  2M\cdot r_i}d\mu_{\Lambda}(y).
\end{equation}

Hence, if $\epsilon'< \frac{\min\{\dim \gamma^u, \dim_H \mu\}}{12\dim_H \mu+12\epsilon} $, then we can use (\ref{21}) to continue the estimate (\ref{12'}). Namely,

\[\int_{B_{M\cdot r} (z')} 1_{\bigcup_{N\le k \le n^{\epsilon'}}(f^{{R}})^{-k}B_{ M\cdot r} (z')}d\mu_{\Lambda}\le (M \cdot r_i)^{\dim_H\mu -\epsilon} \cdot (M \cdot r_i)^{\frac{\min\{\dim \gamma^u, \dim_H \mu\}}{6}}\cdot (\frac{T}{r_i^{\dim_H \mu+\epsilon}})^{\epsilon'}\]
\[\le (M \cdot r \cdot  \frac{r_i}{r_{i+1}})^{\dim_H\mu -\epsilon} \cdot (M \cdot  r \cdot  \frac{r_i}{r_{i+1}})^{\frac{\min\{\dim \gamma^u, \dim_H \mu\}}{6}}\cdot (\frac{T}{r^{\dim_H \mu+\epsilon}})^{\epsilon'}\]
\[\precsim_{T,\epsilon} (M \cdot r)^{\dim_H\mu -\epsilon} \cdot M^{\frac{\min\{\dim \gamma^u, \dim_H \mu\}}{6}}\cdot r^{\frac{\min\{\dim \gamma^u, \dim_H \mu\}}{12}}.\]

The last inequality holds because $\frac{r_i}{r_{i+1}} \precsim 1$.

\end{proof}
\begin{lemma}[Not super-short, but short returns]\label{notveryshort}\ \par
Let $n= \lfloor \frac{T}{\mu(B_r(z))}\rfloor,p=\lfloor {\lfloor \frac{T}{\mu(B_r(z))}\rfloor}^{\frac{\text{dim}_H\mu-\epsilon}{\text{dim}_H\mu}}\rfloor$ and $n^{\epsilon'}\ll p$. Then, for almost all $z\in \mathcal{M}$, $z'\in \interior{(\Lambda)}$, there exists $r_{z,z',M}>0$, such that $\forall r< r_{z,z',M}$
\[\int_{B_{ M\cdot r} (z')} 1_{\bigcup_{n^{\epsilon'}\le k \le p}(f^R)^{-k}B_{ M\cdot r} (z')}d\mu_{\Lambda}\]
\[\precsim_{T,\epsilon} (M\cdot r)^{\dim_H\mu-\epsilon^3} \cdot {\beta_2}^{r^{-(\dim_H\mu-\epsilon)\epsilon'}}+\mu_{\Lambda}[ B_{M \cdot r + 2C \cdot \beta_2^{r^{-(\dim_H\mu-\epsilon)\epsilon'}}}(z') \setminus B_{M \cdot r}(z')]\cdot \beta_2^{r^{-(\dim_H\mu-\epsilon)\epsilon'}}\]
\[+r^{-(\dim_H\mu+\epsilon) \cdot \frac{\text{dim}_H\mu-\epsilon}{\text{dim}_H\mu}} \cdot \{(M\cdot r)^{\dim_H\mu-\epsilon^3} +\mu_{\Lambda}[B_{M \cdot r + 2C \cdot \beta_2^{r^{-(\dim_H\mu-\epsilon)\epsilon'}}}(z') \setminus B_{M \cdot r}(z')]\}^2,\]
where $C>1$ is the same as in the Definition \ref{gibbs}, $\beta_2\in (0,1) $ does not depend on $z,z',r,M$.
\end{lemma}
\begin{proof}
Observe first that $n^{\epsilon'}\ll p$ implies $\epsilon'<\frac{\dim_H \mu-\epsilon}{\dim_H \mu}$.
Cover $B_{M \cdot r}(z')$ by two-sided cylinders $\xi_{i_{-m}\cdots i_0 \cdots i_m}$, where  $m:=\frac{n^{\epsilon'}}{4}$ and $\xi_{i_{-m}\cdots i_0 \cdots i_m} \bigcap B_{M \cdot r}(z')\neq \emptyset$. By Lemma \ref{diamcylinder}, we have $\diam \xi_{i_{-m}\cdots i_0 \cdots i_m} \le 2C \cdot \beta^m$.  So $(\bigcup_{}\xi_{i_{-m}\cdots i_0 \cdots i_m}) \setminus B_{M\cdot r}(z') \subseteq B_{M \cdot r +2 C \cdot \beta^m}(z') \setminus B_{M \cdot r}(z')$. Denote $i_0:=p-n^{\epsilon'}$. Then
\[\int 1_{B_{M \cdot r}(z')} \cdot 1_{\ge 1} \circ [1_{B_{M \cdot r}(z')}+ \cdots +1_{B_{M \cdot r}(z')}\circ (f^R)^{i_0}] \circ (f^R)^{n^{\epsilon'}} d\mu_{\Lambda}\]
\begin{equation}\label{2}
    =\int 1_{(f^R)^{-m}B_{M \cdot r}(z')} \cdot 1_{\ge 1} \circ [1_{(f^R)^{-m}B_{M \cdot r}(z')}+ \cdots +1_{(f^R)^{-m}B_{M \cdot r}(z')}\circ (f^R)^{i_0}] \circ (f^R)^{n^{\epsilon'}} d\mu_{\Lambda}.
\end{equation}

Let $A_1:=(f^R)^{-m}B_{M \cdot r}(z'), A_2:=(f^R)^{-m} \bigcup_{\xi_{i_{-m}\cdots i_0 \cdots i_m} \bigcap \partial B_{M\cdot r}(z')\neq \emptyset} \xi_{i_{-m}\cdots i_0 \cdots i_m}$, $A_0:=(f^R)^{-m} \bigcup \xi_{i_{-m}\cdots i_0 \cdots i_m}$. Observe that 
\[A_0,A_2 \in \mathcal{M}_{2m}, A_0=A_1 \bigcup A_2, \text{ and }  (f^R)^mA_2 \subseteq B_{M \cdot r + 2C \cdot \beta^m}(z') \setminus B_{M \cdot r-2C \cdot \beta^m}(z'),\]
and, moreover, the function $1_{\ge 1} \circ [1_{A_0}+ \cdots +1_{A_0}\circ (f^R)^{i_0}]$ is constant along each $\gamma^s\in \Gamma^s$. Then we can continue the equality (\ref{2}) as
\[=[\int 1_{A_1} \cdot 1_{\ge 1} \circ [1_{A_1}+ \cdots +1_{A_1}\circ (f^R)^{i_0}] \circ (f^R)^{n^{\epsilon'}} d\mu_{\Lambda}\]
\[-\int 1_{A_1} \cdot 1_{\ge 1} \circ [1_{A_0}+ \cdots +1_{A_0}\circ (f^R)^{i_0}] \circ (f^R)^{n^{\epsilon'}} d\mu_{\Lambda}]\]
\[-\int 1_{A_0\setminus{A_1}} \cdot 1_{\ge 1} \circ [1_{A_0}+ \cdots +1_{A_0}\circ (f^R)^{i_0}] \circ (f^R)^{n^{\epsilon'}} d\mu_{\Lambda}\]
\[+[\int 1_{A_0} \cdot 1_{\ge 1} \circ [1_{A_0}+ \cdots +1_{A_0}\circ (f^R)^{i_0}] \circ (f^R)^{n^{\epsilon'}} d\mu_{\Lambda}\]
\[-\int 1_{A_0} d\mu_{\Lambda} \cdot \int 1_{\ge 1} \circ [1_{A_0}+ \cdots +1_{A_0}\circ (f^R)^{i_0}] \circ (f^R)^{n^{\epsilon'}} d\mu_{\Lambda}]\]
\begin{equation}\label{7}
    +\int 1_{A_0} d\mu_{\Lambda} \cdot \int 1_{\ge 1} \circ [1_{A_0}+ \cdots +1_{A_0}\circ (f^R)^{i_0}] \circ (f^R)^{n^{\epsilon'}} d\mu_{\Lambda}.
\end{equation}

Apply now Lemma \ref{decorrelationcylinder}, the relation $A_0\setminus A_1 \subseteq A_2$ and the inequality
\[|1_{\ge 1} \circ [1_{A_1}+ \cdots +1_{A_1}\circ (f^R)^{i_0}]- 1_{\ge 1} \circ [1_{A_0}+ \cdots +1_{A_0}\circ (f^R)^{i_0}]|\le 1_{\ge 1} \circ [1_{A_2}+ \cdots +1_{A_2}\circ (f^R)^{i_0}] \]
to the right hand side of  (\ref{7}). Then it can be estimated as

\[\le \int 1_{A_1} \cdot 1_{\ge 1} \circ [1_{A_2}+ \cdots +1_{A_2}\circ (f^R)^{i_0}] \circ (f^R)^{n^{\epsilon'}} \]
\[+\int 1_{A_2} \cdot 1_{\ge 1} \circ [1_{A_0}+ \cdots +1_{A_0}\circ (f^R)^{i_0}] \circ (f^R)^{n^{\epsilon'}} d\mu_{\Lambda}+C'' \cdot \mu_{\Lambda}(A_0) \cdot \beta_1^{n^{\epsilon'}}+\mu_{\Lambda}(A_0)^2 \cdot i_0\]
\[\le 2\int 1_{A_0} \cdot 1_{\ge 1} \circ [1_{A_0}+ \cdots +1_{A_0}\circ (f^R)^{i_0}] \circ (f^R)^{n^{\epsilon'}} d\mu_{\Lambda}+C'' \cdot \mu_{\Lambda}(A_0) \cdot \beta_1^{n^{\epsilon'}}+\mu_{\Lambda}(A_0)^2 \cdot i_0\]
\[-2\int 1_{A_0} d\mu_{\Lambda}\cdot \int 1_{\ge 1} \circ [1_{A_0}+ \cdots +1_{A_0}\circ (f^R)^{i_0}] \circ (f^R)^{n^{\epsilon'}} d\mu_{\Lambda}\]
\[+2\int 1_{A_0} d\mu_{\Lambda}\cdot \int 1_{\ge 1} \circ [1_{A_0}+ \cdots +1_{A_0}\circ (f^R)^{i_0}] \circ (f^R)^{n^{\epsilon'}} d\mu_{\Lambda}.\]
By applying Lemma \ref{decorrelationcylinder} again, with  $i_0 \le p, m=\frac{n^{\epsilon'}}{4}$, we have
\[\le 3C'' \cdot \mu_{\Lambda}(A_0) \cdot \beta_1^{n^{\epsilon'}}+3 i_0 \cdot \mu_{\Lambda}(A_0)^2\le 3C'' \cdot \mu_{\Lambda}(A_0) \cdot \beta_1^{\frac{n^{\epsilon'}}{2}}+3 p \cdot \mu_{\Lambda}(A_0)^2\]
\[=3C'' \cdot \mu_{\Lambda}((f^R)^mA_0) \cdot \beta_1^{\frac{n^{\epsilon'}}{2}}+3 p \cdot \mu_{\Lambda}((f^R)^mA_0)^2\precsim \mu_{\Lambda}(B_{M\cdot r}(z')) \cdot \beta_1^{\frac{n^{\epsilon'}}{2}}\]
\begin{equation}\label{3}
    +\mu_{\Lambda}[ B_{M \cdot r + 2C \cdot \beta^{\frac{n^{\epsilon'}}{4}}}(z') \setminus B_{M \cdot r}(z')]\cdot \beta_1^{\frac{n^{\epsilon'}}{2}}+p \cdot \{\mu_{\Lambda}(B_{M\cdot r}(z')) +\mu_{\Lambda}[B_{M \cdot r + 2C \cdot \beta^{\frac{n^{\epsilon'}}{4}}}(z') \setminus B_{M \cdot r}(z')]\}^2.
\end{equation}

By making use of the Assumption \ref{geoassumption}, we choose now the same $\epsilon>0$ as in the Proposition \ref{rate}. Then  for almost all $z\in \mathcal{M}$, $ z'\in \interior{(\Lambda)}$  there is $r_{z,z',M}>0$ s.t. $\forall r< r_{z,z',M}$ \[B_{M \cdot r}(z') \subseteq \Lambda \text{ almost surely, and } \frac{T}{r^{\dim_H\mu-\epsilon}}\precsim n \precsim \frac{T}{r^{\dim_H\mu+\epsilon}},\]
\[r^{\dim_H\mu+\epsilon}\le \mu(B_r(z)) \le r^{\dim_H\mu-\epsilon}, \frac{(M \cdot r)^{\dim_H\mu+\epsilon^3}}{\mu(\Lambda)}\le  \mu_{\Lambda}(B_{M \cdot r}(z')) \le \frac{(M\cdot r)^{\dim_H\mu-\epsilon^3}}{\mu(\Lambda) },\]
\[ (\frac{T}{r^{\dim_H\mu-\epsilon}})^{\frac{\dim_H\mu-\epsilon}{\dim_H\mu}} \precsim p\precsim (\frac{T}{r^{\dim_H\mu+\epsilon}})^{\frac{\dim_H\mu-\epsilon}{\dim_H\mu}},(\frac{T}{r^{\dim_H\mu-\epsilon}})^{\epsilon'}\precsim n^{\epsilon'} \precsim (\frac{T}{r^{\dim_H\mu+\epsilon}})^{\epsilon'}.\]
Then we can continue the estimate in the inequality (\ref{3}) as
\[\precsim_{T,\epsilon} (M\cdot r)^{\dim_H\mu-\epsilon^3} \cdot {\beta_2}^{r^{-(\dim_H\mu-\epsilon)\epsilon'}}+\mu_{\Lambda}[ B_{M \cdot r + 2C \cdot \beta_2^{r^{-(\dim_H\mu-\epsilon)\epsilon'}}}(z') \setminus B_{M \cdot r}(z')]\cdot \beta_2^{r^{-(\dim_H\mu-\epsilon)\epsilon'}}\]
\[+r^{-(\dim_H\mu+\epsilon) \cdot \frac{\dim_H\mu-\epsilon}{\dim_H\mu}} \cdot \{(M\cdot r)^{\dim_H\mu-\epsilon^3} +\mu_{\Lambda}[B_{M \cdot r + 2C \cdot \beta_2^{r^{-(\dim_H\mu-\epsilon)\epsilon'}}}(z') \setminus B_{M \cdot r}(z')]\}^2,\]
where $\beta_2 \in (0,1)$ depends on $T, \epsilon', \beta, \beta_1$. 

\end{proof}

By combining the Lemmas \ref{fixreturn}, \ref{veryshortreturn} and \ref{notveryshort}, we obtain the following summary of the results.

\begin{proposition}[Rates of short returns ]\label{shortreturnrate}\ \par
Let $\epsilon, \epsilon'>0$ satisfy the relations $\alpha \cdot (\dim_H\mu-\epsilon)>\frac{\dim_H\mu}{\dim_H\mu-\epsilon}>1$, $p=\lfloor {\lfloor \frac{T}{\mu(B_r(z))}\rfloor}^{\frac{\text{dim}_H\mu-\epsilon}{\text{dim}_H\mu}}\rfloor$, $\epsilon'<\min\{\frac{\min\{\dim_H \mu, { \dim \gamma^u}\}}{12 \dim_H \mu+12 \epsilon}, \frac{\dim_H\mu-\epsilon}{\dim_H\mu}\}$. Then for almost all $z\in \mathcal{M}$, $z'\in \interior{(\Lambda)}$ and for each integer $M>0$ there exists a small enough $r_{z,z',M}>0$ such that $\forall r< r_{z,z',M}$
\[\int_{B_{ M\cdot r} (z')} 1_{\bigcup_{1\le k \le p}(f^R)^{-k}B_{ M\cdot r} (z')}d\mu_{\Lambda}\precsim_{T, \epsilon} (M \cdot r)^{\dim_H\mu -\epsilon} \cdot M^{\frac{\min\{{\dim \gamma^u}, \dim_H \mu\}}{6}}\cdot r^{\frac{\min\{{\dim \gamma^u}, \dim_H \mu\}}{12}}\]
\[+(M\cdot r)^{\dim_H\mu-\epsilon^3} \cdot {\beta_2}^{r^{-(\dim_H\mu-\epsilon)\epsilon'}}+\mu_{\Lambda}[ B_{M \cdot r + 2C \cdot \beta_2^{r^{-(\dim_H\mu-\epsilon)\epsilon'}}}(z') \setminus B_{M \cdot r}(z')]\cdot \beta_2^{r^{-(\dim_H\mu-\epsilon)\epsilon'}}\]
\[+r^{-(\dim_H\mu+\epsilon) \cdot \frac{\dim_H\mu-\epsilon}{\dim_H\mu}} \cdot \{(M\cdot r)^{\dim_H\mu-\epsilon^3} +\mu_{\Lambda}[B_{M \cdot r + 2C \cdot \beta_2^{r^{-(\dim_H\mu-\epsilon)\epsilon'}}}(z') \setminus B_{M \cdot r}(z')]\}^2,\]
where a constant $C>1$ is the same as in the Definition \ref{gibbs}, a constant $\beta_2\in (0,1) $ is independent on $z,z',r,M$.
\end{proposition}
\begin{proof}
 Recall that in the Lemmas \ref{fixreturn} \ref{veryshortreturn}, \ref{notveryshort}, we fixed an integer $M>0$. Then for almost all $z \in \mathcal{M}$ and  $z' \in \interior{(\Lambda)}$ the desired estimates hold. The set of such points in $ \mathcal{M} \times \interior{(\Lambda)}$ has the full measure with respect to $\mu(\mathcal{M}) \times \mu(\interior{(\Lambda)})$, and it does depend on $M$. However, since $M$ is an integer, we can find a smaller set of points $(z,z') \in \mathcal{M} \times \interior{(\Lambda)}$ such that it does not depend on $M>0$ and has full measure $\mu(\mathcal{M}) \times \mu(\interior{(\Lambda)})$. Therefore the required estimate holds.

\end{proof}

\subsection{Coronas}
Here we study two coronas, which appeared in the previous sections. One of them is in $\mathcal{M}$ and has a measure $\mu[B_{r+C' \cdot r^{\frac{(\dim_H\mu-\epsilon)^2\cdot \alpha }{\dim_H\mu} } }(z)\setminus B_{r-C' \cdot r^{\frac{(\dim_H\mu-\epsilon)^2\cdot \alpha }{\dim_H\mu} }}(z)] $ (see the Proposition \ref{rate}). The other one is in $\Lambda$ and its measure is $\mu_{\Lambda}[ B_{M \cdot r + 2C \cdot \beta_2^{r^{-(\dim_H\mu-\epsilon)\epsilon'}}}(z') \setminus B_{M \cdot r}(z')]$ (see the Proposition \ref{shortreturnrate}).

\begin{proposition}[Coronas in $\mathcal{M}$ and $\Lambda$]\label{coronaestimate}\ \par
For almost every $z \in \mathcal{M}, z' \in \interior{(\Lambda)}$ there are $r_z, r_{z',M}>0$ such that $\forall r< \min \{r_z, r_{z',M}\} $
\[\mu\{B_{r+C' \cdot r^{\frac{(\dim_H\mu-\epsilon)^2\cdot \alpha }{\dim_H\mu}} }(z)\setminus B_{r-C' \cdot r^{\frac{(\dim_H\mu-\epsilon)^2\cdot \alpha }{\dim_H\mu} }}(z)\} \precsim_{z, \dim \gamma^u} r^{\frac{(1+\frac{(\dim_H\mu-\epsilon)^2\cdot \alpha }{\dim_H\mu}) \cdot \dim \gamma^u}{2}},\]
\[\mu_{\Lambda}[ B_{M \cdot r + 2C \cdot \beta_2^{r^{-(\dim_H\mu-\epsilon)\epsilon'}}}(z') \setminus B_{M \cdot r}(z')]\precsim_{z, \dim \gamma^u} (M\cdot r)^{\frac{\dim \gamma^u }{2}} \cdot \beta_2^{\frac{r^{-(\dim_H\mu-\epsilon)\epsilon'}\cdot \dim \gamma^u}{2}}.\]
\end{proposition}
\begin{proof}
For corona in $\mathcal{M}$ we have from Lemma \ref{ballinhorseshoe} for almost all $z \in  \bigcup_{i \ge 1} \bigcup_{0 \le j < R_i} f^j(\Lambda_i)$, $z=f^{j_z}(z')$, $z' \in \interior{(\Lambda)}$ that there exists  such $r'_z>0$, that $\forall r< r'_z$ 

\[\mu(f^{-j_z}[B_{r+C' \cdot r^{\frac{(\dim_H\mu-\epsilon)^2\cdot \alpha }{\dim_H\mu} } }(z)\setminus B_{r-C' \cdot r^{\frac{(\dim_H\mu-\epsilon)^2\cdot \alpha }{\dim_H\mu} }}(z)] \bigcap \{\bigcup_{i \ge 1} \bigcup_{1 \le j < R_i} f^j(\Lambda_i)\})=0.\]

Hence by the invariance of $\mu$ (i.e. $f_* \mu=\mu$) we have
\[\mu([B_{r+C' \cdot r^{\frac{(\dim_H\mu-\epsilon)^2\cdot \alpha }{\dim_H\mu} } }(z)\setminus B_{r-C' \cdot r^{\frac{(\dim_H\mu-\epsilon)^2\cdot \alpha }{\dim_H\mu} }}(z)] \bigcap f^{j_z}\{\bigcup_{i \ge 1} \bigcup_{1 \le j < R_i} f^j(\Lambda_i)\})=0.\]

Therefore
\[[B_{r+C' \cdot r^{\frac{(\dim_H\mu-\epsilon)^2\cdot \alpha }{\dim_H\mu} } }(z)\setminus B_{r-C' \cdot r^{\frac{(\dim_H\mu-\epsilon)^2\cdot \alpha }{\dim_H\mu} }}(z)] \subseteq f^{j_z}\Lambda \text{ almost surely}.\]

Because $f^{j_z}$ is a local diffeomorphism, then for a sufficiently small $r$ all manifolds $f^{j_z}\gamma^u$ ($\gamma^u\in \Gamma^u$) in any non-empty set $(f^{j_z}\gamma^u) \bigcap [B_{r+C' \cdot r^{\frac{(\dim_H\mu-\epsilon)^2\cdot \alpha }{\dim_H\mu}} }(z)\setminus B_{r-C' \cdot r^{\frac{(\dim_H\mu-\epsilon)^2\cdot \alpha }{\dim_H\mu} }}(z)] $ are almost flat. From the Gauss lemma for the exponential map $\exp_{z}$ we have, that in a neughborhood of $z \in \mathcal{M}$ 
\[(\exp_z)^{-1}[B_{r+C' \cdot r^{\frac{(\dim_H\mu-\epsilon)^2\cdot \alpha }{\dim_H\mu}} }(z)\setminus B_{r-C' \cdot r^{\frac{(\dim_H\mu-\epsilon)^2\cdot \alpha }{\dim_H\mu} }}(z)] \]
\[\subseteq \Cor:=\{v \in \mathbb{R}^{\dim \mathcal{M}}: \langle v,v \rangle_{z}\in [r-C' \cdot r^{\frac{(\dim_H\mu-\epsilon)^2\cdot \alpha }{\dim_H\mu}}, r+C' \cdot r^{\frac{(\dim_H\mu-\epsilon)^2\cdot \alpha }{\dim_H\mu}}]\}, \]
where $\langle \cdot, \cdot \rangle_z$ is the Riemannian metric at $z\in \mathcal{M}$. This is a corona in an ellipse. If $r$ is sufficiently small, say $r<r_z<r_z'$ for some $r_z>0$, then all manifolds $(\exp_z)^{-1} f^{j_z}\gamma^u$ in $\Cor$ are almost flat. So their diameters (in Euclidean norm) satisfy the inequality  
\[\diam \{[(\exp_z)^{-1} f^{j_z}\gamma^u]\bigcap \Cor\} \precsim_z \cdot \sqrt{(r+C' \cdot r^{\frac{(\dim_H\mu-\epsilon)^2\cdot \alpha }{\dim_H\mu}})^2-(r-C' \cdot r^{\frac{(\dim_H\mu-\epsilon)^2\cdot \alpha }{\dim_H\mu}})^2}\]
$\precsim \sqrt{r \cdot r^{\frac{(\dim_H\mu-\epsilon)^2\cdot \alpha }{\dim_H\mu}}}=r^{\frac{1+\frac{(\dim_H\mu-\epsilon)^2\cdot \alpha }{\dim_H\mu}}{2}}$. Therefore the Lebesgue measure of $\Cor$ on $(\exp_z)^{-1} f^{j_z}\gamma^u$ can be controlled by  diameters, i.e.
\[\Leb_{(\exp_z)^{-1} f^{j_z}\gamma^u}\{[(\exp_z)^{-1} f^{j_z}\gamma^u]\bigcap \Cor\} \precsim_{z, \dim \gamma^u} r^{\frac{(1+\frac{(\dim_H\mu-\epsilon)^2\cdot \alpha }{\dim_H\mu}) \cdot \dim \gamma^u}{2}}.\] 
%is it lebesgue measure of ellipse/(fixed)%

Since $\exp_z$ is a local diffeomorphism, then 
\[\Leb_{f^{j_z}\gamma^u} \{B_{r+C' \cdot r^{\frac{(\dim_H\mu-\epsilon)^2\cdot \alpha }{\dim_H\mu}} }(z)\setminus B_{r-C' \cdot r^{\frac{(\dim_H\mu-\epsilon)^2\cdot \alpha }{\dim_H\mu} }}(z)\} \precsim_{z, \dim \gamma^u} r^{\frac{(1+\frac{(\dim_H\mu-\epsilon)^2\cdot \alpha }{\dim_H\mu}) \cdot \dim \gamma^u}{2}}.\]

Now, since $f^{j_z}$ is also a local diffeomorphism, we have 
\[\Leb_{\gamma^u} \{f^{-j_z}[B_{r+C' \cdot r^{\frac{(\dim_H\mu-\epsilon)^2\cdot \alpha }{\dim_H\mu}} }(z)\setminus B_{r-C' \cdot r^{\frac{(\dim_H\mu-\epsilon)^2\cdot \alpha }{\dim_H\mu} }}(z)]\} \precsim_{z, \dim \gamma^u} r^{\frac{(1+\frac{(\dim_H\mu-\epsilon)^2\cdot \alpha }{\dim_H\mu}) \cdot \dim \gamma^u}{2}}.\]

By making use of (\ref{horsesrbleb}) and integrating over all $\gamma^u \in \Gamma^u$, we get
\[\mu\{f^{-j_z}[B_{r+C' \cdot r^{\frac{(\dim_H\mu-\epsilon)^2\cdot \alpha }{\dim_H\mu}} }(z)\setminus B_{r-C' \cdot r^{\frac{(\dim_H\mu-\epsilon)^2\cdot \alpha }{\dim_H\mu} }}(z)]\} \precsim_{z, \dim \gamma^u} r^{\frac{(1+\frac{(\dim_H\mu-\epsilon)^2\cdot \alpha }{\dim_H\mu}) \cdot \dim \gamma^u}{2}}.\]

Hence, thanks to the invariance of the measure $\mu$ ($f_{*}\mu=\mu$)

\[\mu\{B_{r+C' \cdot r^{\frac{(\dim_H\mu-\epsilon)^2\cdot \alpha }{\dim_H\mu}} }(z)\setminus B_{r-C' \cdot r^{\frac{(\dim_H\mu-\epsilon)^2\cdot \alpha }{\dim_H\mu} }}(z)\} \precsim_{z, \dim \gamma^u} r^{\frac{(1+\frac{(\dim_H\mu-\epsilon)^2\cdot \alpha }{\dim_H\mu}) \cdot \dim \gamma^u}{2}}.\]

For the corona in $\Lambda$, i.e. for $B_{M \cdot r +2 C \cdot \beta_2^{r^{-(\dim_H\mu-\epsilon)\epsilon'}}}(z') \setminus B_{M \cdot r}(z')$, a trick is the same and even more straightforward. Indeed, because  $\gamma^u\in \Gamma^u$ intersects with this corona directly, there is $r_{z', M}>0$, such that $\forall r<r_{z',M}$ 
\[\mu_{\Lambda}[ B_{M \cdot r + 2C \cdot \beta_2^{r^{-(\dim_H\mu-\epsilon)\epsilon'}}}(z') \setminus B_{M \cdot r}(z')]\precsim_{z', \dim \gamma^u} (M \cdot r)^{\frac{\dim \gamma^u }{2}} \cdot \beta_2^{ \frac{r^{-(\dim_H\mu-\epsilon)\epsilon'}\cdot\dim \gamma^u}{2}}.\]
\end{proof}

\subsection{Conclusion of the Proof of the Theorem \ref{thm}}
We will start with a rough estimate of $\dim_H \mu$.
\subsection*{Proof that $\dim_H\mu \ge \dim {\gamma}^u$:}
\begin{lemma}\label{roughdim}
It follows from the Assumptions \ref{geoassumption} and  \ref{assumption} that $\dim_H\mu \ge \dim {\gamma}^u$.
\end{lemma}
\begin{proof}
Due to the ergodicity of $\mu$, it is enough to consider $z' \in \interior{(\Lambda)}$ and a ball $B_r(z')$ in $\Lambda$. If $r$ is sufficiently small then $\gamma^u \in \Gamma^u$ is almost flat in $B_r(z')\bigcap \gamma^u$. Therefore $\diam \{B_r(z')\bigcap \gamma^u\} \precsim_{z'} r$. By the same trick as in the Proposition \ref{coronaestimate} we get  \[\Leb_{\gamma^u}(B_r(z')) \precsim_{z',\dim \gamma^u}r^{\dim \gamma^u}  \text{ and } \mu(B_r(z')) \precsim_{z',\dim \gamma^u}r^{\dim \gamma^u}.\]
On the other hand, by the Assumption \ref{geoassumption} for almost all $z' \in \mathcal{M}$ and any $m \in \mathbb{N}$ there is $r_{z',m}>0$, such that $\mu(B_r(z')) \ge r^{\dim_H \mu+ \frac{1}{m}}$ for any $r<r_{z',m}$. This implies that $\dim \gamma^u \le \dim_H \mu + \frac{1}{m}, \forall m \in \mathbb{N}$. By taking the limit $m\to \infty$ one gets $\dim_H\mu \ge \dim {\gamma}^u$.
\end{proof}

We will address now the convergence rates in the Theorem \ref{thm}. According to the Proposition \ref{rate} it is enough to find the convergence rates for short returns and coronas.
A proof will consist of several steps. 
\subsection*{Pull back short returns $\int_{B_r(z)} 1_{\bigcup_{1\le j \le p}f^{-j}B_r(z)}d\mu$.}

For almost any $z \in  \bigcup_{i \ge 1} \bigcup_{0 \le j < R_i} f^j(\Lambda_i)$ we have $z=f^{j_z}(z')$ for some $z' \in \interior{\Lambda}$. By the Lemma \ref{comparetopballmetricball}, there is a topological ball $TB_r(z')=TB_r(f^{-j_z}(z))$ and a constant $C_z>1$, such that $B_{C^{-1}_z\cdot r} (f^{-j_z}z) \subseteq TB_r(f^{-j_z}z) \subseteq B_{C_z\cdot r} (f^{-j_z}z)$.
\begin{lemma}[Pull short returns back to $\Lambda$]\label{pullbackhorseshoe}\ \par
There exists a small enough $r_z>0$ such that $\forall r< r_z$ the following inequality holds
\[ \int_{B_r(z)} 1_{\bigcup_{1\le k \le p}f^{-k}B_r(z)}d\mu \le \mu(\Lambda) \cdot \int_{B_{C_z\cdot r} (f^{-j_z}z)} 1_{\bigcup_{1\le k \le p}(f^R)^{-k}B_{C_z\cdot r} (f^{-j_z}z)}d\mu_{\Lambda} .\]
\end{lemma}
\begin{proof}
It follows from the invariance of $\mu$ (i.e. $f_*\mu=\mu$) that 
\[ \int_{B_r(z)} 1_{\bigcup_{1\le k \le p}f^{-k}B_r(z)}d\mu=\int_{TB_r(f^{-j_z}z)} 1_{\bigcup_{1\le k \le p}f^{-k}TB_r(f^{-j_z}z)}d\mu.\]

By Lemma \ref{ballinhorseshoe}, we have
$TB_r(f^{-j_z}z) \subseteq \Lambda$ almost surely $\forall r< r_z$ for some small enough $r_z>0$. Then the equality above can be continued as

\begin{equation}\label{4}
  = \int_{TB_r(f^{-j_z}z)} 1_{\bigcup_{1\le k \le p}f^{-k}TB_r(f^{-j_z}z)}d\mu|_{\Lambda} = \mu(\Lambda) \cdot \int_{TB_r(f^{-j_z}z)} 1_{\bigcup_{1\le k \le p}f^{-k}TB_r(f^{-j_z}z)}d\mu_{\Lambda}.   
\end{equation}

Denote $R^i:=R^{i-1}+R \circ f^R, R^1:=R\ge 1$. Then $p\le R^p$. Also note that, since $R$ is the first return time, then $f^k \notin \Lambda$ almost surely (a.s.) if $R^i<k<R^{i+1}$. Therefore
\[1_{TB_r(f^{-j_z}z)} \cdot 1_{\bigcup_{1\le k \le R^p}f^{-k}TB_r(f^{-j_z}z)}=1_{TB_r(f^{-j_z}z)} \cdot 1_{\bigcup_{1\le k \le p}(f^R)^{-k}TB_r(f^{-j_z}z)} \text{ a.s.}.\]
Then we can estimate (\ref{4}) as 
\[\le \mu(\Lambda) \cdot \int_{TB_r(f^{-j_z}z)} 1_{\bigcup_{1\le k \le R^p}f^{-k}TB_r(f^{-j_z}z)}d\mu_{\Lambda}\]
\begin{equation}\label{5}
    =\mu(\Lambda) \cdot \int_{TB_r(f^{-j_z}z)} 1_{\bigcup_{1\le k \le p}(f^R)^{-k}TB_r(f^{-j_z}z)}d\mu_{\Lambda}.
\end{equation}

By Lemma \ref{comparetopballmetricball} there is a constant $C_z\ge 1$, such that $TB_r(f^{-j_z}z) \subseteq B_{C_z\cdot r} (f^{-j_z}z)$, and we can continue the equality (\ref{5}) as
\[\le\mu(\Lambda) \cdot \int_{B_{C_z\cdot r} (f^{-j_z}z)} 1_{\bigcup_{1\le k \le p}(f^R)^{-k}B_{C_z\cdot r} (f^{-j_z}z)}d\mu_{\Lambda}.\]
\end{proof}

Hence the short returns problem on $\mathcal{M}$ becomes a short returns problem on $\Lambda$. 

The next task is to 

\subsection*{Estimate $\frac{1}{\mu(B_r(z))}\cdot \int_{B_{C_z\cdot r} (f^{-j_z}z)} 1_{\bigcup_{1\le k \le p}(f^R)^{-k}B_{C_z\cdot r} (f^{-j_z}z)}d\mu_{\Lambda}$.}

\begin{lemma}\label{afterpullback}
 If $\epsilon<\min\{\frac{\min\{\dim \gamma^u, \dim_H \mu\}}{24},\frac{1}{3\dim_H\mu}\}$ satisfies $\alpha \cdot (\dim_H\mu-\epsilon)>\frac{\dim_H\mu}{\dim_H\mu-\epsilon}>1$, then for almost every $z\in \mathcal{M}$ there is $r_z>0$, such that $\forall r<r_z$
 \[\frac{1}{\mu(B_r(z))}\cdot \int_{B_{C_z\cdot r} (f^{-j_z}z)} 1_{\bigcup_{1\le k \le p}(f^R)^{-k}B_{C_z\cdot r} (f^{-j_z}z)}d\mu_{\Lambda}\precsim_{T,z,\epsilon} r^{\ \frac{\min\{{\dim \gamma^u}, \dim_H \mu\}}{12}-2\epsilon}+ r^{\frac{\epsilon^2}{\dim_H \mu}-3\epsilon^3}\to 0.\]

\end{lemma}
\begin{proof}
Observe that $C_z < \infty $ for almost all $z\in \mathcal{M}$. Denote $F_1:=\{C_z<\infty\}$, $F_2 \times F_3:=\{(z,z')\in \mathcal{M}\times \interior{(\Lambda)}: \text{ the Propositions }\ref{shortreturnrate}, \ref{coronaestimate} \text{ hold}\}$. Define a new measure one set in $\mathcal{M}$ as $F:=F_1 \bigcap F_2 \bigcap \{\bigcup_{j\ge 0}f^j(\interior{(\Lambda)}\bigcap F^c_3 )\}^c$. Then for any $z \in F$ we have $ C_z < \infty$ and $(z,f^{-j_z}z)\in F_2 \times F_3$. Let $M:=\lfloor C_z \rfloor+1, z'=f^{-j_z}z $. Then by the Proposition \ref{shortreturnrate} there is $r_{z,z',M}=r_{z,f^{-j_z}z,\lfloor C_z \rfloor+1}>0$ such that $\forall r< r_{z,f^{-j_z}z,\lfloor C_z \rfloor+1}$
\[\int_{B_{C_z\cdot r} (f^{-j_z}z)} 1_{\bigcup_{1\le k \le p}(f^R)^{-k}B_{C_z\cdot r} (f^{-j_z}z)}d\mu_{\Lambda}\le \int_{B_{M\cdot r} (z')} 1_{\bigcup_{1\le k \le p}(f^R)^{-k}B_{M\cdot r} (z')}d\mu_{\Lambda}\]
\[\precsim_{T,\epsilon} (M \cdot r)^{\dim_H\mu -\epsilon} \cdot M^{\frac{\min\{{\dim \gamma^u}, \dim_H \mu\}}{6}}\cdot r^{\frac{\min\{{\dim \gamma^u}, \dim_H \mu\}}{12}}\]
\[+(M\cdot r)^{\dim_H\mu-\epsilon^3} \cdot {\beta_2}^{r^{-(\dim_H\mu-\epsilon)\epsilon'}}+\mu_{\Lambda}[ B_{M \cdot r +2 C \cdot \beta_2^{r^{-(\dim_H\mu-\epsilon)\epsilon'}}}(z') \setminus B_{M \cdot r}(z')]\cdot \beta_2^{r^{-(\dim_H\mu-\epsilon)\epsilon'}}\]
\[+r^{-(\dim_H\mu+\epsilon) \cdot \frac{\dim_H\mu-\epsilon}{\dim_H\mu}} \cdot \{(M\cdot r)^{\dim_H\mu-\epsilon^3} +\mu_{\Lambda}[B_{M \cdot r + 2C \cdot \beta_2^{r^{-(\dim_H\mu-\epsilon)\epsilon'}}}(z') \setminus B_{M \cdot r}(z')]\}^2.\]

By the Proposition \ref{coronaestimate}, the right hand side of the inequality above can be estimated  as
\[\precsim_{T,C_z,z, \epsilon} r^{\dim_H\mu -\epsilon+\frac{\min\{{\dim \gamma^u}, \dim_H \mu\}}{12}}+ r^{\dim_H \mu-\epsilon^3} \cdot {\beta_2}^{r^{-(\dim_H\mu-\epsilon)\epsilon'}}+r^{\frac{\dim \gamma^u }{2}} \cdot \beta_2^{\frac{r^{-(\dim_H\mu-\epsilon)\epsilon'}\cdot \dim \gamma^u}{2}+r^{-(\dim_H\mu-\epsilon)\epsilon'}}\]
\[+r^{-(\dim_H\mu+\epsilon) \cdot \frac{\dim_H\mu-\epsilon}{\dim_H\mu}} \cdot \{ r^{\dim_H\mu-\epsilon^3} +r^{\frac{\dim \gamma^u }{2}} \cdot \beta_2^{\frac{r^{-(\dim_H\mu-\epsilon)\epsilon'}\cdot \dim \gamma^u}{2}}\}^2.\]

By the Assumption \ref{geoassumption} for almost all $z\in \mathcal{M}$ there is $r_z>0$ such that $\forall r<r_{z}$ \[r^{\dim_H\mu+\epsilon^3}\le \mu(B_r(z)) \le r^{\dim_H\mu-\epsilon^3}.\] 

Then $\forall r<\min \{ r_{z,f^{-j_z}z,\lfloor C_z \rfloor+1}, r_z\}$, \[\frac{1}{\mu(B_r(z))}\cdot \int_{B_{C_z\cdot r} (f^{-j_z}z)} 1_{\bigcup_{1\le k \le p}(f^R)^{-k}B_{C_z\cdot r} (f^{-j_z}z)}d\mu_{\Lambda}\]
\[\precsim_{T,z, \epsilon} \frac{1}{r^{\dim_H\mu+\epsilon^3}} \cdot \{r^{\dim_H\mu -\epsilon+\frac{\min\{{\dim \gamma^u}, \dim_H \mu\}}{12}}+r^{\dim_H \mu-\epsilon} \cdot {\beta_2}^{r^{-(\dim_H\mu-\epsilon)\epsilon'}}\]
\[+r^{\frac{\dim \gamma^u }{2}} \cdot \beta_2^{\frac{r^{-(\dim_H\mu-\epsilon)\epsilon'}\cdot \dim \gamma^u}{2}+r^{-(\dim_H\mu-\epsilon)\epsilon'}}+r^{-(\dim_H\mu+\epsilon) \cdot \frac{\dim_H\mu-\epsilon}{\dim_H\mu}} \cdot ( r^{\dim_H\mu-\epsilon^3} +r^{\frac{\dim \gamma^u }{2}} \cdot \beta_2^{\frac{r^{-(\dim_H\mu-\epsilon)\epsilon'}\cdot \dim \gamma^u}{2}})^2\}\]
\[\le r^{\ -2\epsilon+\frac{\min\{{\dim \gamma^u}, \dim_H \mu\}}{12}}+r^{-\epsilon-\epsilon^3} \cdot {\beta_2}^{r^{-(\dim_H\mu-\epsilon)\epsilon'}}+r^{\frac{\dim \gamma^u }{2}-\epsilon^3-\dim_H\mu} \cdot \beta_2^{\frac{r^{-(\dim_H\mu-\epsilon)\epsilon'}\cdot \dim \gamma^u}{2}+r^{-(\dim_H\mu-\epsilon)\epsilon'}}\]
\[+\frac{1}{r^{\dim_H\mu+\epsilon^3}} \cdot r^{-(\dim_H\mu+\epsilon) \cdot \frac{\dim_H\mu-\epsilon}{\dim_H\mu}} \cdot \{ r^{\dim_H\mu-\epsilon^3} +r^{\frac{\dim \gamma^u }{2}} \cdot \beta_2^{\frac{r^{-(\dim_H\mu-\epsilon)\epsilon'}\cdot \dim \gamma^u}{2}}\}^2\]
\[\precsim r^{\ \frac{\min\{{\dim \gamma^u}, \dim_H \mu\}}{12}-2\epsilon}+r^{\frac{\epsilon^2}{\dim_H \mu}-3\epsilon^3}\to 0,\]
this last $\precsim$ is because $\beta_2^{r^{-c}}\ll r^{c'}$  for any $ c,c'>0$.
\end{proof}

By combining the Lemmas \ref{pullbackhorseshoe} and \ref{afterpullback} we get an 
\subsection*{Estimate of the short returns rates $\frac{1}{\mu(B_r(z))}\int_{B_r(z)} 1_{\bigcup_{1\le j \le p}f^{-j}B_r(z)}d\mu$.}
\begin{lemma}\label{lastshortreturnrate}
 If $\epsilon<\min\{\frac{\min\{\dim \gamma^u, \dim_H \mu\}}{24},\frac{1}{3\dim_H\mu}\}$ satisfies $\alpha \cdot (\dim_H\mu-\epsilon)>\frac{\dim_H\mu}{\dim_H\mu-\epsilon}>1$, then for almost all $z\in \mathcal{M}$ 
 
 \[\frac{1}{\mu(B_r(z))}\int_{B_r(z)} 1_{\bigcup_{1\le j \le p}f^{-j}B_r(z)}d\mu \precsim_{z,T,\epsilon} r^{\ \frac{\min\{{\dim \gamma^u}, \dim_H \mu\}}{12}-2\epsilon}+ r^{\frac{\epsilon^2}{\dim_H \mu}-3\epsilon^3}\to 0.\]

\end{lemma}

We finished now the estimates of short returns and move to
\subsection*{Estimate of coronas rates $\frac{\mu[B_{r+C' \cdot r^{\frac{(\dim_H\mu-\epsilon)^2\cdot \alpha }{\dim_H\mu} } }(z)\setminus B_{r-C' \cdot r^{\frac{(\dim_H\mu-\epsilon)^2\cdot \alpha }{\dim_H\mu} }}(z)]}{\mu(B_r(z))}$. }
\begin{lemma}\label{lastcoronarate}
If $\alpha> \frac{2}{\dim \gamma^u}-\frac{1}{\dim_H\mu}$ and $\epsilon<\min\{\frac{\min\{\dim \gamma^u, \dim_H \mu\}}{24},\frac{1}{3\dim_H\mu}\}$ is small enough, so that 
$\alpha >\frac{\frac{2}{\dim \gamma^u}-\frac{1}{\dim_H \mu}+\frac{2\epsilon}{\dim \gamma^u \cdot \dim_H \mu}}{(1-\frac{\epsilon}{\dim_H\mu})^2}$,
then for almost all $z\in \mathcal{M}$  
\[\frac{\mu[B_{r+C' \cdot r^{\frac{(\dim_H\mu-\epsilon)^2\cdot \alpha }{\dim_H\mu}} }(z)\setminus B_{r-C' \cdot r^{\frac{(\dim_H\mu-\epsilon)^2\cdot \alpha }{\dim_H\mu} }}(z)]} {\mu(B_r(z))}\precsim_{z} r^{\frac{(1+\frac{(\dim_H\mu-\epsilon)^2\cdot \alpha }{\dim_H\mu}) \cdot \dim \gamma^u-2\dim_H \mu -2\epsilon}{2}}\to 0.\]
\end{lemma}
\begin{proof}
By the Assumption \ref{geoassumption} for almost all $z\in \mathcal{M}$ there exists $r_z>0$, such that $\forall r<r_{z}$ \[r^{\dim_H\mu+\epsilon}\le \mu(B_r(z)) \le r^{\dim_H\mu-\epsilon}.\]
Now, in view of the Proposition \ref{coronaestimate} and the choice of $\epsilon$, we have \[(1+\frac{(\dim_H\mu-\epsilon)^2\cdot \alpha }{\dim_H\mu}) \cdot \dim \gamma^u-2\dim_H \mu -2\epsilon>0,\] 
\[\frac{\mu[B_{r+C' \cdot r^{\frac{(\dim_H\mu-\epsilon)^2\cdot \alpha }{\dim_H\mu}} }(z)\setminus B_{r-C' \cdot r^{\frac{(\dim_H\mu-\epsilon)^2\cdot \alpha }{\dim_H\mu} }}(z)] }{\mu(B_r(z))}\precsim_{z, \dim \gamma^u} r^{\frac{(1+\frac{(\dim_H\mu-\epsilon)^2\cdot \alpha }{\dim_H\mu}) \cdot \dim \gamma^u-2\dim_H \mu -2\epsilon}{2}}.\]
\end{proof}

By that we finished the estimation of the rates for coronas, and can now conclude a proof of the Theorem \ref{thm}.
\subsection*{Convergence rates $a>0$ in $d_{TV}(N^{r,T,z},P)\precsim_{T,z}r^a$.}

\begin{lemma}\label{lastproofs}
    If $\alpha> \frac{2}{\dim \gamma^u}-\frac{1}{\dim_H\mu}$ and  $\epsilon<\min\{\frac{\min\{\dim \gamma^u, \dim_H \mu\}}{24},\frac{1}{3\dim_H\mu}\}$ are such that  \[\alpha >\max \{\frac{\frac{1}{\dim_H\mu}}{(1-\frac{\epsilon}{\dim_H\mu})^2}, \frac{\frac{2}{\dim \gamma^u}-\frac{1}{\dim_H \mu}+\frac{2\epsilon}{\dim \gamma^u \cdot \dim_H \mu}}{(1-\frac{\epsilon}{\dim_H\mu})^2}\},\] then from the Proposition \ref{rate}, Lemma \ref{lastcoronarate} and Lemma \ref{lastshortreturnrate} we have 

\[a:=\min \{ \frac{(\dim_H\mu-\epsilon)^2(\xi-1)}{\dim_H\mu}, \frac{\epsilon\cdot (\dim_H\mu-\epsilon)}{\dim_H\mu}, \frac{(1+\frac{(\dim_H\mu-\epsilon)^2\cdot \alpha }{\dim_H\mu}) \cdot \dim \gamma^u-2\dim_H \mu -2\epsilon}{2},\] \[\frac{\epsilon^2}{\dim_H \mu}-3\epsilon^3, \frac{\min\{{\dim \gamma^u}, \dim_H \mu\}}{12}-2\epsilon\}\]
\[=\min \{ \frac{(\dim_H\mu-\epsilon)^2(\xi-1)}{\dim_H\mu}, \frac{(1+\frac{(\dim_H\mu-\epsilon)^2\cdot \alpha }{\dim_H\mu}) \cdot \dim \gamma^u-2\dim_H \mu -2\epsilon}{2},\] 
\[\frac{\epsilon^2}{\dim_H \mu}-3\epsilon^3, \frac{\min\{{\dim \gamma^u}, \dim_H \mu\}}{12}-2\epsilon\},\] 
 where the last equality comes from the relation $\frac{\epsilon^2}{\dim_H \mu}-3\epsilon^3 \le \frac{\epsilon\cdot (\dim_H\mu-\epsilon)}{\dim_H\mu}.$

If $\frac{d\mu}{d\Leb_{\mathcal{M}}}\in L^{\infty}_{loc}(\mathcal{M})$ and $\epsilon<\min\{\frac{\min\{\dim \gamma^u, \dim_H \mu\}}{24},\frac{1}{3\dim_H\mu}\}$ are such that $\alpha >\frac{\frac{1}{\dim_H\mu}}{(1-\frac{\epsilon}{\dim_H\mu})^2}$,
then \[\frac{\mu(B_{r+C \cdot r^{\frac{(\dim_H\mu-\epsilon)^2\cdot \alpha }{\dim_H\mu}  } }(z)\setminus B_{r-C \cdot r^{\frac{(\dim_H\mu-\epsilon)^2\cdot \alpha }{\dim_H\mu}}}(z))}{\mu(B_r(z))}\precsim_{z} r^{\frac{(\dim_H \mu-\epsilon)^2\cdot \alpha}{\dim_H\mu}-1},\] 
and from the Proposition \ref{rate} and Lemma \ref{lastshortreturnrate} we have
\[a:=\min \{\frac{(\dim_H\mu-\epsilon)^2(\xi-1)}{\dim_H\mu}, \frac{\epsilon\cdot (\dim_H\mu-\epsilon)}{\dim_H\mu},\frac{(\dim_H \mu-\epsilon)^2\cdot \alpha}{\dim_H\mu}-1, \frac{\epsilon^2}{\dim_H \mu}-3\epsilon^3,\]
\[\frac{\min\{{\dim \gamma^u}, \dim_H \mu\}}{12}-2\epsilon \}=\min \{\frac{(\dim_H\mu-\epsilon)^2(\xi-1)}{\dim_H\mu}, \frac{(\dim_H \mu-\epsilon)^2\cdot \alpha}{\dim_H\mu}-1,\]
\[\frac{\epsilon^2}{\dim_H \mu}-3\epsilon^3, \frac{\min\{{\dim \gamma^u}, \dim_H \mu\}}{12}-2\epsilon \},\]
where again the last equality holds because $\frac{\epsilon^2}{\dim_H \mu}-3\epsilon^3 \le \frac{\epsilon\cdot (\dim_H\mu-\epsilon)}{\dim_H\mu}$.
\end{lemma}
This completes a proof of the Theorem \ref{thm}. 

\section{Proof of the Theorem \ref{thm2}}\label{provethm2}
The scheme of a proof of the Theorem \ref{thm2} is analogous to the one of the Theorem \ref{thm}, i.e. it contains  estimates of the short returns and coronas. But the ways to establish these estimates are the other ones, because of the differences of assumptions in the Theorem \ref{thm2} and in the Theorem \ref{thm}.
\subsection{Properties of the First Return Map $f^{\overline{R}}$}
\begin{lemma}[Properties of the first returns]\ \par
The map $f^{\overline{R}}$ is ergodic with respect to the probability $\mu_U:=\frac{\mu|_{U}}{\mu(U)}$, and it is bijective on $U \bigcap \bigcup_{i \ge 1} \bigcup_{0 \le j < R_i} f^j(\Lambda_i)$.
\end{lemma}
\begin{proof} A proof that
$f^{\overline{R}}$ is one-to-one is the same as in the Lemma \ref{inducemapbi}, which uses the first return $\overline{R}$ and replaces $R_i,R_j$ by $\overline{R}(x), \overline{R}(y)$. Clearly, the map $f^{\overline{R}}$ is ergodic due to the exponential decay of correlations, and it is also onto due to ergodicity.
\end{proof}

Now, since $\mu\{\interior{(U)}\}>0$, then by the Birkhoff's Ergodic Theorem for almost every $z \in  \bigcup_{i \ge 1} \bigcup_{0 \le j < R_i} f^j(\Lambda_i)$ we have $z=f^{j_z}(z')$ for some $z' \in \interior{(U)}$ and $j_z \in \mathbb{N}$. Analogously to the proof of Lemma \ref{ballinhorseshoe}, we obtain
\begin{lemma}[Pull metric balls back to $U$]\ \label{ballinhorseshoe1}\ \par
For almost every $z \in \bigcup_{i \ge 1} \bigcup_{0 \le j < R_i} f^j(\Lambda_i)$ there exists $j_z \in \mathbb{N}$, such that $f^{-j_z}B_{r}(z) \subseteq U$ $\mu$-almost surely, i.e. \[\mu(f^{-j_z}B_{r}(z) \bigcap U^c)=0.\]

\end{lemma}

\subsection{Short returns}
Let $z' \in \interior{(U)}$. Take now any fixed positive integer $M>0$, a sufficiently small constant $\epsilon'>0$ such that $ n^{\epsilon'}\ll p$ and $B_{M\cdot r} (z') \subseteq \interior{(U)}$ almost surely, where $n=\lfloor \frac{T}{\mu(B_r(z))}\rfloor, p=\lfloor {\lfloor \frac{T}{\mu(B_r(z))}\rfloor}^{\frac{\text{dim}_H\mu-\epsilon}{\text{dim}_H\mu}}\rfloor$, and the same $\epsilon$ as in the Proposition \ref{rate}. We now consider short returns for the induced map $f^{\overline{R}}:U\to U$, namely
\begin{equation}\label{sumhorshose1}
    \int_{B_{M\cdot r} (z')} 1_{\bigcup_{1\le k \le p}(f^{\overline{R}})^{-k}B_{M\cdot r} (z')}d\mu_{U}=\int_{B_{M\cdot r} (z')} 1_{\bigcup_{1\le k \le N}(f^{\overline{R}})^{-k}B_{ M\cdot r} (z')}d\mu_{U}
\end{equation}

\[+\int_{B_{M\cdot r} (z')} 1_{\bigcup_{N\le k \le n^{\epsilon'}}(f^{\overline{R}})^{-k}B_{ M\cdot r} (z')}d\mu_{U}+\int_{B_{ M\cdot r} (z')} 1_{\bigcup_{n^{\epsilon'}\le k \le p}(f^{\overline{R}})^{-k}B_{ M\cdot r} (z')}d\mu_{U},\]
where $N= \lfloor\frac{-\log 2C }{\log \beta}\rfloor+1$, and $C, \beta$ are the ones from the Definition \ref{mpartition}. 

\begin{lemma}[Short fixed-length returns]\label{fixreturn1}\ \par
For almost every $z' \in \interior{(U)} $ and sufficiently small $r_{N,M, z'}>0$ we have $\forall r < r_{N,M,z'}$ \[\int_{B_{M\cdot r} (z')} 1_{\bigcup_{1\le k \le N}(f^{\overline{R}})^{-k}B_{ M\cdot r} (z')}d\mu_{U}=0.\]

(Actually, a stronger result will be proved, i.e.  $\forall k\in \mathbb{N}$ a map $f^{\overline{R}^k}$ is a local diffeomorphism at almost every point $z' \in \interior{U}$).
\end{lemma}
\begin{proof}
According to the Definition \ref{mpartition}, we have $\mu(\partial U)=0$.  Also, the exponential decay of correlations property holds for $f^{\overline{R}}$.  Let $A_{per}$ be a set of all periodic points for $f^{\overline{R}}$. Then $\mu(A_{per})=0$ and $\mu(\bigcup_{n \in \mathbb{Z}}f^{-n}\partial U)=0$. Choose $z'\notin  \bigcup_{n \in \mathbb{Z}}f^{-n}\partial U \bigcup A_{per}$. The rest of the proof is exactly the same as in the Lemma \ref{fixreturn}, replacing $R$ by $\overline{R}$ and $\Lambda$ by $U$. 

\end{proof}

Before estimating the super-short returns $\int_{B_{M\cdot r} (z')} 1_{\bigcup_{N\le k \le n^{\epsilon'}}(f^{\overline{R}})^{-k}B_{ M\cdot r} (z')}d\mu_{U}$ we will need one more lemma.

\begin{lemma}[Recurrences]\label{recurrence1}\ \par
There exists $r_M>0$, such that $ \forall r < r_M$ and $i\ge N$
\[\mu_U \{z'\in U: d((f^{\overline{R}})^{i}z',z') \le M \cdot r\}\precsim_{\dim \gamma^u}  (M \cdot r)^{\frac{b \cdot \dim \gamma^u}{b+\dim \gamma^u}},\]
where a constant in $\precsim_{\dim \gamma^u}$ depends upon $\dim \gamma^u$, but it does not depend upon $i\ge {N}$ and $\gamma^u \in \Theta$.

\end{lemma}
\begin{proof}
For any $\gamma^u \in \Theta, z_1',z_2' \in \{z'\in U\bigcap \gamma^u: d((f^{\overline{R}})^{-i}z',z')\le M \cdot r\} $ we have 
\[d((f^{\overline{R}})^{-i}z'_1,z_1') \le M \cdot r,\text{ } d((f^{\overline{R}})^{-i}z'_2,z_2') \le M \cdot r.\]

The u-contraction (see the Definition \ref{mpartition}), together with $i \ge N$, give that  $C \cdot \beta^i \le C \cdot \beta^{N}< \frac{1}{2}$ and 
\[d(z'_1,z'_2)\le d(z'_1, (f^{\overline{R}})^{-i}z'_1)+d((f^{\overline{R}})^{-i}z'_1, (f^{\overline{R}})^{-i}z'_2)+d((f^{\overline{R}})^{-i}z'_2, z'_2)\]
\[\le 2M \cdot r+ C \cdot \beta^i \cdot d(z'_1,z'_2) \le 2M \cdot r +\frac{1}{2}  d(z'_1,z'_2).\]

So $d(z'_1,z'_2) \le 4M \cdot r \implies \diam \{z'\in U \bigcap \gamma^u: d((f^{\overline{R}})^{-i}z',z') \le M \cdot r\} \le 4M \cdot r$.

Since each $\gamma^u\in \Theta$ has uniformly bounded sectional curvature, then  $\Leb_{\gamma^u}(B_{4M\cdot r}(z')) \precsim (4M \cdot r)^{\dim \gamma^u}$ for any $z' \in \gamma^u$ and 
\[\Leb_{\gamma^u}\{z'\in U: d((f^{\overline{R}})^{-i}z',z') \le M \cdot r\} \precsim_{\dim \gamma^u} (M \cdot r)^{\dim \gamma^u},\]
where a constant in $\precsim_{\dim \gamma^u}$  depends only on $\dim \gamma^u$.

From the Definition \ref{mpartition} we get 
$r_M:=\frac{1}{M}$ such that $\forall r<r_M$, $\mu_U \{x \in U: |\gamma^u(x)|<(M \cdot r)^{\frac{\dim \gamma^u}{b+\dim \gamma^u}}\}\le C \cdot (M \cdot r)^{\frac{b \cdot \dim \gamma^u}{b+\dim \gamma^u}}$, and for any $y \in\gamma^u \in \Theta$, $\frac{d\mu_{\gamma^u}}{d\Leb_{\gamma^u}}(y)=C^{\pm} \cdot \frac{1}{\Leb_{\gamma^u} (\gamma^u)}$. Hence
\[\mu_U \{z'\in U: d((f^{\overline{R}})^{-i}z',z') \le M \cdot r\}=\int \mu_{\gamma^u(x)}\{z'\in U: d((f^{\overline{R}})^{-i}z',z') \le M \cdot r\} d\mu_U(x)\]
\[=\int_{|\gamma^u(x)|\le (M \cdot r)^{\frac{\dim \gamma^u}{b+\dim \gamma^u}}} \mu_{\gamma^u(x)}\{z'\in U: d((f^{\overline{R}})^{-i}z',z') \le M \cdot r\} d\mu_U(x)\]
\[+\int_{|\gamma^u(x)|\ge (M \cdot r)^{\frac{\dim \gamma^u}{b+\dim \gamma^u}}} \mu_{\gamma^u(x)}\{z'\in U: d((f^{\overline{R}})^{-i}z',z') \le M \cdot r\} d\mu_U(x) \]
\[\le C \cdot (M \cdot r)^{\frac{b \cdot \dim \gamma^u}{b+\dim \gamma^u}} +\int_{|\gamma^u(x)|\ge (M \cdot r)^{\frac{\dim \gamma^u}{b+\dim \gamma^u}}} \mu_{\gamma^u(x)}\{z'\in U: d((f^{\overline{R}})^{-i}z',z') \le M \cdot r\} d\mu_U(x) \]
\[\precsim_{\dim \gamma^u} (M \cdot r)^{\frac{b \cdot \dim \gamma^u}{b+\dim \gamma^u}}+\int_{|\gamma^u(x)|\ge (M \cdot r)^{\frac{\dim \gamma^u}{b+\dim \gamma^u}}} \frac{\Leb_{\gamma^u(x)}\{z'\in U: d((f^{\overline{R}})^{-i}z',z') \le M \cdot r\}}{\Leb_{\gamma^u(x)} (\gamma^u(x))} d\mu_U(x)\]
\[\precsim_{\dim \gamma^u} (M \cdot r)^{\frac{b \cdot \dim \gamma^u}{b+\dim \gamma^u}}+ (M\cdot r)^{\dim \gamma^u-\dim \gamma^u\cdot {\frac{\dim \gamma^u}{b+\dim \gamma^u}}} \precsim (M \cdot r)^{\frac{b \cdot \dim \gamma^u}{b+\dim \gamma^u}},\]
where in the last two inequalities we used that the sectional curvature of $\gamma^u(x) \in \Theta$ is uniformly bounded, and $\dim \gamma^u-\dim \gamma^u\cdot {\frac{\dim \gamma^u}{b+\dim \gamma^u}}=\frac{b \cdot \dim \gamma^u}{b+\dim \gamma^u}$.

Finally, $(f^{\overline{R}})_{*} \mu_U=\mu_U$ implies that 
\[\mu_U \{z'\in U: d((f^{\overline{R}})^{i}z',z') \le M \cdot r\}\precsim_{\dim \gamma^u} (M \cdot r)^{\frac{b \cdot \dim \gamma^u}{b+\dim \gamma^u}}.\]
\end{proof}
\begin{lemma}[Super-short returns]\label{veryshortreturn1}\ \par
For almost every $z' \in \interior{(U)}, z \in \mathcal{M}$ there exist $r_{M,z'}>0, r_z>0$, such that $\forall r< \min \{r_{z',M}, r_z\}$ 
\[\int_{B_{M\cdot r} (z')} 1_{\bigcup_{N\le k \le n^{\epsilon'}}(f^{\overline{R}})^{-k}B_{ M\cdot r} (z')}d\mu_{U}\precsim_{T,b} (M \cdot r)^{\dim_H\mu -\epsilon} \cdot M^{\frac{\min\{{\frac{b \cdot \dim \gamma^u}{b+\dim \gamma^u}}, \dim_H \mu\}}{6}}\cdot r^{\frac{\min\{{\frac{b \cdot \dim \gamma^u}{b+\dim \gamma^u}}, \dim_H \mu\}}{12}},\]
where  $\epsilon>0$ is the same as in the Proposition \ref{rate} and $\epsilon'<\frac{\min\{\dim_H \mu, {\frac{b \cdot \dim \gamma^u}{b+\dim \gamma^u}}\}}{12 \dim_H \mu+12 \epsilon}$.
\end{lemma}
\begin{proof}
In the Lemma \ref{recurrence1} we already proved that
\[\mu_U \{z'\in U: d((f^{\overline{R}})^{i}z',z') \le M \cdot r\}\precsim_{\dim \gamma^u} (M \cdot r)^{\frac{b \cdot \dim \gamma^u}{b+\dim \gamma^u}}.\]
The rest of the proof is exactly the same as in the Proposition \ref{veryshortreturn}, where one should replace $\dim \gamma^u, \Lambda, R$ by $\frac{b \cdot \dim \gamma^u}{b+\dim \gamma^u}, U, \overline{R}$.
\end{proof}

\begin{lemma}[Not super-short, but short returns]\label{notveryshort1}\ \par
Let $n= \lfloor \frac{T}{\mu(B_r(z))}\rfloor,p=\lfloor {\lfloor \frac{T}{\mu(B_r(z))}\rfloor}^{\frac{\text{dim}_H\mu-\epsilon}{\text{dim}_H\mu}}\rfloor$ and $n^{\epsilon'}\ll p$. Then for almost all $z'\in \interior{(U)}$ and $z\in \mathcal{M}$  there is $r_{z,z',M}>0$, such that $\forall r< r_{z,z',M}$
\[\int_{B_{ M\cdot r} (z')} 1_{\bigcup_{n^{\epsilon'}\le k \le p}(f^R)^{-k}B_{ M\cdot r} (z')}d\mu_{U}\]
\[\precsim_{T,\epsilon} (\frac{1}{r^{\dim_H\mu+\epsilon}})^{\frac{\dim_H\mu-\epsilon}{\dim_H\mu}} \cdot [(M\cdot r)^{2\dim_H\mu-2\epsilon^3}+\mu_U(B_{M \cdot r+(\sqrt[4]{\beta})^{T^{\epsilon'}\cdot{r^{-(\dim_H\mu-\epsilon) \epsilon'}}}}(z')\setminus B_{M \cdot r}(z'))],\]
where $\beta\in (0,1) $ is the same as in the Definition \ref{mpartition}.
\end{lemma}
\begin{proof}
The approach to proving the required estimate is quite standard. It uses Lipschitz functions to approximate $B_{M \cdot r}(z')$. However, we will write it down for completeness.

Let $L\in \Lip(U)$, such that $L=1$ on ${B_{M \cdot r}(z')}$, $L=0 $ on $ {B^c_{M \cdot r+(\sqrt[4]{\beta})^{T^{\epsilon'}\cdot{r^{-(\dim_H\mu-\epsilon) \epsilon'}}}}(z')} $, and $L$ is linear on $ B_{M \cdot r+(\sqrt[4]{\beta})^{T^{\epsilon'}\cdot{r^{-(\dim_H\mu-\epsilon) \epsilon'}}}}(z') \setminus B_{M \cdot r}(z')$, where $\epsilon$ and $\alpha$ are the same as in the Proposition \ref{rate}. Hence $L \in \Lip(U)$ with a Lipschitz constant $(\sqrt[4]{\beta})^{-T^{\epsilon'}\cdot{r^{-(\dim_H\mu-\epsilon) \epsilon'}}}$. Therefore for any $k \in [n^{\epsilon'}, p]$
\[\int_{B_{ M\cdot r} (z')} 1_{B_{ M\cdot r} (z')} \circ (f^R)^{k}d\mu_{U} \le \int L \cdot  L \circ (f^R)^{k}d\mu_{U}+2 \mu_U( B_{M \cdot r+(\sqrt[4]{\beta})^{T^{\epsilon'}\cdot{r^{-(\dim_H\mu-\epsilon) \epsilon'}}}}(z') \setminus B_{M \cdot r}(z')) \]
\[\le |\int L \cdot  L \circ (f^R)^{k}d\mu_{U}-\int L d\mu_U \cdot \int L \circ (f^R)^{k}d\mu_{U}|+\int L d\mu_U \cdot \int L \circ (f^R)^{k}d\mu_{U}\]
\[+2 \mu_U(B_{M \cdot r+(\sqrt[4]{\beta})^{T^{\epsilon'}\cdot{r^{-(\dim_H\mu-\epsilon) \epsilon'}}}}(z')\setminus B_{M \cdot r}(z')).\]

Now, by making use of the exponential decay of correlation in the Definition \ref{mpartition}, we can continue the estimate as
\[\precsim \Lip(L)^2 \cdot \beta^{n^{\epsilon'}}+\mu_U(B_{M \cdot r+(\sqrt[4]{\beta})^{T^{\epsilon'}\cdot{r^{-(\dim_H\mu-\epsilon) \epsilon'}}}}(z'))^2\]
\[+2 \mu_U(B_{M \cdot r+(\sqrt[4]{\beta})^{T^{\epsilon'}\cdot{r^{-(\dim_H\mu-\epsilon) \epsilon'}}}}(z')\setminus B_{M \cdot r}(z')) \le (\sqrt[4]{\beta})^{-T^{\epsilon'}\cdot{r^{-(\dim_H\mu-\epsilon) \epsilon'}}}\cdot \beta^{n^{\epsilon'}}\]
\[+\mu_U(B_{M \cdot r+(\sqrt[4]{\beta})^{T^{\epsilon'}\cdot{r^{-(\dim_H\mu-\epsilon) \epsilon'}}}}(z'))^2+2 \mu_U(B_{M \cdot r+(\sqrt[4]{\beta})^{T^{\epsilon'}\cdot{r^{-(\dim_H\mu-\epsilon) \epsilon'}}}}(z')\setminus B_{M \cdot r}(z')).\]
Hence
\[\int_{B_{ M \cdot r} (z')} 1_{\bigcup_{n^{\epsilon'}\le k \le p}(f^R)^{-k}B_{ M\cdot r} (z')}d\mu_{U} \precsim p \cdot (\sqrt{\beta})^{-T^{\epsilon'}\cdot{r^{-(\dim_H\mu-\epsilon) \epsilon'}}}\cdot \beta^{n^{\epsilon'}}\]
\begin{equation}\label{9}
    +p \cdot \mu_U(B_{M \cdot r+(\sqrt[4]{\beta})^{T^{\epsilon'}\cdot{r^{-(\dim_H\mu-\epsilon) \epsilon'}}}}(z'))^2+2 p \cdot \mu_U(B_{M \cdot r+(\sqrt[4]{\beta})^{T^{\epsilon'}\cdot{r^{-(\dim_H\mu-\epsilon) \epsilon'}}}}(z')\setminus B_{M \cdot r}(z')).
\end{equation}

Choose now from the Assumption \ref{geoassumption} the same $\epsilon>0$ as in the Proposition \ref{rate}. Then,  for almost all $z\in \mathcal{M}$ and $ z'\in \interior{(U)}$,  there exists such $r_{z,z',M}>0$, that $\forall r< r_{z,z',M}$ \[B_{M \cdot r+(\sqrt[4]{\beta})^{T^{\epsilon'}\cdot{r^{-(\dim_H\mu-\epsilon) \epsilon'}}}}(z') \subseteq U, \frac{T}{r^{\dim_H\mu-\epsilon}}\precsim n \precsim \frac{T}{r^{\dim_H\mu+\epsilon}},\]
\[r^{\dim_H\mu+\epsilon}\le \mu(B_r(z)) \le r^{\dim_H\mu-\epsilon}, \frac{(M \cdot r)^{\dim_H\mu+\epsilon^3}}{\mu(U)}\le  \mu_{U}(B_{M \cdot r}(z')) \le \frac{(M\cdot r)^{\dim_H\mu-\epsilon^3}}{\mu(U) },\]
\[ (\frac{T}{r^{\dim_H\mu-\epsilon}})^{\frac{\dim_H\mu-\epsilon}{\dim_H\mu}} \precsim p\precsim (\frac{T}{r^{\dim_H\mu+\epsilon}})^{\frac{\dim_H\mu-\epsilon}{\dim_H\mu}},(\frac{T}{r^{\dim_H\mu-\epsilon}})^{\epsilon'}\precsim n^{\epsilon'} \precsim (\frac{T}{r^{\dim_H\mu+\epsilon}})^{\epsilon'}.\]

Therefore, the estimates from $(\ref{9})$ can be continued as
\[\precsim (\frac{T}{r^{\dim_H\mu+\epsilon}})^{\frac{\dim_H\mu-\epsilon}{\dim_H\mu}} \cdot (\sqrt{\beta})^{-T^{\epsilon'}\cdot{r^{-(\dim_H\mu-\epsilon) \epsilon'}}} \cdot \beta^{T^{\epsilon'}\cdot{r^{-(\dim_H\mu-\epsilon) \epsilon'}}}+(\frac{T}{r^{\dim_H\mu+\epsilon}})^{\frac{\dim_H\mu-\epsilon}{\dim_H\mu}}\]
\[\times [\mu_U(B_{M \cdot r+(\sqrt[4]{\beta})^{T^{\epsilon'}\cdot{r^{-(\dim_H\mu-\epsilon) \epsilon'}}}}(z'))^2+\mu_U(B_{M \cdot r+(\sqrt[4]{\beta})^{T^{\epsilon'}\cdot{r^{-(\dim_H\mu-\epsilon) \epsilon'}}}}(z')\setminus B_{M \cdot r}(z'))]\]
\[\precsim_{T,\epsilon} (\frac{1}{r^{\dim_H\mu+\epsilon}})^{\frac{\dim_H\mu-\epsilon}{\dim_H\mu}} \cdot [(M\cdot r)^{2\dim_H\mu-2\epsilon^3}+\mu_U(B_{M \cdot r+(\sqrt[4]{\beta})^{T^{\epsilon'}\cdot{r^{-(\dim_H\mu-\epsilon) \epsilon'}}}}(z')\setminus B_{M \cdot r}(z'))].\]

The last inequality holds because $\beta^{r^{-c}} \precsim r^{c'}$ for any $c',c>0$.

\end{proof}

Combining now the Lemmas \ref{fixreturn1}, \ref{veryshortreturn1}, \ref{notveryshort1}, and arguing as in the Proposition \ref{shortreturnrate}, we can formulate the following summary of the obtained results.

\begin{proposition}[Short returns rates]\label{shortreturnrate1}\ \par
Let $\epsilon, \epsilon'>0$ satisfy the relations $\alpha \cdot (\dim_H\mu-\epsilon)>\frac{\dim_H\mu}{\dim_H\mu-\epsilon}>1$, $p=\lfloor {\lfloor \frac{T}{\mu(B_r(z))}\rfloor}^{\frac{\text{dim}_H\mu-\epsilon}{\text{dim}_H\mu}}\rfloor$, $\epsilon'<\min\{\frac{\min\{{\frac{b \cdot \dim \gamma^u}{b+\dim \gamma^u}}, \dim_H \mu\}}{12\dim_H \mu+12\epsilon} , \frac{\dim_H\mu-\epsilon}{\dim_H\mu}\}$. Then for almost all $z\in \mathcal{M}$, $z'\in \interior{(U)}$ and for each integer $M>0$ there exists a small enough $r_{z,z',M}>0$, such that $\forall r< r_{z,z',M}$
\[\int_{B_{M\cdot r} (z')} 1_{\bigcup_{1\le k \le p}(f^{\overline{R}})^{-k}B_{M\cdot r} (z')}d\mu_{U}\precsim_{\epsilon, T,b} (M \cdot r)^{\dim_H\mu -\epsilon}\cdot M^{\frac{\min\{{\frac{b \cdot \dim \gamma^u}{b+\dim \gamma^u}}, \dim_H \mu\}}{6}}\cdot r^{\frac{\min\{{\frac{b \cdot \dim \gamma^u}{b+\dim \gamma^u}}, \dim_H \mu\}}{12}}\]
\[+^{-(\dim_H\mu+\epsilon) \cdot \frac{\dim_H\mu-\epsilon}{\dim_H\mu}} \cdot [(M\cdot r)^{2\dim_H\mu-2\epsilon^3}+\mu_U(B_{M \cdot r+(\sqrt[4]{\beta})^{T^{\epsilon'}\cdot{r^{-(\dim_H\mu-\epsilon) \epsilon'}}}}(z')\setminus B_{M \cdot r}(z'))],\]
where $\beta \in (0,1)$ is the same as in the Definition \ref{mpartition}.
\end{proposition}

\subsection{Coronas}
\begin{proposition}[Coronas in $\mathcal{M}$ and $U$]\label{coronaestimate1}\ \par
For almost every $z \in \mathcal{M}$ and $z' \in \interior{(U)}$ there exist $r_z, r_{z',M}>0$, such that $\forall r< \min \{r_z, r_{z',M}\} $ the following estimate holds for the corona in $\mathcal{M}$
\[\mu\{B_{r+C' \cdot r^{\frac{(\dim_H\mu-\epsilon)^2\cdot \alpha }{\dim_H\mu}} }(z)\setminus B_{r-C' \cdot r^{\frac{(\dim_H\mu-\epsilon)^2\cdot \alpha }{\dim_H\mu} }}(z)\} \precsim_{z, \dim \gamma^u}  r^{\frac{(1+\frac{(\dim_H\mu-\epsilon)^2\cdot \alpha }{\dim_H\mu}) \cdot \dim \gamma^u \cdot b}{2(b+\dim \gamma^u)}},\]
and the estimate for the corona in $U$ is
\[\mu_U[B_{M \cdot r+(\sqrt[4]{\beta})^{T^{\epsilon'}\cdot{r^{-(\dim_H\mu-\epsilon) \epsilon'}}}}(z')\setminus B_{M \cdot r}(z')]\precsim_{z, \dim \gamma^u} (M\cdot r)^{\frac{\dim \gamma^u \cdot b }{2(b+\dim \gamma^u)}} \cdot \beta^{ \frac{T^{\epsilon'}\cdot{r^{-(\dim_H\mu-\epsilon) \epsilon'}}\cdot\dim \gamma^u \cdot b}{8(b+\dim \gamma^u)}}.\]
\end{proposition}
\begin{proof}
For corona in $\mathcal{M}$ we have from the Lemma \ref{ballinhorseshoe1} for almost all $z \in  \bigcup_{i \ge 1} \bigcup_{0 \le j < R_i} f^j(\Lambda_i)$ that $z=f^{j_z}(z')$ for some $z' \in \interior{(U)}$. Moreover,  there exists  $r'_z>0$, such that $\forall r< r'_z$ 

\[\mu(f^{-j_z}[B_{r+C' \cdot r^{\frac{(\dim_H\mu-\epsilon)^2\cdot \alpha }{\dim_H\mu} } }(z)\setminus B_{r-C' \cdot r^{\frac{(\dim_H\mu-\epsilon)^2\cdot \alpha }{\dim_H\mu} }}(z)] \bigcap U^c)=0.\]

Since $f_* \mu=\mu$ and $f$ is bijective on $\bigcup_{i \ge 1} \bigcup_{0 \le j < R_i} f^j(\Lambda_i)$, then
\[\mu([B_{r+C' \cdot r^{\frac{(\dim_H\mu-\epsilon)^2\cdot \alpha }{\dim_H\mu} } }(z)\setminus B_{r-C' \cdot r^{\frac{(\dim_H\mu-\epsilon)^2\cdot \alpha }{\dim_H\mu} }}(z)] \bigcap f^{j_z} U^c)=0.\]

Therefore
\[[B_{r+C' \cdot r^{\frac{(\dim_H\mu-\epsilon)^2\cdot \alpha }{\dim_H\mu} } }(z)\setminus B_{r-C' \cdot r^{\frac{(\dim_H\mu-\epsilon)^2\cdot \alpha }{\dim_H\mu} }}(z)] \subseteq f^{j_z}U \text{ almost surely}.\]

The rest of the proof is exactly the same as in the Proposition \ref{coronaestimate}, where one should replace $\gamma^u\in \Gamma^u$ by $\gamma^u \in \Theta$, and then use that the sectional curvatures of all $\gamma^u \in \Theta$ are uniformly bounded. Then there exists $r_z>0$ such that $\forall r< r_z$
\[\Leb_{\gamma^u} \{f^{-j_z}[B_{r+C' \cdot r^{\frac{(\dim_H\mu-\epsilon)^2\cdot \alpha }{\dim_H\mu}} }(z)\setminus B_{r-C' \cdot r^{\frac{(\dim_H\mu-\epsilon)^2\cdot \alpha }{\dim_H\mu} }}(z)]\} \precsim_{z, \dim \gamma^u} r^{\frac{(1+\frac{(\dim_H\mu-\epsilon)^2\cdot \alpha }{\dim_H\mu}) \cdot \dim \gamma^u}{2}}.\]

Similarly to the argument in a proof in the Lemma \ref{recurrence1}, we have
\[\mu\{f^{-j_z}[B_{r+C' \cdot r^{\frac{(\dim_H\mu-\epsilon)^2\cdot \alpha }{\dim_H\mu}} }(z)\setminus B_{r-C' \cdot r^{\frac{(\dim_H\mu-\epsilon)^2\cdot \alpha }{\dim_H\mu} }}(z)]\}\]
\[=\int \mu_{\gamma^u(x)}\{f^{-j_z}[B_{r+C' \cdot r^{\frac{(\dim_H\mu-\epsilon)^2\cdot \alpha }{\dim_H\mu}} }(z)\setminus B_{r-C' \cdot r^{\frac{(\dim_H\mu-\epsilon)^2\cdot \alpha }{\dim_H\mu} }}(z)]\}d\mu(x)\]
\[=\int_{|\gamma^u(x)|\ge r^{\frac{(1+\frac{(\dim_H\mu-\epsilon)^2\cdot \alpha }{\dim_H\mu}) \cdot \dim \gamma^u}{2(b+\dim \gamma^u)}}} \mu_{\gamma^u(x)}\{f^{-j_z}[B_{r+C' \cdot r^{\frac{(\dim_H\mu-\epsilon)^2\cdot \alpha }{\dim_H\mu}} }(z)\setminus B_{r-C' \cdot r^{\frac{(\dim_H\mu-\epsilon)^2\cdot \alpha }{\dim_H\mu} }}(z)]\} d\mu(x) \]
\[+\int_{|\gamma^u(x)|\le r^{\frac{(1+\frac{(\dim_H\mu-\epsilon)^2\cdot \alpha }{\dim_H\mu}) \cdot \dim \gamma^u}{2(b+\dim \gamma^u)}}} \mu_{\gamma^u(x)}\{f^{-j_z}[B_{r+C' \cdot r^{\frac{(\dim_H\mu-\epsilon)^2\cdot \alpha }{\dim_H\mu}} }(z)\setminus B_{r-C' \cdot r^{\frac{(\dim_H\mu-\epsilon)^2\cdot \alpha }{\dim_H\mu} }}(z)]\} d\mu(x) \]
\[\precsim_{z, \dim \gamma^u} r^{\frac{(1+\frac{(\dim_H\mu-\epsilon)^2\cdot \alpha }{\dim_H\mu}) \cdot \dim \gamma^u \cdot b}{2(b+\dim \gamma^u)}}+r^{\frac{(1+\frac{(\dim_H\mu-\epsilon)^2\cdot \alpha }{\dim_H\mu}) \cdot \dim \gamma^u}{2}} \cdot r^{-\frac{(1+\frac{(\dim_H\mu-\epsilon)^2\cdot \alpha }{\dim_H\mu}) \cdot \dim \gamma^u \cdot \dim \gamma^u}{2(b+\dim \gamma^u)}}\]
\[\precsim_{z, \dim \gamma^u} r^{\frac{(1+\frac{(\dim_H\mu-\epsilon)^2\cdot \alpha }{\dim_H\mu}) \cdot \dim \gamma^u \cdot b}{2(b+\dim \gamma^u)}}.\]

Then, using that $f_*\mu=\mu$, we have
\[\mu\{B_{r+C' \cdot r^{\frac{(\dim_H\mu-\epsilon)^2\cdot \alpha }{\dim_H\mu}} }(z)\setminus B_{r-C' \cdot r^{\frac{(\dim_H\mu-\epsilon)^2\cdot \alpha }{\dim_H\mu} }}(z)\} \precsim_{z, \dim \gamma^u}  r^{\frac{(1+\frac{(\dim_H\mu-\epsilon)^2\cdot \alpha }{\dim_H\mu}) \cdot \dim \gamma^u \cdot b}{2(b+\dim \gamma^u)}}.\]

For corona in $U$ there is such $r_{z', M}\in (0, \frac{1}{M})$, that $\forall r<r_{z',M}$
\[B_{M \cdot r+(\sqrt{\beta})^{T^{\epsilon'}\cdot{r^{-(\dim_H\mu-\epsilon) \epsilon'}}}}(z') \subseteq U,\]

and, using the same argument as in the Lemma \ref{recurrence1}, we have

\[\mu_U[B_{M \cdot r+(\sqrt[4]{\beta})^{T^{\epsilon'}\cdot{r^{-(\dim_H\mu-\epsilon) \epsilon'}}}}(z')\setminus B_{M \cdot r}(z')]\]
\[=\int_{|\gamma^u(x)|\ge (M\cdot r)^{\frac{\dim \gamma^u }{2(b+\dim \gamma^u)}} \cdot \beta^{ \frac{T^{\epsilon'}\cdot{r^{-(\dim_H\mu-\epsilon) \epsilon'}}\cdot\dim \gamma^u}{8(b+\dim \gamma^u)}}} \mu_{\gamma^u(x)}\{B_{M \cdot r+(\sqrt[4]{\beta})^{T^{\epsilon'}\cdot{r^{-(\dim_H\mu-\epsilon) \epsilon'}}}}(z')\setminus B_{M \cdot r}(z')\} d\mu(x) \]
\[+\int_{|\gamma^u(x)|\le (M\cdot r)^{\frac{\dim \gamma^u }{2(b+\dim \gamma^u)}} \cdot \beta^{ \frac{T^{\epsilon'}\cdot{r^{-(\dim_H\mu-\epsilon) \epsilon'}}\cdot\dim \gamma^u}{8(b+\dim \gamma^u)}}} \mu_{\gamma^u(x)}\{B_{M \cdot r+(\sqrt[4]{\beta})^{T^{\epsilon'}\cdot{r^{-(\dim_H\mu-\epsilon) \epsilon'}}}}(z')\setminus B_{M \cdot r}(z')\} d\mu(x) \]
\[\precsim_{z, \dim \gamma^u} (M\cdot r)^{\frac{\dim \gamma^u \cdot b }{2(b+\dim \gamma^u)}} \cdot \beta^{ \frac{T^{\epsilon'}\cdot{r^{-(\dim_H\mu-\epsilon) \epsilon'}}\cdot\dim \gamma^u \cdot b}{8(b+\dim \gamma^u)}}. \]

\end{proof}
\subsection{Conclusion of the Proof of the Theorem \ref{thm2}}
We will start with a rough estimate of $\dim_H \mu$.
\subsection*{Proof that $\dim_H\mu \ge \frac{b}{b+\dim \gamma^u} \cdot \dim {\gamma}^u$:}
\begin{lemma}\label{roughdim1}
It follows from the Assumptions \ref{geoassumption} and the Definition \ref{mpartition} that $\dim_H\mu \ge \frac{b}{b+\dim \gamma^u} \cdot \dim {\gamma}^u$.
\end{lemma}
\begin{proof}
Due to the ergodicity of $\mu$, it is enough to consider $z' \in \interior{(U)}$ and a ball $B_r(z')$ in $U$.  By the same trick as in the Proposition \ref{coronaestimate1} we get for any $\gamma^u \in \Theta$ \[\Leb_{\gamma^u}(B_r(z')) \precsim_{z',\dim \gamma^u}r^{\dim \gamma^u}  \text{ and } \mu(B_r(z')) \precsim_{z',\dim \gamma^u}r^{\frac{\dim \gamma^u \cdot b}{b + \dim \gamma^u}}.\]
On the other hand, by the Assumption \ref{geoassumption} for almost all $z' \in \mathcal{M}$ and any $m \in \mathbb{N}$ there is $r_{z',m}>0$ such that $\mu(B_r(z')) \ge r^{\dim_H \mu+ \frac{1}{m}}$ for any $r<r_{z',m}$. This implies that $\dim \gamma^u \cdot \frac{b}{b+\dim \gamma^u} \le \dim_H \mu + \frac{1}{m}, \forall m \in \mathbb{N}$. By taking the limit $m\to \infty$ one gets $\dim_H\mu \ge \dim \gamma^u \cdot \frac{b}{b+\dim \gamma^u}$.
\end{proof}

We will obtain now the estimates of convergence rates in the Theorem \ref{thm2}. According to the Proposition \ref{rate} it is enough to find  convergence rates for short returns and coronas.
A proof will consist of several steps. 
\subsection*{Pull back short returns $\int_{B_r(z)} 1_{\bigcup_{1\le j \le p}f^{-j}B_r(z)}d\mu$}

For almost any $z \in  \bigcup_{i \ge 1} \bigcup_{0 \le j < R_i} f^j(\Lambda_i)$ we have $z=f^{j_z}(z')$ for some $z' \in \interior{(U)}$. By Lemma \ref{comparetopballmetricball}, there exist a topological ball $TB_r(z')=TB_r(f^{-j_z}(z))$ and a constant $C_z>1$ such that $B_{C^{-1}_z\cdot r} (f^{-j_z}z) \subseteq TB_r(f^{-j_z}z) \subseteq B_{C_z\cdot r} (f^{-j_z}z)$.
\begin{lemma}[Pull short returns back to $U$]\label{pullbackhorseshoe1}\ \par
There exists a small enough $r_z>0$, such that $\forall r< r_z$ the following inequality holds
\[ \int_{B_r(z)} 1_{\bigcup_{1\le k \le p}f^{-k}B_r(z)}d\mu \le \mu(U) \cdot \int_{B_{C_z\cdot r} (f^{-j_z}z)} 1_{\bigcup_{1\le k \le p}(f^{\overline{R}})^{-k}B_{C_z\cdot r} (f^{-j_z}z)}d\mu_{U} .\]
\end{lemma}
\begin{proof}
It follows from the invariance of $\mu$ (i.e. $f_*\mu=\mu$) that 
\[ \int_{B_r(z)} 1_{\bigcup_{1\le k \le p}f^{-k}B_r(z)}d\mu=\int_{TB_r(f^{-j_z}z)} 1_{\bigcup_{1\le k \le p}f^{-k}TB_r(f^{-j_z}z)}d\mu.\]

By Lemma \ref{ballinhorseshoe1}, we have
$TB_r(f^{-j_z}z) \subseteq U$ for any $ r< r_z$ if $r_z>0$ is small enough. The rest of the proof is the same as in Lemma \ref{pullbackhorseshoe}, where one should use the first return map $f^{\overline{R}}$.
\end{proof}

\subsection*{Estimate  of $\frac{1}{\mu(B_r(z))}\cdot \int_{B_{C_z\cdot r} (f^{-j_z}z)} 1_{\bigcup_{1\le k \le p}(f^{\overline{R}})^{-k}B_{C_z\cdot r} (f^{-j_z}z)}d\mu_{U}$.}

\begin{lemma}\label{afterpullback1}
 For $\epsilon<\min\{\frac{\min\{\frac{b\cdot \dim \gamma^u}{b+\dim \gamma^u}, \dim_H \mu\}}{24}, \frac{1}{3\dim_H\mu}\}$, satisfying $\alpha \cdot (\dim_H\mu-\epsilon)>\frac{\dim_H\mu}{\dim_H\mu-\epsilon}>1$, and for almost every $z\in \mathcal{M}$, there exists $r_z>0$, such that $\forall r<r_z$
 \[\frac{1}{\mu(B_r(z))}\cdot \int_{B_{C_z\cdot r} (f^{-j_z}z)} 1_{\bigcup_{1\le k \le p}(f^{\overline{R}})^{-k}B_{C_z\cdot r} (f^{-j_z}z)}d\mu_{U}\precsim_{T,z,\epsilon}  r^{\frac{\min\{\frac{b\cdot \dim \gamma^u}{b+\dim \gamma^u}, \dim_H \mu\}}{12}-2\epsilon}+r^{\frac{\epsilon^2}{\dim_H \mu}-3\epsilon^3}.\]

\end{lemma}
\begin{proof}
Note that $C_z < \infty $ for almost all $z\in \mathcal{M}$. Denote $F_1:=\{C_z<\infty\}$, $F_2 \times F_3:=\{(z,z')\in \mathcal{M}\times \interior{(U)}: \text{ the Propositions }\ref{shortreturnrate1}, \ref{coronaestimate1} \text{ hold}\}$. Define a new measure one subset in $\mathcal{M}$ as $F:=F_1 \bigcap F_2 \bigcap \{\bigcup_{j\ge 0}f^j(\interior{(U)}\bigcap F^c_3 )\}^c$. Then for any $z \in F$ we have $ C_z < \infty$ and $(z,f^{-j_z}z)\in F_2 \times F_3$. Let now $M:=\lfloor C_z \rfloor+1, z'=f^{-j_z}z $.  By the Proposition \ref{shortreturnrate1}, there exists $r_{z,z',M}=r_{z,f^{-j_z}z,\lfloor C_z \rfloor+1}>0$, such that $\forall r< r_{z,f^{-j_z}z,\lfloor C_z \rfloor+1}$
\[\int_{B_{C_z\cdot r} (f^{-j_z}z)} 1_{\bigcup_{1\le k \le p}(f^{\overline{R}})^{-k}B_{C_z\cdot r} (f^{-j_z}z)}d\mu_{U}\le \int_{B_{M\cdot r} (z')} 1_{\bigcup_{1\le k \le p}(f^{\overline{R}})^{-k}B_{M\cdot r} (z')}d\mu_{U}\]
\[\precsim_{T,b,\epsilon} (M \cdot r)^{\dim_H\mu -\epsilon} \cdot M^{\frac{\min\{\frac{b\cdot \dim \gamma^u}{b+\dim \gamma^u}, \dim_H \mu\}}{6}}\cdot r^{\frac{\min\{\frac{b\cdot \dim \gamma^u}{b+\dim \gamma^u}, \dim_H \mu\}}{12}}\]
\[+(\frac{1}{r^{\dim_H\mu+\epsilon}})^{\frac{\dim_H\mu-\epsilon}{\dim_H\mu}} \cdot [(M\cdot r)^{2\dim_H\mu-2\epsilon^3}+\mu_U(B_{M \cdot r+(\sqrt[4]{\beta})^{T^{\epsilon'}\cdot{r^{-(\dim_H\mu-\epsilon) \epsilon'}}}}(z')\setminus B_{M \cdot r}(z'))].\]

Next, by the Proposition \ref{coronaestimate1}, the right hand side of the inequality above can be estimated  as
\[\precsim_{T,C_z,z, \epsilon} (M \cdot r)^{\dim_H\mu -\epsilon} \cdot r^{\frac{\min\{\frac{b\cdot \dim \gamma^u}{b+\dim \gamma^u}, \dim_H \mu\}}{12}}+r^{-(\dim_H\mu+\epsilon) \cdot \frac{\dim_H\mu-\epsilon}{\dim_H\mu}} \cdot  r^{2\dim_H\mu-2\epsilon^3}\] \[+r^{-(\dim_H\mu+\epsilon) \cdot \frac{\dim_H\mu-\epsilon}{\dim_H\mu}} \cdot (M\cdot r)^{\frac{\dim \gamma^u \cdot b }{2(b+\dim \gamma^u)}} \cdot \beta^{ \frac{T^{\epsilon'}\cdot{r^{-(\dim_H\mu-\epsilon) \epsilon'}}\cdot\dim \gamma^u \cdot b}{8(b+\dim \gamma^u)}}.\]

Now, by the Assumption \ref{geoassumption} for almost all $z\in \mathcal{M}$ there exists $r_z>0$, such that $\forall r<r_{z}$ \[r^{\dim_H\mu+\epsilon^3}\le \mu(B_r(z)) \le r^{\dim_H\mu-\epsilon^3}.\] 

Then $\forall r<\min \{ r_{z,f^{-j_z}z,\lfloor C_z \rfloor+1}, r_z\}$ \[\frac{1}{\mu(B_r(z))}\cdot \int_{B_{C_z\cdot r} (f^{-j_z}z)} 1_{\bigcup_{1\le k \le p}(f^{\overline{R}})^{-k}B_{C_z\cdot r} (f^{-j_z}z)}d\mu_{\Lambda}\]
\[\precsim_{T,z, \epsilon} \frac{1}{r^{\dim_H\mu+\epsilon^3}}\{(M \cdot r)^{\dim_H\mu -\epsilon} \cdot r^{\frac{\min\{\frac{\dim \gamma^u \cdot b }{b+\dim \gamma^u}, \dim_H \mu\}}{12}}+r^{-(\dim_H\mu+\epsilon) \cdot \frac{\dim_H\mu-\epsilon}{\dim_H\mu}} \cdot  r^{2\dim_H\mu-2\epsilon^3}\]
\[+r^{-(\dim_H\mu+\epsilon) \cdot \frac{\dim_H\mu-\epsilon}{\dim_H\mu}} \cdot (M\cdot r)^{\frac{\dim \gamma^u \cdot b}{2(b+\dim \gamma^u)}} \cdot \beta^{ \frac{T^{\epsilon'}\cdot{r^{-(\dim_H\mu-\epsilon) \epsilon'}}\cdot\dim \gamma^u \cdot b}{8(b+\dim \gamma^u)}}\}\]
\[\precsim_{T,z,\epsilon,C_z} r^{\frac{\min\{\frac{\dim \gamma^u \cdot b }{b+\dim \gamma^u}, \dim_H \mu\}}{12}-2\epsilon}+r^{\frac{\epsilon^2}{\dim_H \mu}-3\epsilon^3}\to 0,\]
The last inequality holds because $\beta_2^{r^{-c}}\ll r^{c'}, \forall c,c'>0.$
\end{proof}

By combining the Lemmas \ref{pullbackhorseshoe1} and \ref{afterpullback1}, we get an
\subsection*{Estimate of short returns rates $\frac{1}{\mu(B_r(z))}\int_{B_r(z)} 1_{\bigcup_{1\le j \le p}f^{-j}B_r(z)}d\mu$.}
\begin{lemma}\label{lastshortreturnrate1}.
For $\epsilon<\min \{\frac{\min\{\frac{\dim \gamma^u \cdot b }{b+\dim \gamma^u}, \dim_H \mu\}}{24}, \frac{1}{3\dim_H \mu}\}$, satisfying $\alpha \cdot (\dim_H\mu-\epsilon)>\frac{\dim_H\mu}{\dim_H\mu-\epsilon}>1$, and for almost every $z\in \mathcal{M}$ \[\frac{1}{\mu(B_r(z))}\int_{B_r(z)} 1_{\bigcup_{1\le j \le p}f^{-j}B_r(z)}d\mu \precsim_{T,z,\epsilon}  r^{\frac{\min\{\frac{\dim \gamma^u \cdot b }{b+\dim \gamma^u}, \dim_H \mu\}}{12}-2\epsilon}+r^{\frac{\epsilon^2}{\dim_H \mu}-3\epsilon^3}\to 0.\]

\end{lemma}

\subsection*{Estimate of coronas rates $\frac{\mu[B_{r+C' \cdot r^{\frac{(\dim_H\mu-\epsilon)^2\cdot \alpha }{\dim_H\mu} } }(z)\setminus B_{r-C' \cdot r^{\frac{(\dim_H\mu-\epsilon)^2\cdot \alpha }{\dim_H\mu} }}(z)]}{\mu(B_r(z))}$ }.
\begin{lemma}\label{lastcoronarate1}.
Let $\alpha>\frac{2}{\dim \gamma^u} \cdot \frac{b+\dim \gamma^u}{b}-\frac{1}{\dim_H\mu}$ and $\epsilon<\min \{\frac{\min\{\frac{\dim \gamma^u \cdot b }{b+\dim \gamma^u}, \dim_H \mu\}}{24},\frac{1}{3\dim_H\mu}\}$ is small enough, so that 
$\alpha >\frac{(\frac{2}{\dim \gamma^u}+\frac{2\epsilon}{\dim \gamma^u \cdot \dim_H \mu})\cdot \frac{b+\dim \gamma^u}{b}-\frac{1}{\dim_H \mu}}{(1-\frac{\epsilon}{\dim_H\mu})^2}$.
Then for almost all $z\in \mathcal{M}$  
\[\frac{\mu[B_{r+C' \cdot r^{\frac{(\dim_H\mu-\epsilon)^2\cdot \alpha }{\dim_H\mu}} }(z)\setminus B_{r-C' \cdot r^{\frac{(\dim_H\mu-\epsilon)^2\cdot \alpha }{\dim_H\mu} }}(z)]} {\mu(B_r(z))}\precsim_{z} r^{\frac{(1+\frac{(\dim_H\mu-\epsilon)^2\cdot \alpha }{\dim_H\mu}) \cdot \frac{\dim \gamma^u \cdot b }{b+\dim \gamma^u}-2\dim_H \mu -2\epsilon}{2}}\to 0,\]
The calculations here are exactly the same as in the Lemma \ref{lastcoronarate}, using the Proposition \ref{coronaestimate1}. Therefore we will not repeat them.
\end{lemma}

\subsection*{Convergence rates $a>0$ in $d_{TV}(N^{r,T,z},P)\precsim_{T,z}r^a$.}
\begin{lemma}\label{lastlastproof}.
Let $\alpha> \frac{2}{\dim \gamma^u}-\frac{1}{\dim_H\mu}$ and $\epsilon<\min \{\frac{\min\{\frac{\dim \gamma^u \cdot b }{b+\dim \gamma^u}, \dim_H \mu\}}{24},\frac{1}{3\dim_H\mu}\}$, so that  \[\alpha >\max \{\frac{\frac{1}{\dim_H\mu}}{(1-\frac{\epsilon}{\dim_H\mu})^2}, \frac{(\frac{2}{\dim \gamma^u}+\frac{2\epsilon}{\dim \gamma^u \cdot \dim_H \mu})\cdot \frac{b+\dim \gamma^u}{b}-\frac{1}{\dim_H \mu}}{(1-\frac{\epsilon}{\dim_H\mu})^2}\}.\]
Now from the Proposition \ref{rate}, Lemma \ref{lastcoronarate1} and Lemma \ref{lastshortreturnrate1} we have

\[a:=\min \{ \frac{(\dim_H\mu-\epsilon)^2(\xi-1)}{\dim_H\mu}, \frac{\epsilon\cdot (\dim_H\mu-\epsilon)}{\dim_H\mu}, \frac{(1+\frac{(\dim_H\mu-\epsilon)^2\cdot \alpha }{\dim_H\mu}) \cdot \frac{b \cdot \dim \gamma^u}{b+\dim \gamma^u}-2\dim_H \mu -2\epsilon}{2},\]
\[\frac{\epsilon^2}{\dim_H \mu}-3\epsilon^3, \frac{\min\{\frac{\dim \gamma^u \cdot b }{b+\dim \gamma^u}, \dim_H \mu\}}{12}-2\epsilon \}=\min \{ \frac{(\dim_H\mu-\epsilon)^2(\xi-1)}{\dim_H\mu}, \frac{\epsilon^2}{\dim_H \mu}-3\epsilon^3,\]
\[\frac{(1+\frac{(\dim_H\mu-\epsilon)^2\cdot \alpha }{\dim_H\mu}) \cdot \frac{\dim \gamma^u \cdot b}{b+\dim \gamma^u}-2\dim_H \mu -2\epsilon}{2}, \frac{\min\{\frac{\dim \gamma^u \cdot b }{b+\dim \gamma^u}, \dim_H \mu\}}{12}-2\epsilon\},\]
where the last equality holds because $\frac{\epsilon^2}{\dim_H \mu}-3\epsilon^3 \le \frac{\epsilon\cdot (\dim_H\mu-\epsilon)}{\dim_H\mu}$.

 If $\frac{d\mu}{d\Leb_{\mathcal{M}}}\in L^{\infty}_{loc}(\mathcal{M})$, then choose  $\epsilon<\min \{\frac{\min\{\frac{\dim \gamma^u \cdot b }{b+\dim \gamma^u}, \dim_H \mu\}}{24},\frac{1}{3\dim_H\mu}\}$, so that $\alpha >\frac{\frac{1}{\dim_H\mu}}{(1-\frac{\epsilon}{\dim_H\mu})^2}$.
Then \[\frac{\mu(B_{r+C' \cdot r^{\frac{(\dim_H\mu-\epsilon)^2\cdot \alpha }{\dim_H\mu}  } }(z)\setminus B_{r-C' \cdot r^{\frac{(\dim_H\mu-\epsilon)^2\cdot \alpha }{\dim_H\mu}}}(z))}{\mu(B_r(z))}\precsim_{z} r^{\frac{(\dim_H \mu-\epsilon)^2\cdot \alpha}{\dim_H\mu}-1}.\] 
 
Then from the Proposition \ref{rate} and Lemma \ref{lastshortreturnrate1} we have $a:=$
\[\min \{\frac{(\dim_H\mu-\epsilon)^2(\xi-1)}{\dim_H\mu}, \frac{\epsilon\cdot (\dim_H\mu-\epsilon)}{\dim_H\mu},\frac{(\dim_H \mu-\epsilon)^2\cdot \alpha}{\dim_H\mu}-1,\frac{\epsilon^2}{\dim_H \mu}-3\epsilon^3, \]
\[\frac{\min\{\frac{\dim \gamma^u \cdot b }{b+\dim \gamma^u}, \dim_H \mu\}}{12}-2\epsilon\}=\min \{\frac{(\dim_H\mu-\epsilon)^2(\xi-1)}{\dim_H\mu},\frac{(\dim_H \mu-\epsilon)^2\cdot \alpha}{\dim_H\mu}-1,\]
\[\frac{\epsilon^2}{\dim_H \mu}-3\epsilon^3, \frac{\min\{\frac{\dim \gamma^u \cdot b }{b+\dim \gamma^u}, \dim_H \mu\}}{12}-2\epsilon\},\] 
 where the last equality holds because  $\frac{\epsilon^2}{\dim_H \mu}-3\epsilon^3 \le \frac{\epsilon\cdot (\dim_H\mu-\epsilon)}{\dim_H\mu}.$
\end{lemma}
This completes a
proof of the Theorem \ref{thm2}.

\section{Applications}\label{app}
 For all classes of dynamical systems, which will be considered in this section, there exist Gibbs-Markov-Young structures (see the Definition \ref{gibbs}) and SRB measures $\mu$. The Assumption \ref{geoassumption}, sometimes except $\alpha \cdot \dim_H\mu>1$, also always holds. Therefore, only the following conditions must be verified:

\begin{itemize}
\item for the Theorem \ref{thm}

\begin{enumerate}
    \item $R$ is the first return time for $\Lambda$ and $f$,
\item there exist constants $\alpha>0$ and $C>0$ such that
\[ \sup_{x,y \in \gamma^s \in \Gamma^s, x',y' \in \gamma^u \in \Gamma^u} \{d(f^nx, f^ny), d(f^{-n}x',f^{-n}y')\} \le \frac{C}{n^{\alpha}},\]
\item $\mu\{\interior{(\Lambda)}\}>0$ and $\mu(\partial \Lambda)=0$,
\item verify, whether or not $\alpha>\frac{2}{\dim \gamma^u}-\frac{1}{\dim_H \mu}> \frac{1}{\dim_H\mu}$, and whether or not $\frac{d\mu}{d\Leb_{\mathcal{M}}}\in L_{loc}^{\infty}(\mathcal{M})$.
\end{enumerate}
\item for the Theorem \ref{thm2}
\begin{enumerate}
    \item find such a reference area $U \subseteq \mathcal{M}$ that its first return map $f^{\overline{R}}$ has an exponential decay of correlations,
    \item find a measurable  partition $\Theta:=\{\gamma^u(x)\}_{x \in U}$ with required properties,
    \item check that the estimate $\mu_U \{x \in U: |\gamma^u(x)|< \epsilon\}\le C \cdot \epsilon^b$ holds,
    \item get an estimate of distortion, i.e. $\frac{d\mu_{\gamma^u(x)}}{d\Leb_{\gamma^u(x)}}(y)=C^{\pm 1} \cdot \frac{d\mu_{\gamma^u(x)}}{d\Leb_{\gamma^u(x)}}(z) \text{ for any }y, z \in \gamma^u(x) \in \Theta$,
    \item show that there exist constants $\alpha>0$ and $C>0$, such that 
    \[\sup_{x,y \in \gamma^s \in \Gamma^s, x',y' \in \gamma^u \in \Gamma^u} \{d(f^nx, f^ny), d(f^{-n}x',f^{-n}y')\} \le \frac{C}{n^{\alpha}},\] 

    \item verify, whether or not $\alpha>\frac{2}{\dim \gamma^u}\cdot \frac{b+\dim \gamma^u}{b}-\frac{1}{\dim_H \mu}> \frac{1}{\dim_H \mu}$, and whether or not $\frac{d\mu}{d\Leb_{\mathcal{M}}}\in L_{loc}^{\infty}(\mathcal{M})$.
    
\end{enumerate}

\end{itemize}
All other required conditions hold for the studied below classes of dynamical systems.

\subsection{Intermittent solenoids}\label{appsolenoid}
Following \cite{pene, Alves} let $\mathcal{M}=S^1 \times \mathbb{D}$, $f_{\gamma}(x,z)=(g_{\gamma}(x), \theta \cdot z+\frac{1}{2}e^{2\pi i x})$, where $g_{\gamma}: S^1 \to S^1$ is a continuous map of degree $d\ge 2$ and $\gamma \in (0, +\infty)$ such that 
\begin{enumerate}
    \item $g_{\gamma}$ is $C^2$ on $S^1 \setminus \{0\}$ and $Dg_{\gamma}>1$ on $S^1 \setminus \{0\}$,
    \item $g_{\gamma}(0)=0, Dg_{\gamma}(0+)=1$ and  $xD^2g_{\gamma}(x) \sim x^{\gamma}$ for sufficiently small positive $x$,
    \item $Dg_{\gamma}(0-)>1$,
    \item $\theta>0$ is so small that $\theta \cdot ||Dg_{\gamma}||_{\infty}< 1-\theta$.
\end{enumerate}

It was proved in \cite{Alves} that the SRB probability measure $\mu$ exists iff $\gamma\in (0,1)$, the attractor is $A:=\bigcap_{i\ge 0}f_{\gamma}^i(\mathcal{M}), \xi=\frac{1}{\gamma}>1, \alpha=1+\frac{1}{\gamma}$, $\Lambda=(I \times \mathbb{D}) \bigcap A$, where $I$ is one of the intervals of hyperbolicity, and  $f_{\gamma}: I \to S^1$ is a $C^2$-diffeomorphism. Then $\partial \Lambda=(\partial I \times \mathbb{D}) \bigcap A$ and $\mu(\partial \Lambda)\precsim \Leb_{S^1}(\partial I)=0$ due to (\ref{allquotientmeasure}) and (\ref{horsesrbleb}). Let  $R$ be the first return time constructed in \cite{Alves}, and $f_{\gamma}^R:\Lambda \to \Lambda$ is the corresponding first return map. Clearly, $\dim \gamma^u=1$.

With all of these, we have that $\alpha=1+\frac{1}{\gamma}>2>\frac{2}{\dim \gamma^u}-\frac{1}{\dim_H \mu}$. The condition  $\mu\{\interior{(\Lambda)}\}>0$ holds also. Hence, by the Theorem \ref{thm} we have the following
\begin{corollary}\label{fplsolenoid}
The functional Poisson limit laws  hold for $f_{\gamma}$ for any $\gamma \in (0,1)$ with convergence rates specified in the Lemma \ref{lastproofs}. 
\end{corollary}

\begin{remark}\ \par
\begin{enumerate}
    \item We could also use here the Theorem \ref{thm2}. Indeed, let $U=I \times \mathbb{D}$, $\overline{R}$ is the first return time to $U$, $\Theta$ is the set of all unstable manifolds in $U$ (observe that their union is, actually, $\Lambda$). The lengths of all $\gamma^u \subseteq \Lambda$ are uniformly bounded from below. Therefore, if $\epsilon$ is small enough, then $\mu_U \{x \in U: |\gamma^u(x)|< \epsilon\}=0 \le \epsilon^b$. Note, that here $b$ is arbitrary large. For each $\gamma^u \in \Theta$, $\mu_{\gamma^u} \approx \Leb_{\gamma^u}$ and $\alpha=1+\frac{1}{\gamma}$. It is well known that correlations for $f^{\overline{R}}: U \to U$ decay exponentially. Therefore the Corollary \ref{fplsolenoid} holds.
    \item In \cite{pene} a ``maximum" metric was chosen, instead of the Riemannian metric. It was proved there that the Poisson limit laws hold for $\gamma \in (0,\frac{\sqrt{2}}{2})$. After checking the details therein, we found that our approach allows to improve the results proved there to $\gamma\in (0,1)$, i.e. by using their metric. Note however, that we consider only a Riemannian metric everywhere in the present paper. Therefore we omit here these calculations.
\end{enumerate}
\end{remark}

\subsection{Axiom A attractors}\label{appA}

It was proved in \cite{collet} that Axiom A attractors $\Sigma \subset \mathcal{M}$ with $\dim \gamma^u=1$ satisfy the Poisson limit laws. Later \cite{pene} it was established that the Poisson limit laws hold for the ergodic dynamics $f: \Sigma \to \Sigma$ if $\dim_H\mu>\dim{\mathcal{M}}-1$. We will show that the conditions on $\dim_H\mu$ and $\dim \gamma^u$ can be dropped. 

\begin{definition}[Axiom A attractors, see \cite{Bowen, sinai, Y}]\ \par

Let $f : \mathcal{M} \to  \mathcal{M}$ be a  $C^2$-diffeomorphism. A compact set $\Sigma \subseteq \mathcal{M}$ is called an Axiom A attractor
if 
\begin{enumerate}
    \item There is a neighborhood $U$ of $\Sigma$, called its basin, such that $f^n(x) \to \Sigma$ for every $x \in U$.
\item The tangent bundle over $\Sigma$ is split into $E^u \oplus E^s$,  where $E^u$ and $E^s$ are $df$-invariant subspaces. 
\item $df|_{E^u}$ is uniformly expanding and $df|_{E^s}$ is uniformly contracting.
\item $f: \Sigma \to \Sigma$ is topologically mixing.
\end{enumerate}
\end{definition}

Before turning to the proofs, we need one lemma from \cite{Bowen}:
\begin{lemma}[Markov partitions, see the chapter 3 of \cite{Bowen}]\ \par The set
$\Sigma$ has a Markov partition $\{\Sigma_1, \Sigma_2, \cdots, \Sigma_m\}$ into the elements with arbitrarily small diameters. Here the sets $\Sigma_i$ are  proper rectangles (i.e., $\Sigma_i=\overline{\interior{(\Sigma_i)}}$ and $\interior{(\Sigma_i)}\bigcap \interior{(\Sigma_j)}=\emptyset$ for $i \neq j$, where the interior and the closure are taken with respect to the topology of $\Sigma$, rather than to the topology of $\mathcal{M}$).
\end{lemma}
We will verify now  conditions imposed in our main theorems.

\begin{enumerate}
    \item Let the horseshoe $\Lambda$ coincides with $\Sigma_1$. Then a  return time for a hyperbolic tower $\Delta$ is actually the first return due to the existence of a Markov partition.
    \item The contraction rates of (un)stable manifolds are exponential, i.e. faster than required $O(\frac{1}{n^{\alpha}})$. 
    \item a constant $\alpha$ can be in this case arbitrary large, namely ($\alpha>\max\{2,\frac{1}{\dim_H \mu}\}$). Then $\alpha> 2> \frac{2}{\dim \gamma^u}-\frac{1}{\dim_H \mu}>\frac{1}{\dim_H \mu}$.
    \item Finally, from the existence of a finite Markov partition follows that $\mu(\interior{(\Lambda)})>0$. And $\mu(\partial \Lambda)=0$ due to the structure of $\partial \Sigma_1$, according to the Lemma 3.11 of \cite{Bowen}.
    
\end{enumerate}

Therefore the Theorem \ref{thm} holds, and we obtain the following

\begin{corollary}\label{axiomfpl}
The functional Poisson limit laws hold for Axiom A attractors with a convergence rate specified in the Lemma \ref{lastproofs}.
\end{corollary}

\subsection{Dispersing billiards with and without finite horizon}
The existence of the Gibbs-Markov-Young structure for dispersing billiards was established in \cite{chernovhighdim}. Denote by $D$ a billiard table, i.e. a closed region on the Euclidean plane with a piecewise $C^3$-smooth boundary $\partial D$. The phase space of a billiard  $\mathcal {M}$ is $\partial D \times [-\frac{\pi}{2}, \frac{\pi}{2}]$. In dispersing billiards the boundary is convex inwards. These billiards are hyperbolic dynamical systems with singularities, which appear because of the orbits tangent to the boundary and orbits hitting the singularities of the boundary.  For some technical reasons (see \cite{bunimovich2}) it is convenient to introduce some extra (artificial) singularities, and represent the phase space $\mathcal{M}$ as $\bigsqcup_{k \ge k_0} \partial D \times [-\frac{\pi}{2}+\frac{1}{(k+1)^2}, -\frac{\pi}{2}+\frac{1}{k^2}] \bigsqcup \bigsqcup_{-k \ge k_0} \partial D \times [\frac{\pi}{2}-\frac{1}{k^2}, \frac{\pi}{2}-\frac{1}{(k+1)^2}] \bigsqcup \partial D \times [-\frac{\pi}{2}+\frac{1}{k_0^2}, \frac{\pi}{2}-\frac{1}{k_0^2}]$, where $k_0 \gg 1$. Then $\mathcal{M}$ becomes formally closed, non-compact and disconnected. Moreover, the billiard map $f: \partial{\mathcal{M}} \to \partial{\mathcal{M}}$ becomes multi-valued because the phase space $\mathcal{M}$ gets partitioned into infinitely many pieces, and the boundary $\partial \mathcal{M}$ acquires infinitely many new components. As a result, the billiard map  acquires additional singularities. However, this trick allows to get proper estimates of the distortions, probability densities and the Jacobian of the holonomy map due to partition of unstable manifolds into homogeneous ones (see the details in \cite{bunimovich2} or in the chapter 5 of \cite{CMbook}).  Denote by $\mathbb{S}$ the union of all singular manifolds.  

We will verify now for dispersing billiards conditions of our main theorems.
\begin{enumerate}
    \item $\overline{R}=1, U=\mathcal{M}, \mu \ll \Leb_{\mathcal{M}}$. Also, it was proved in \cite{Y, chernovhighdim} that  correlations decay exponentially.
    \item $\alpha>0$ is arbitrarily large, $\dim_H \mu=2, \dim \gamma^u=1$.
     \item Let $\mathcal{Q}^{\mathbb{H}}_n(x)$ be a connected component of $\mathcal{M} \setminus \bigcup^n_{m=0}f^m(\mathbb{S})$ containing a point $x$. The partition $\Theta:=(\bigcap_{n\ge 1}\overline{\mathcal{Q}^{\mathbb{H}}_n(x)})_{x\in \mathcal{M}}$, which consists of maximal homogeneous unstable manifolds, is measurable.
    \item A required distortion's estimate holds for each $\gamma^u(x) \in \Theta$ by the Corollary 5.30 in \cite{CMbook}.
    \item  The Theorem 5.17 of \cite{CMbook} gives the estimate $\mu_U \{x \in U: |\gamma^u(x)|< \epsilon\}\le C \cdot \epsilon$. (Observe that here $b=1$).
\end{enumerate}

Hence the Theorem \ref{thm2} can be applied, and we have

\begin{corollary}\label{fpldispersing}
The functional Poisson limit laws hold for two-dimensional dispersing billiards with or without a finite horizon, and the corresponding convergence rates satisfy the estimates from Lemma \ref{lastlastproof}.
\end{corollary}

\subsection{Billiards with focusing components of the boundary}
In this section we consider two-dimensional hyperbolic billiards, which have convex outwards of the billiard table circular boundary components together with dispersing and neutral (zero curvature) components of the boundary. The main assumption is that the entire circle, which contains any focusing component, belongs to a billiard table $D$. This class of billiards was introduced and studied in \cite{buni74,bunimovich3}. Standard coordinates for the billiard map $f$ are $(r,\phi)$, where $r$ fixes a point on the boundary of a billiard table and $\phi$ is an angle of reflection off the boundary at this point.  To simplify the exposition, we will consider now only the most studied and popular example in this class, called a stadium billiard. (Actually, all the reasoning for a general case is the same \cite{Bun90}).

The boundary of a stadium consists of two semicircles of the same radius  connected by two tangent to them neutral components. The existence of the Gibbs-Markov-Young structure for the stadium  was proved in \cite{hongkun,markarian}. The phase space in this case is $\mathcal{M}:=\partial D \times [-\frac{\pi}{2}, \frac{\pi}{2}]$, where $\partial D$ is the boundary of a stadium.
\begin{enumerate}
    \item Let $U \subseteq \mathcal{M}$ consists of all  points, where the first or last collisions of the billiard map orbits with the semicircles occur. By the first (resp., last) collisions we mean here the first (resp., last) collisions with a circular component of the boundary, which occur after (resp., before) the last (resp., first) collision of the orbit in a series of consecutive collisions with the neutral part of the boundary or with another focusing component. Clearly this set is a disjoint union of two similar hexagons. Hence, it is enough to consider one of them, say the hexagon attached to $\{(r,\phi) \in \mathcal{M}: r=0\}$, (see the Figure 8.10 of \cite{CMbook}). We have $\mu \ll \Leb_{\mathcal{M}}$ and $\mu(\partial U)=0$. Let $\overline{R}$ be the first return time to $U$. Using the Theorems 4 and 5 in \cite{hongkun} we can prove that the first return map $f^{\overline{R}}: U \to U$ has an exponential decay of correlations. Consider the set of singular points $\mathbb{S}$ which  correspond to hitting the four singular points of the boundary, where focusing and neutral components meet and generate jumps of the curvature. Denote $\mathbb{S}_1:=(f^{\overline{R}})^{-1}(\mathbb{S})$.
    \item Let $\mathcal{Q}_n(x)$ be the connected component of $\mathcal{M} \setminus \bigcup^n_{m=0}(f^{\overline{R}})^m(\mathbb{S}\bigcup \mathbb{S}_1)$ containing a point $x$. The partition $\Theta:=(\bigcap_{n\ge 1}\overline{\mathcal{Q}_n(x)})_{x\in \mathcal{M}}$ is measurable.
    \item The required estimate of distortion holds for each $\gamma^u(x) \in \Theta$ by the Corollary 8.53 in \cite{CMbook}.
    \item We will prove now that  $\mu_U \{x \in U: |\gamma^u(x)|< r\}\precsim r^{\frac{1}{2}}$. Let $\mathcal{Q}'_n(x)$ be the connected component of $\mathcal{M} \setminus \bigcup^n_{m=0}(f^{\overline{R}})^m(\mathbb{S})$ which contains a point $x$.  Then some smooth unstable manifolds $\gamma^{u'}(x) \in \Theta':=(\bigcap_{n\ge 1}\overline{\mathcal{Q}'_n(x)})_{x\in \mathcal{M}}$ are cut by the set $\mathbb{S}_1$ into smaller pieces, which belong to $\Theta$. (Observe that some of them could be disjoint with $\mathbb{S}_1$). It follows from the Theorem 8.42 of \cite{CMbook} that $\mu_U \{x \in U: |\gamma^{u'}(x)|< r\}\le C \cdot r$.
    
    The connected components of the set $\mathbb{S}_1$ are of two types:
    \begin{enumerate}
        \item $L_k$ is a straight (increasing in the $(r, \phi)$- coordinates) segment in $U$ with the slope $\frac{1}{k}$, representing $k$ successive reflections at one and the same semicircle (see e.g. the Figure 8.11 in \cite{CMbook}).
        \item $F_m$ is an increasing curve (in $(r,\phi)$-coordinates) in $U$ with slope $\approx 1$ (i.e. it is bounded away from $0$ and $+\infty$), which corresponds to $m$ successive bounces on the flat sides of the boundary (see e.g. the Figure 8.12 of \cite{CMbook}).
    \end{enumerate}
     Moreover, $L_k$ is located at the distance $\approx \frac{1}{k}$ from the set  $\{(r,\phi) \in \mathcal{M}: \phi=\pm \frac{\pi}{2}\}$, and $F_m$ is at the distance $\approx \frac{1}{m}$ from $\{(r,\phi) \in \mathcal{M}: \phi=0\}$.  Let \[V_1:=\{(r,\phi) \in \mathcal{M}: \phi \in [\frac{\pi}{2}-\sqrt{r}, \frac{\pi}{2}] \bigcup [-\sqrt{r}, \sqrt{r}] \bigcup  [-\frac{\pi}{2}, -\frac{\pi}{2}+\sqrt{r}]\},\]
     \[V_2:=\bigcup_{k \le \frac{1}{\sqrt{r}}}B_{r}(L_k) \bigcup \bigcup_{m \le \frac{1}{\sqrt{r}}}B_{r}(F_m),\]
     
     where $B_r(L_k), B_r(F_m)$ are the $r$-neighborhoods of $L_k, F_m$. Then $\mu_U(V_1 \bigcup V_2) \precsim \sqrt{r}+\frac{1}{\sqrt{r}} \cdot r \precsim \sqrt{r}$. A curve $\gamma^u(x)\in \Theta$ is decreasing (in $(r,\phi)$ coordinates) for almost every $x \in U \bigcap (V_1\bigcup V_2)^c$. Also, if its length $|\gamma^u(x)|<r $, then $\gamma^u(x)$ is disjoint with $\mathbb{S}_1$, and thus $\gamma^u(x) \in \Theta'$. Therefore, 
    \[\mu_U \{x \in U: |\gamma^u(x)|< r\} \precsim \mu_U(V_1\bigcup V_2)+\mu_U\{x\in (V_1 \bigcup V_2)^c: |\gamma^u(x)|< r\}\] \[\precsim r^{\frac{1}{2}}+\mu_U\{x\in (V_1\bigcup V_2)^c: |\gamma^{u'}(x)|< r\} \precsim \sqrt{r}.\]
    \item It was proved in \cite{pene} that $\alpha=1$. Also, $\dim_H \mu=2$ and $\dim \gamma^u=1$. 
    
\end{enumerate}

Therefore all conditions of the Theorem \ref{thm2} are satisfied, and we have 

\begin{corollary}\label{fplstaduim}
The functional Poisson limit laws hold for the stadium-type billiards,
and the corresponding convergence rates are provided by Lemma \ref{lastlastproof}.
\end{corollary}
\begin{remark}[A general remark on billiards]\ \par

All the considerations in our paper were traditionally dealing with hitting of small sets (e.g. small balls)
in the phase spaces of hyperbolic (chaotic) dynamical systems. However, in case of billiards, the most interesting and natural questions are about hitting (or escape through) some small sets (particularly ``holes") on the boundary of billiard tables, rather than in the interior of a billiard table. These sets are small in the space (e.g. $r$) coordinate, but they are large (have a ``full" size) along the angle ($\phi$) coordinate. 

It is worthwhile to mention though, that there are some real life situations, when actually escape (radiation, emission) from various physical devices (cavities, lasers, etc) occurs only in some small range of angles (see e.g. \cite{nature,science}). Our results could be directly applied to such cases.
However, when a target set is a strip (or a cylinder) with a finite fixed height in the angle $\phi$-coordinate, the results of the present paper can also be used/adapted by cutting a cylinder into small sets. Then  the obtained estimates are valid for these pieces of a cylinder. Clearly, this approach does not generally work for recurrences, but it could be applied for the first hitting probabilities because an orbit cannot escape through one hole and then again escape through another hole. (By holes we mean here disjoint pieces of a cylinder). Therefore, one can take in such cases a relevant maximum or minimum of the obtained estimates for ``small" sets, i.e. for the pieces of a cylinder in the phase space of a billiard.

It is worthwhile to mention though, that the functional Poisson limit laws for billiards with holes in the boundary (``cylindric" holes in the phase space) is a work in progress. It requires some new arguments and lengthy computations.
\end{remark}
\subsection{H\'enon Attractors}\label{henon}
 
The Poisson limit laws for the H\'enon attractors (obeying hyperbolic Young towers)  have been proved in \cite{collet}. However, the   convergence rate for this class of dynamical systems was obtained  in a weaker form (\ref{weak}). Here we will derive a stronger rate of convergence (\ref{equaintro}) by making use of some tricks from the proofs of the main Theorems \ref{thm} and \ref{thm2}.
\begin{corollary}\label{fplhenon}
The functional Poisson limit laws with convergence rates (\ref{equaintro}) hold for H\'enon attractors (modelled by Young towers).
\end{corollary}
\begin{proof}
For H\'enon attractors $\alpha$ is an arbitrary large number and the decay of correlation is exponential. Hence, according to the Proposition \ref{rate}, we just need to estimate  convergence rates of  \[ \frac{\int_{B_r(z)} 1_{\bigcup_{1\le j \le [\frac{T}{\mu(B_r(z))}]^{\delta}}f^{-j}B_r(z)}d\mu}{\mu(B_r(z))} \text{ and } \frac{\mu(B_{r+C' \cdot r^{\alpha'} }(z)\setminus B_{r-C' \cdot r^{\alpha'}}(z))}{\mu(B_r(z))},\] 
where $\alpha'>0$ is an arbitrary large number and $\delta>0$ is an arbitrary small number. 

It follows from the Proposition 4.1 in \cite{collet} that there are constants $c,d, C>0$ and a set $\mathcal{U}_r $ with $\mu(\mathcal{U}_r) \le C \cdot r^c$, such that $\frac{\int_{B_r(z)} 1_{\bigcup_{1\le j \le q}f^{-j}B_r(z)}d\mu}{\mu(B_r(z))} \precsim_{T} C \cdot (r^d + q \cdot r^c)$ for any $z \notin \mathcal{U}_r, T>0$ and $q \ge 2$.

By the Borel–Cantelli lemma $\mu(\bigcap_{N >0} \bigcup_{n \ge N} \mathcal{U}_{n^{-\frac{2}{c}}})=0$. So for almost every $z \in \mathcal{M}$, $\exists N_z, r_z>0$, such that $\forall n >N_z, z \notin \mathcal{U}_{n^{-\frac{2}{c}}}$ and $\mu(B_r(z)) \in [r^{\dim_H \mu+\delta},  r^{\dim_H \mu-\delta}]$ for any $r< r_z$ (see the Assumption \ref{geoassumption}). Then for $r< \min \{ N_z^{-\frac{2}{c}}, r_z\}$, $\exists n_r\ge N_z$, such that $(n_r+1)^{-\frac{2}{c}} \le r \le n_r^{-\frac{2}{c}}$. Hence, we have
\[\frac{(1+n_r)^{-\frac{2}{c} \cdot (\dim_H \mu +\delta)}}{n_r^{-\frac{2}{c} \cdot (\dim_H \mu-\delta)}} \le \frac{\mu(B_{n_r^{-\frac{2}{c}}}(z))}{\mu(B_r(z))} \le \frac{n_r^{-\frac{2}{c} \cdot (\dim_H \mu-\delta)}}{(1+n_r)^{-\frac{2}{c} \cdot (\dim_H \mu +\delta)}} \]
\[\frac{\int_{B_r(z)} 1_{\bigcup_{1\le j \le [\frac{T}{\mu(B_r(z))}]^{\delta}}f^{-j}B_r(z)}d\mu}{\mu(B_r(z))} \le \frac{\int_{B_{n_r^{-\frac{2}{c}}}(z)} 1_{\bigcup_{1\le j \le [\frac{T}{\mu(B_r(z))}]^{\delta}}f^{-j}B_{n_r^{-\frac{2}{c}}}(z)}d\mu}{\mu(B_{n_r^{-\frac{2}{c}}}(z))} \cdot \frac{\mu(B_{n_r^{-\frac{2}{c}}}(z))}{\mu(B_r(z))}\]
\[\precsim_{T}  [n_r^{-\frac{2d}{c}}+\frac{n_r^{-2}}{\mu(B_r(z))^{\delta}}] \cdot \frac{n_r^{-\frac{2}{c} \cdot (\dim_H \mu-\delta)}}{(1+n_r)^{-\frac{2}{c} \cdot (\dim_H \mu +\delta)}} \precsim [r^d+r^c \cdot r^{-\delta \cdot (\dim_H \mu+\delta)}] \cdot r^{-2\delta} \to 0, \] provided that $\delta=\frac{\min\{\frac{d}{3}, \dim_H\mu, \frac{c}{2\dim_H \mu +3}\}}{2}$. In other words, 
\[\frac{\int_{B_r(z)} 1_{\bigcup_{1\le j \le [\frac{T}{\mu(B_r(z))}]^{\delta}}f^{-j}B_r(z)}d\mu}{\mu(B_r(z))} \precsim_{z, T} r^{\frac{\min\{\frac{d}{3}, \dim_H\mu, \frac{c}{2\dim_H \mu +3}\}}{2}}\]  for almost every $z\in \mathcal{M}$.

It follows from the last paragraph in the proof of the Proposition 4.2 in \cite{collet}, that there are such constants $e, C>0$ that for any $f>0, z \in \mathcal{M}$  
\[\mu(B_{r+C' \cdot r^{\alpha'} }(z)\setminus B_{r-C' \cdot r^{\alpha'}}(z))\le C \cdot (r^{\frac{\alpha'}{2}} \cdot r^{-e}+r^f).\]
Hence, for almost every $z \in \mathcal{M}$ and $\forall r< r_z$ \[\frac{\mu(B_{r+C' \cdot r^{\alpha'} }(z)\setminus B_{r-C' \cdot r^{\alpha'}}(z))}{\mu(B_r(z))}\le C \cdot \frac{(r^{\frac{\alpha'}{2}} \cdot r^{-e}+r^f)}{r^{\dim_H \mu+\delta}} \to 0,\]
if $\alpha' > 2(e+\delta + \dim_H \mu)$ and $f> \dim_H \mu+ \delta$. Thus 
\[\frac{\mu(B_{r+C' \cdot r^{\alpha'} }(z)\setminus B_{r-C' \cdot r^{\alpha'}}(z))}{\mu(B_r(z))}\precsim_{z} \frac{(r^{\frac{\alpha'}{2}} \cdot r^{-e}+r^{2(\dim_H \mu +\delta)})}{r^{\dim_H \mu+\delta}} \to 0\] for almost every $z \in \mathcal{M}$.

Now the required estimates of the convergence rates for  short returns and coronas are obtained, and the Corollary \ref{fplhenon} holds, according to the Proposition \ref{rate}.

\end{proof}

\section*{Acknowledgements}
We are indebted to I. Melbourne for useful comments and for pointing to the paper \cite{melboune}.

%%%%%%%%%%
%
\bibliographystyle{amsalpha}%
\bibliography{bibfile}
%%%%%%%%%%%%%%%%%%%%%%%%%%%%%%%%%%%%%%%%%%%%%%%

\end{document}